\documentclass[12pt,reqno]{amsart}
\usepackage{amsmath,amsfonts,amssymb,amscd,amsthm,amsbsy,bbm, epsf,calc,  pdfsync}
\usepackage{color}
\usepackage{datetime}

\usepackage{thmtools, thm-restate}
\usepackage{thm-patch, thm-kv, thm-autoref}
\usepackage[colorlinks=true, urlcolor=blue, linkcolor=blue, citecolor=black]{hyperref}

\numberwithin{equation}{section}
\newcounter{hours}\newcounter{minutes}

\theoremstyle{plain}
\declaretheorem[title=Theorem, parent=section]{thm}
\declaretheorem[title=Lemma,sibling=thm]{lem}
\declaretheorem[title=Corollary,sibling=thm]{cor}
\declaretheorem[title=Proposition,sibling=thm]{prop}
\declaretheorem[title=Definition,sibling=thm]{DEF}
\declaretheorem[title=Remark,sibling=thm]{rem}

\declaretheorem[title=Example,sibling=thm]{ex}
\declaretheorem[title=Problem,sibling=thm]{prob}

\def\Swiech{{\accent"13S}wie{\hbox{\kern -0.21em\lower 0.79ex\hbox{$\textfont1=\scriptfont1\lhook$}}}ch}
\def\SWIECH{{\accent"13S}WIE{\hbox{\kern -0.26em\lower 0.77ex\hbox{$\textfont1=\scriptfont1\lhook$}}}CH}

\begin{document}

\title{Coupling L\'evy measures and comparison principles for viscosity solutions} 
\author{Nestor Guillen}
\email{nguillen@math.umass.edu}
\address{Department of Mathematics and Statistics, UMass, Amherst, MA 01003}%
\author{Chenchen Mou}
\email{muchenchen@math.ucla.edu}
\address{Department of Mathematics, UCLA, Los Angeles, CA 90095}
\author{Andrzej \Swiech}
\email{swiech@math.gatech.edu}
\address{School of Mathematics, Georgia Institute of Technology, Atlanta, GA 30332}
\date{\today}


\maketitle
\baselineskip=14pt
\pagestyle{plain}              % page nos. at bottom, no headline
\pagestyle{headings}		% running headline and page nos. at top

\markboth{Nestor Guillen, Chenchen Mou and Andrzej \Swiech}{Coupling L\'evy measures and comparison principles for viscosity solutions}
%%%%%%%%%%%%%%%%%%%%%%%%%%%%%%%%%%%%%%%%%%%%%%

\begin{abstract}
We prove new comparison principles for viscosity solutions of non-linear integro-differential equations. The operators to which the method applies include but are not limited to those of L\'evy-It\^o type. The main idea is to use an optimal transport map to couple two different L\'evy measures, and use the resulting coupling in a doubling of variables argument.

\end{abstract}

\section{Introduction}\label{sec:Introduction}

In this paper we study comparison principles for viscosity subsolutions and supersolutions of integro-differential equations of the form 
\begin{align}\label{eqn:intro equation form}
I(u,x)=\sup \limits_{\alpha\in \mathcal A}\inf \limits_{\beta\in \mathcal B} \{-L^{\alpha\beta}(u,x)+c_{\alpha\beta}(x)u(x)+f_{\alpha\beta}(x)\}=0\quad \text{in}\,\,\mathcal{O}, 
\end{align}
where $\mathcal{O}$ is a bounded domain of $\mathbb{R}^d$, $c_{\alpha\beta}(x)\geq \lambda>0$, and
\begin{align}\label{eqn:intro Levy integral part}
  L^{\alpha\beta}(u,x)=\int_{\mathbb R^d}[u(x+z)-u(x)- \chi_{B_1(0)}(z) Du(x)\cdot z]d\mu_x^{\alpha\beta}(z),
\end{align}
where $\mu_x^{\alpha\beta}$ are the respective L\'evy measures. Equations of the form \eqref{eqn:intro equation form} arise in stochastic optimal control and stochastic differential games where the operators are the generators of pure jump processes. In a work by one of the authors and Schwab \cite{GuiSch2016} it is proved that (roughly speaking) that the class of operators given by a min-max as in \eqref{eqn:intro equation form} is the same as the class of operators satisfying the global comparison property. 

Comparison principles for viscosity solutions of such equations are now well understood in two broad cases. 
The first case is when the operators admit a L\'evy-Ito form. This means that all of the measures $\mu_x^{\alpha\beta}$ are push-forward measures of a single reference measure $\mu$, so that $\mu_x^{\alpha\beta}=(T_x^{\alpha\beta})_\# \mu$, where 
$T_x^{\alpha\beta}:U\to\mathbb R^d$ is a family of Borel measurable maps defined on some separable Hilbert space and $\mu$ is a L\'evy measure on $U\setminus\{0\}$ (see \eqref{eq:levyito} and \eqref{eq:levyito1}). First comparison principles were obtained
by Soner in \cite{son,so}. Further results, including  results for equations with second order PDE terms were obtained subsequently, see \cite{ar, bbp, bi, jk1}. 
The second case is that of equations of order less than or equal to $1$. Here we mention the works of Soner \cite{son,so}, and 
{  the papers of Sayah \cite{say,say1}, where comparison principles are proved} for very general operators in the class where the operators $L^{\alpha\beta}$ are all such that the function $|z|$ is uniformly integrable with respect to the measures 
$\mu_x^{\alpha\beta}$. Also Alvarez and Tourin \cite{at} and Alibaud \cite{al} considered various parabolic equations with non-local terms  of order zero, that is with $\mu_x$ of finite mass.

Little is known when the L\'evy measures arising in \eqref{eqn:intro equation form}-\eqref{eqn:intro Levy integral part} are neither integrable with respect to $|z|$ nor of L\'evy-Ito form. Two of the authors proved in \cite{MS} several comparison results for viscosity solutions which have some regularity. {  Chasseigne and Jakobsen proved in \cite{CJ-2017} comparison results for fully nonlinear equations involving quasilinear
nonlocal operators. We also mention continuous dependence estimates for weak entropy solutions of degenerate parabolic equations with nonlinear fractional diffusion proved by Alibaud, Cifani and Jakobsen in \cite{ACJ-2014}.}
Proving comparison in general is an important question as many operators of interest are not covered by the two situations discussed above, such as the Dirichlet-to-Neumann maps for nonlinear elliptic equations or control/game problems where the processes are not classical L\'evy-Ito diffusions.

In this paper we introduce optimal transport techniques in an attempt to understand this question. We obtain a comparison for non-local equations \eqref{eqn:intro equation form}-\eqref{eqn:intro Levy integral part} that cover the previous two instances without  requiring a L\'evy-Ito structure nor a restriction on the order of the operators. The idea is to use an optimal coupling for the L\'evy measures arising in the non-local terms. Then, the continuity of the L\'evy measures with respect to the base point $x$ is estimated with respect to an optimal-transport based metric. 

The condition we impose is Lipschitz continuity with respect to an $L^p$-transport metric. The exponent $p \in [1,2]$ is related to the order of the singularity at $z=0$ for the L\'evy measures. In the case of operators of order smaller than $1$, it is possible to use the metric corresponding to $p=1$ in which case our condition is (essentially) a dual formulation of the condition used by Sayah \cite{say}. Likewise, in the L\'evy-Ito case our condition reduces to the one typically imposed in the literature \cite{bi, jk1}. 

Unfortunately, it is rather difficult to check the Lipschitz regularity of $\mu_x$ with respect to our $L^p$ transport metric when $p>1$ and $\mu_x$ is not in L\'evy-Ito form (this is precisely the case where comparison is still unknown). Such Lipschitz estimates are even non-trivial to check for L\'evy measures of finite mass and fail to hold\footnote{The authors would like to thank Alessio Figalli for helpful comments regarding this question.}. It is our hope that this paper will spur further research that will expand the class of families of measures $\{\mu_x^{\alpha\beta}\}_{x,\alpha,\beta}$ where this new approach can be applied.

\subsection{The basic idea} Let us illustrate the main idea of the paper in a simple situation. Consider the linear equation
\begin{align}\label{eq;intro}
\lambda u(x)-L(u,x)=0\quad \text{in}\,\,\mathcal{O},
\end{align}
where  $\lambda>0$ and $L(u,x)$ is an operator of the form \eqref{eqn:intro Levy integral part} where we make the following simplifying assumption on the L\'evy measures $\mu_x$: $\mu_x$ is a \emph{probability measure} with finite second moments for every $x$, and there is some $C>0$ such that for any $x,y$
\begin{align}\label{eqn:intro quadratic cost Lipschitz}  
  \textnormal{d}_{2}(\mu_x,\mu_y) \leq C|x-y|.
\end{align}
Here $\textnormal{d}_2$ denotes the optimal transport distance with respect to the square distance (the so called Wasserstein distance). 
Suppose that $u$ is a bounded viscosity subsolution of \eqref{eq;intro} and $v$ is a bounded viscosity supersolution of \eqref{eq;intro} such that $u\leq v$ on $\mathbb{R}^d\setminus \mathcal{O}^c$	. We start with the typical comparison proof. We assume that $u\not \le v$. We double the variables and penalize the doubling considering for $\varepsilon>0$ the function
\begin{align*}
u(x)-v(y)-\frac{1}{\varepsilon}|x-y|^2.
\end{align*}
Suppose that for all sufficiently small $\varepsilon$ the global maximum is attained at $(x_\varepsilon,y_\varepsilon)$ where $u(x_\varepsilon)-v(y_\varepsilon)\geq \ell>0$ for some $\ell>0$. In such circumstances it is well known that
\begin{align*}
\lim_{\varepsilon\to 0}\frac{1}{\varepsilon}|x_{\varepsilon}-y_{\varepsilon}|^2=0,
\end{align*}
so for small $\varepsilon$ we must have $(x_\varepsilon,y_\varepsilon)\in\mathcal{O}\times\mathcal{O}$. Because of the global maximum, we have
\begin{align*}
& \left(u(x_{\varepsilon}+x)-u(x_\varepsilon)-2\frac{x\cdot(x_\varepsilon-y_\varepsilon)}{\varepsilon}\right) \\
& \quad\quad
-\left(v(y_{\varepsilon}+y)-v(y_\varepsilon)-2\frac{y\cdot(x_\varepsilon-y_\varepsilon)}{\varepsilon}\right)
\leq\frac{1}{\varepsilon}|x-y|^2,\;\;\forall\;x,y.
\end{align*}
For $x$ and $y$ let $\pi_{x,y}$ denote a probability measure on $\mathbb{R}^d\times\mathbb{R}^d$ with marginals $\mu_x$ and $\mu_y$ achieving the optimal (quadratic) transport cost between them. Then, we integrate the
above inequality with respect to the measure $\pi_{x_{\varepsilon},y_{\varepsilon}}$ to obtain
\begin{align*}
  L(u,x_{\varepsilon})-L(v,y_{\varepsilon})\leq \frac{1}{\varepsilon} \int_{\mathbb{R}^d\times\mathbb{R}^d} |x-y|^2\;d\pi_{x_{\varepsilon},y_{\varepsilon}}(x,y).
\end{align*}
Thus by the definition of viscosity solution and the definition of $\pi_{x,y}$ we get
\begin{align*}
  \lambda(u(x_{\varepsilon})-v(y_{\varepsilon}))\leq \frac{1}{\varepsilon}\int_{\mathbb R^d\times\mathbb R^d}|x-y|^2
d\pi_{x_{\varepsilon},y_{\varepsilon}}(x,y) = \frac{1}{\varepsilon} \textnormal{d}_2(\mu_{x_\varepsilon},\mu_{y_\varepsilon})^2,
\end{align*}
where the last equality follows from the optimality of $\pi_{x_\varepsilon,y_\varepsilon}$. Then, using \eqref{eqn:intro quadratic cost Lipschitz}, we obtain
\begin{align*}
  0<\lambda \ell \leq \lambda(u(x_{\varepsilon})-v(y_{\varepsilon})) \leq \frac{C}{\varepsilon}|x_{\varepsilon}-y_{\varepsilon}|^2.
\end{align*}
Since the right hand side goes to zero as $\varepsilon\to 0$ we obtain a contradiction. Thus in this model case the proof of comparison reduces to checking if the measures $\mu_x$ satisfy the Lipschitz condition \eqref{eqn:intro quadratic cost Lipschitz} with respect of the (quadratic) optimal transport distance.

Of course, it is atypical for a L\'evy measure to also be a probability or even a finite measure of constant total mass. To deal with this issue, we will make use of an optimal transport problem featuring ``an infinite mass reservoir'' at $0$ (after all, the mass of the L\'evy measure at $0$ is immaterial). This means in particular that one can consider transport between measures which may have unequal or infinite masses. This problem was studied by Figalli and Gigli
in \cite{FigalliGigli2010}, motivated by questions of gradient flows with Dirichlet boundary conditions, and their work is aptly suited for our purposes. 

\subsection{Outline of the paper} 
The notation and definitions are explained in Section \ref{sec:notanddef}. The transport metric is explained in Section \ref{sec:Levy measures}. Section \ref{sec:ass} contains the assumptions and the statement of the main result.
In Section \ref{sec:Comparison} we prove the main comparison principle using the above technique. We then show how the result covers comparison principles for non-local equations involving non-local terms either of L\'evy form of order $\sigma<1$ (see Example \ref{EX:measures with kernels of order sigma<1}) or of L\'evy-It\^o form (see Example \ref{ex:levyito}). { 
In Section \ref{sec:variants} we discuss 
variants of our approach which we illustrate in Example \ref{ex:fracLapl} related to operators of fractional Laplacian type. We also
discuss in Section \ref{sec:variants} two other examples (Examples \ref{ex:Sayah theorem I} and \ref{ex:Sayah theorem II})
comparing our results to these of \cite{say}.}
Finally in Section \ref{sec:other comparison} we derive various comparison principles for equations which have more regular viscosity solutions. They lead to uniqueness of viscosity solutions for a class of uniformly elliptic non-local equations (see Example \ref{ex:uniform}). 
{ The paper ends with an appendix which follows \cite{FigalliGigli2010} collects the main facts about the optimal transport problem ``with boundary''.}

{ 
\subsection{Acknowledgements}

We would like to thank the reviewers for a careful reading of the manuscript and for several important observations and suggestions. We also thank Alessio Figalli for helpful conversations regarding the validity or not of Lipschitz estimates for the Wasserstein distance with respect to other metrics. }

%%%%%%%%%%%%%%%%%%%%%%%%%%%%%%
%%%%%%%%%%%%%%%%%%%%%%%%%%%%%%
%%%%%%%%%%%%%%%%%%%%%%%%%%%%%%
%%%%%%%%%%%%%%%%%%%%%%%%%%%%%%
%%%%%%%%%%%%%%%%%%%%%%%%%%%%%%
%%%%%%%%%%%%%%%%%%%%%%%%%%%%%%
%%%%%%%%%%%%%%%%%%%%%%%%%%%%%%
\section{Notation and definitions}\label{sec:notanddef}

In the whole paper we will consider equation \eqref{eqn:intro equation form} where the operators $L^{\alpha\beta}$ are assumed to be of the form
\eqref{eqn:intro Levy integral part} and $\{\mu_x^{\alpha\beta}\}_{x,\alpha,\beta}$ is a family of L\'evy measures (see Definition \ref{def:space of Levy measures}).
Denoting 
\begin{equation*}
\delta u(x,z):=u(x+z)-u(x)- \chi_{B_1(0)}(z)Du(x)\cdot z,
\end{equation*}
we will write
\begin{align}\label{eqn_intro}
  L^{\alpha\beta}(u,x) = \int_{\mathbb R^d}\delta u(x,z)d\mu_x^{\alpha\beta}(z).
\end{align}

We will denote by {  $B_r(x)$ the open ball in $\mathbb{R}^d$ centered at $x$ with radius $r>0$, and by $B_r$ the open ball centered at $0$ with radius $r$.} Given an open set $\Omega\subset
\mathbb{R}^d$ and $h>0$ we define
\begin{align*}
  \Omega_{h} = \{ x \in \Omega \;\mid d(x,\partial \Omega) > h\}.
\end{align*}
For a subset $A\subset\mathbb{R}^d$ we denote by $A^c$ its complement, i.e. $A^c=\mathbb{R}^d\setminus A$, 
{ and by $\chi_A$ the characteristic function of $A$}.
 
 For $0<\alpha\le1$ and a domain $\mathcal{O}$ in $\mathbb R^d$, we denote by
$C^{0,\alpha}(\mathcal{O})$ the space of $\alpha$-H\"older
continuous functions in $\mathcal{O}$.

We write $C^k(\mathcal{O}), k=1,2,...,$ for the usual spaces of k-times continuously differentiable functions in $\mathcal{O}$. 
%{ We denote by
%$C^p(\mathcal{O})$, for any non-integer $p>0$, the space of functions in $C^k(\mathcal{O})$ whose partial derivatives are in $C^{0,p-k}
%(\mathcal{O})$ where $k$ is the largest integer smaller that $p$ and by 
%$C^{k,1}(\mathcal{O}), k=1,2,...,$ the space of functions in $C^k(\mathcal{O})$ whose partial derivatives are in $C^{0,1}(\mathcal{O})$.}
The space $C_{b}^k(\mathcal{O})$ (respectively, $C^p_b(\mathcal{O})$) consists of functions in $C^{k}(\mathcal{O})$ 
(respectively, $C^p(\mathcal{O})$) which are bounded. 
We write $\textnormal{BUC}(\mathbb R^d)$ for the set of bounded and uniformly continuous functions in $\mathbb R^d$. For two bounded measures $\mu,\nu$, we will write
$\textnormal{d}_{\textnormal{TV}}(\mu,\nu)$ to denote the total variation of $\mu-\nu$.

Let $\mu$ be a Borel measure in $\mathbb{R}^d\setminus\{0\}$ and $1\leq p\leq 2$. We define
\begin{align}\label{eqn:Levy measure p-mass }
  \mathcal{N}_p(\mu) := \int_{\mathbb{R}^d\setminus \{0\}} \min\{1,|z|^p\}d\mu(z).
\end{align}

\begin{DEF}\label{def:space of Levy measures} 
  Let $1\leq p\leq 2$. We define
  \begin{align*}
    \mathbb{L}_p(\mathbb{R}^d) := \Big\{ \mu \textnormal{ positive Borel measure in } \mathbb{R}^d\setminus \{0\} \;\mid\;   \mathcal{N}_p(\mu) < \infty  \Big\}.
  \end{align*}
  The set of all L\'evy measures is $\mathbb{L}_2(\mathbb{R}^d)$. If $\Omega$ is an open subset of $\mathbb{R}^d$, we will consider the set
  \begin{align*}
    \mathbb{L}_p(\Omega) := \Big\{ \mu \in \mathbb{L}_p(\mathbb{R}^d) \;\mid\; \textnormal{spt}(\mu) \subset \Omega  \Big\}.
  \end{align*}
  Note that
  \begin{align*}
    \mathbb{L}_p(\Omega) \subset \mathbb{L}_q(\Omega), \textnormal{ whenever } p\leq q.
  \end{align*}
\end{DEF}
In other words, measures $\mu$ in $\mathbb{L}_p(\Omega)$ are measures in
$\mathbb{L}_p(\mathbb{R}^d)$ such that $\mu(\mathbb{R}^d\setminus\Omega)=0$. We decompose every measure $\mu \in \mathbb{L}_{p}(\mathbb{R}^d)$ as
\begin{align}\label{eqn:Levy measure singular regular decomposition}
  \mu = \hat \mu + \check \mu,
\end{align}
where $\hat\mu(\cdot):=\mu(\cdot\cap B_1) \in \mathbb{L}_{p}(B_1)$ and $\check \mu(\cdot):=\mu(\cdot\cap
(\mathbb{R}^d\setminus B_1))$. We note that $\check \mu$ is a bounded measure.

Consider the L\'evy operator given by some measure $\mu$, 
\begin{align}
  L(u,x) = \int_{\mathbb{R}^d} [u(x+z)-u(x)-\chi_{B_1(0)}(z) D u(x)\cdot z ]\;d\mu(z).
\end{align}
Let us decompose this operator as the sum of two operators, corresponding to the L\'evy measure decomposition in \eqref{eqn:Levy measure singular regular decomposition},
\begin{align*}
  L(u,x) & = \hat L(u,x) + \check L(u,x),
\end{align*}
where
\begin{align}
  \hat L(u,x) & = \int_{\mathbb{R}^d} [u(x+z)-u(x)-\chi_{B_1(0)}(z) D u(x)\cdot z] \;d\hat \mu(z),\label{eqn:Levy operator decomposition singular part}\\
  \check L(u,x) & = \int_{\mathbb{R}^d} [u(x+z)-u(x)] \;d\check \mu(z). \label{eqn:Levy operator decomposition regular part} 
\end{align}

\begin{DEF}\label{def:Levy measure as a functional}
  Given a L\'evy measure $\mu$ in $\mathbb{R}^d$, we define
  \begin{align*}	
    L_\mu(u,x) := \int_{\mathbb{R}^d} [u(x+z)-u(x)-\chi_{B_1(0)}(z)D u(x)\cdot z ]\;d\mu(z).	
  \end{align*}	
	
\end{DEF}

\begin{DEF}\label{def:pointwise Cp}
  For $p \in (1,2)$, a function $u$ is said to be pointwise-$C^p$ at a point $x_0$ if $u$ is differentiable at $x_0$ and if there exists a constant $C>0$ such that for all $x$ in a neighborhood of $x_0$,
  \begin{equation}\label{eq:pointcp}
    |u(x)-u(x_0)-Du(x_0)\cdot (x-x_0)|\leq C|x-x_0|^p.
  \end{equation}
If $u$ is differentiable at $x_0$ and (\ref{eq:pointcp}) is satisfied with $p=2$ we say that $u$ is pointwise-$C^{1,1}$ at $x_0$.
For $p \in (0,1]$, a function is said to be pointwise-$C^p$ at a point $x_0$ if there is a constant $C>0$ such that for all $x$ in a neighborhood of $x_0$, 
  \begin{align*}
    |u(x)-u(x_0)| \leq C|x-x_0|^p.
  \end{align*}
\end{DEF}

%%%%%%%%%%%%%%%%%%%%%%%%%%%%%%%%%%%%%%%%%%%%%%%%
%%%%%%%%%%%%%%%%%%%%%%%%%%%%%%%%%%%%%%%%%%%%%%%%
%%%%%%%%%%%%%%%%%%%%%%%%%%%%%%%%%%%%%%%%%%%%%%%%
%%%%%%%%%%%%%%%%%%%%%%%%%%%%%%%%%%%%%%%%%%%%%%%%
%%%%%%%%%%%%%%%%%%%%%%%%%%%%%%%%%%%%%%%%%%%%%%%%
%%%%%%%%%%%%%%%%%%%%%%%%%%%%%%%%%%%%%%%%%%%%%%%%
%%%%%%%%%%%%%%%%%%%%%%%%%%%%%%%%%%%%%%%%%%%%%%%%
\section{A transportation metric for L\'evy measures}\label{sec:Levy measures}

We will use a transportation metric on the space of L\'evy measures. This metric takes advantage of an ``infinite reservoir'' of mass which allows one to handle measures which may not have equal (or  finite) total mass. Such a metric was considered by Figalli and Gigli \cite{FigalliGigli2010}, where they studied the basic properties of such a metric, and used it to analyze gradient flows with Dirichlet boundary conditions. Our presentation here generally follows that of \cite{FigalliGigli2010}. This is not the only possible extension of the transport metric to the case of unequal masses, other notions have been considered by Kantorovich and Rubinstein. Another notion of distance for L\'evy measures is considered in \cite{ghkk}.

We consider the following set of measures
\begin{align*}
  \mathcal{M}_p(\mathbb{R}^d) := \big \{ \mu \in \mathbb{L}_p(\mathbb{R}^d)\;\mid\; \int_{\mathbb{R}^d\setminus\{0\}}|z|^pd\mu(z) < \infty \big \}.
\end{align*}
That is, the set of L\'evy measures with finite $p$-moment. Note in particular that { $\mathbb{L}_p(B_1)\subset \mathcal{M}_p(\mathbb{R}^d)$} due to the measures being supported in $B_1$.

First we define the notion of admissible couplings between L\'evy measures (see also Definition \ref{def:Appendix admissible measures} for an analogous definition in a more general setting).
\begin{DEF}\label{def:admissible plans}
  Let $\mu_1,\mu_2 \in\mathcal{M}_{p}(\mathbb{R}^d)$. An admissible transport plan between $\mu_1$ and $\mu_2$ is any positive Borel measure on $\mathbb{R}^d \times \mathbb{R}^d$ such that $\gamma(\{0\}\times \{0\})=0$ and
  \begin{align*}
    \pi^{1}_{\#} \gamma_{\mid_{\mathbb{R}^d\setminus \{0\}} } = \mu_1,\;\;\pi^{2}_{\#} \gamma_{\mid_{\mathbb{R}^d\setminus \{0\}} } = \mu_2,
  \end{align*}
  where for $i=1,2$, $\pi^{i}:\mathbb{R}^d\times \mathbb{R}^d\to \mathbb{R}^d$ is defined by $\pi^{i}(x_1,x_2)=x_i$. The set of admissible transport plans will be denoted by $\textnormal{Adm}(\mu_1,\mu_2)$.
\end{DEF}
In particular, if $\gamma$ is admissible then for a Borel set $A$ compactly supported in $\mathbb{R}^d\setminus \{0\}$, 
\begin{align*} 
  \gamma(A\times \mathbb{R}^d) = \mu_1(A),\;\gamma(\mathbb{R}^d \times A) = \mu_2(A). 
\end{align*}

The key point in Definition \ref{def:admissible plans} which distinguishes it from the notion of optimal transport plans is that the marginals of $\gamma$ only coincide with $\mu_1$ and $\mu_2$ away from the origin. In particular, the marginals of $\gamma$ may assign any amount of mass to the origin.

\begin{DEF} 
Let $1\leq p \leq 2$. For a positive Borel measure $\gamma$ on $\mathbb{R}^d\times \mathbb{R}^d$, we define
\begin{align*}
  J_p(\gamma) := \int_{\mathbb{R}^d \times \mathbb{R}^d}|x-y|^p\;d\gamma(x,y).
\end{align*}

\end{DEF}
In the Appendix we study the problem of minimizing $J_p(\gamma)$ over $\gamma\in\textnormal{Adm}(\mu_1,\mu_2)$ in greater generality. In this section we limit ourselves to stating a few further definitions and a few results needed in latter sections.

{ 
\begin{DEF}\label{def:p-distance Levy measures}
  Let $1\leq p\leq 2$. The $p$-distance between measures $\mu,\nu \in \mathcal{M}_{p}(\mathbb{R}^d)$ is defined by
  \begin{align*}
    \textnormal{d}_{\mathbb{L}_p}(\mu,\nu) := \left ( \inf_{\gamma\in \textnormal{Adm}(\mu,\nu)} J_p(\gamma) \right )^{\frac{1}{p}}.
  \end{align*}
\end{DEF}}
The optimization problem used in the definition of $\textnormal{d}_{\mathbb{L}_p}(\mu_1,\mu_2)$ shares many properties with the usual optimal transportation problem.
\begin{thm} For $\mu_1,\mu_2 \in \mathcal{M}_p(\mathbb{R}^d)$ there is at least one $\gamma \in \textnormal{Adm}(\mu_1,\mu_2)$ that achieves the minimum value of $J_p$.
\end{thm}
\begin{proof}
  The theorem is a special case of Theorem \ref{thm:existence OT with boundary} (see the Appendix).
\end{proof}

The fact that $\textnormal{d}_{\mathbb{L}_2}$ defines a distance was proved in \cite[Theorem 2.2, Proposition 2.7]{FigalliGigli2010}. We will need this result for any $p$. 
\begin{thm}\label{thm:transport distance is a distance}
$\textnormal{d}_{\mathbb{L}_p}$ defines a metric in $\mathcal{M}_p(\mathbb{R}^d)$. 
\end{thm}
\begin{proof}
  The theorem is a special case of Theorem \ref{th:metric}.
\end{proof}

The main tool at our disposal when estimating $\textnormal{d}_{\mathbb{L}_p}$ is the following duality result.
\begin{lem}\label{lem:duality}
  For $1\leq p\leq 2$ and $\mu_1,\mu_2 \in \mathcal{M}_p(\mathbb{R}^d)$, we have
  \begin{align*}
    \textnormal{d}_{\mathbb{L}_p}(\mu_1,\mu_2)^p = \sup \Big \{ \int_{\mathbb{R}^d} \phi(x)\;d\mu_1(x) +  \int_{\mathbb{R}^d}\psi(y) \;d\mu_2(y)  \; \mid \; (\phi,\psi) \in \textnormal{Adm}^p\Big \}.
  \end{align*}	  
  Here, $\textnormal{Adm}^p$ denotes the set
  \begin{align*}
    \textnormal{Adm}^p & := \Big \{ (\phi,\psi) \;\mid\; \phi \in L^1(\mu),\; \psi\in L^1(\nu),\\
	& \quad \quad\quad\quad\quad\quad \phi \textnormal{ and } \psi \textnormal{ are upper semicontinuous},\\ 
	& \quad \quad\quad\quad\quad\quad \; \phi(0)=\psi(0) = 0 \textnormal{ and } \phi(x)+\psi(y) \leq |x-y|^p \;\forall\;x,y \in \mathbb{R}^d \Big \}.	  
  \end{align*}	  
\end{lem}
\begin{proof}
  The lemma is a special case of Lemma \ref{lem:duality general}
\end{proof}

\begin{rem}
In most of the paper we only need to take $\textnormal{d}_{\mathbb{L}_p}(\mu,\nu)$ for $\mu,\nu\in \mathbb{L}_p(B_1)$.
In this case we could equivalently define the distance by considering the transport problem in $\overline B_1\times
\overline B_1$, i.e. taking $\Omega= B_1$ instead of $\Omega=\mathbb{R}^d$ (see the Appendix). We note that
if $\mu,\nu\in \mathbb{L}_p(B_1)$ and $\gamma\in\textnormal{Adm}(\mu,\nu)$ then
$\gamma((B_1\times B_1)^c)=0$.
\end{rem}

The following proposition (proved in the Appendix) will be used in Section \ref{sec:Comparison}.
\begin{prop}\label{prop:bound for distances of restricted measures}
  Let $\psi$ be a Lipschitz continuous function with compact support in $\overline{B}_1 \setminus \{0\}$. If $\mu,\nu \in \mathbb{L}_p(B_1)$ then
 \begin{align*}
    \left | \int_{B_1} \psi \;d\mu - \int_{B_1}\psi\;d\nu \right | & \leq \left(\mu(\textnormal{spt}(\psi))+\nu(\textnormal{spt}(\psi))\right)^{\frac{p-1}{p}} [\psi]_{\textnormal{Lip}}\textnormal{d}_{\mathbb{L}_p}(\mu,\nu),
  \end{align*}
where $[\psi]_{\textnormal{Lip}}$ is the Lipschitz constant of $\psi$.
\end{prop}

%%%%%%%%%%%%%%%%%%%%%%%%%%%%%%
%%%%%%%%%%%%%%%%%%%%%%%%%%%%%%
%%%%%%%%%%%%%%%%%%%%%%%%%%%%%%
%%%%%%%%%%%%%%%%%%%%%%%%%%%%%%
%%%%%%%%%%%%%%%%%%%%%%%%%%%%%%
%%%%%%%%%%%%%%%%%%%%%%%%%%%%%%
%%%%%%%%%%%%%%%%%%%%%%%%%%%%%%
\section{Assumptions and main results}\label{sec:ass}

In this section we make the necessary assumptions about the measures and various functions appearing in the operator $I(u,x)$ in \eqref{eqn:intro equation form}. We recall that throughout the whole paper $\mathcal{O}\subset\mathbb{R}^d$ is a bounded domain. The measures { $\mu_x^{\alpha\beta}\in \mathbb{L}_p(\mathbb{R}^d)$} for all $x\in \mathcal{O},
\alpha\in \mathcal A,\beta\in \mathcal B$ for some index sets $\mathcal A, \mathcal B$. {  Last but not least, we recall that in \eqref{eqn:Levy measure singular regular decomposition} we introduced the decomposition of a measure $\mu$ in terms of measures $\hat \mu$ and $\check \mu$ supported in $B_1(0)$ and in $\mathbb{R}^d\setminus B_1(0)$, respectively. }

\textbf{Assumption A}. There are $p\in [1,2]$ and a constant $C\geq 0$ such that 
\begin{align}\label{eqn:LipL2}
  \textnormal{d}_{\mathbb{L}_p}(\hat \mu_x^{\alpha\beta},\hat \mu_y^{\alpha\beta})\leq C|x-y|,\quad \forall x,y\in\mathcal{O},\,\,\forall\alpha,\beta.
\end{align}

\textbf{Assumption B}. There is a modulus of continuity $\theta$ such that
\begin{align}\label{eqn:Levy measure regular TV continuity}
  \textnormal{d}_{\textnormal{TV}}(\check \mu_x^{\alpha\beta},\check \mu_y^{\alpha\beta})\leq \theta(|x-y|),\quad \forall x,y\in\mathcal{O},\,\,\forall\alpha,\beta.
\end{align}

\textbf{Assumption C}. There are a modulus of continuity $\theta$ and a constant $C\geq 0$ such that
\begin{align}\label{eqn:f term modulus of continuity}
  |f_{\alpha\beta}(x)-f_{\alpha\beta}(y)|\leq \theta(|x-y|), \quad \forall x,y\in\mathcal{O},\,\,\forall\alpha,\beta,
\end{align}
\begin{align}\label{eqn:f sup bound}
  |f_{\alpha\beta}(x)|\leq C, \quad \forall x\in\mathcal{O},\,\,\forall\alpha,\beta.
\end{align}

\textbf{Assumption D}. There are constants $0<\lambda\leq \lambda_1$ such that
\begin{align}\label{eqn:Calphabeta eigenvalue bound}  
 \lambda\leq  \inf \limits_{x\in \mathcal{O}}\inf \limits_{\alpha,\beta}c_{\alpha\beta}(x)\leq
 \sup \limits_{x\in \mathcal{O}}\sup \limits_{\alpha,\beta}c_{\alpha\beta}(x)\leq \lambda_1
\end{align}
and there is a modulus $\theta$ such that 
\begin{align*}
  |c_{\alpha\beta}(x)-c_{\alpha\beta}(y)|\leq \theta(|x-y|)\quad \forall x\in\mathcal{O},\,\,\forall\alpha,\beta.
\end{align*}

\textbf{Assumption E}. Let $p$ be from Assumption A. There exist a modulus of continuity $\theta$ and a constant $\Lambda\geq 0$ such that
\begin{align}\label{eqn:uniform integrability Levy measures}
  \sup \limits_{x\in \mathcal{O}}\sup \limits_{\alpha\beta}\int_{B_r} |z|^p \;d\mu^{\alpha\beta}_{x}(z) \leq \theta(r),
\end{align}
\begin{align}\label{eqn:uniform integrability Levy measures1}
  \sup \limits_{x\in \mathcal{O}}\sup \limits_{\alpha,\beta}\mathcal{N}_p(\mu^{\alpha\beta}_{x})\leq \Lambda.
\end{align}

Assumption B can be weakened, however we want to keep its simpler form to focus on the main difficulty of dealing with the singular part of the L\'evy measures. We leave such generalizations to the interested reader.

\medskip
We recall two definitions of viscosity solutions of \eqref{eqn:intro equation form} which will be used in this paper. To minimize the technicalities we will assume that viscosity sub/supersolutions are in $\textnormal{BUC}(\mathbb R^d)$. The same results could be obtained assuming that they are just bounded and continuous in $\mathbb{R}^d$.

\begin{DEF}\label{def:viscosity solution}
Let $p\in [1,2]$. A function $u\in \textnormal{BUC}(\mathbb R^d)$ is a viscosity subsolution of \eqref{eqn:intro equation form} if whenever $u-\varphi$ has a global maximum over $\mathbb R^d$ at $x\in \mathcal{O}$ for some $\varphi\in C_b^2(\mathbb R^d)$ and $\varphi(x)=u(x)$, then 
$I(\varphi,x)\leq 0$.
A function $u\in \textnormal{BUC}(\mathbb R^d)$ is a viscosity supersolution of \eqref{eqn:intro equation form} if whenever $u-\varphi$ has a 
global minimum over $\mathbb R^d$ at $x\in \mathcal{O}$ for some $\varphi\in C_b^2(\mathbb R^d)$ and $\varphi(x)=u(x)$, then
$I(\varphi,x)\geq 0$.
A function $u$ is a viscosity solution of \eqref{eqn:intro equation form} if it is both a viscosity subsolution and viscosity supersolution of 
\eqref{eqn:intro equation form}.
\end{DEF}

\begin{DEF}\label{def:viscosity solution ver1}
Let $p\in [1,2]$.
A function $u\in \textnormal{BUC}(\mathbb R^d)$ is a viscosity subsolution of \eqref{eqn:intro equation form} if whenever $u-\varphi$ has a global maximum over $\mathbb R^d$ at $x\in\mathcal{O}$ for some $\varphi\in C^2(\mathbb R^d)$, then for every $0<\delta<1$
\begin{eqnarray*}
&&\sup_{\alpha\in \mathcal A}\inf_{\beta\in \mathcal B}\Big\{-\int_{|z|<\delta}\delta\varphi(x,z)d\mu_x^{\alpha\beta}(z)\\
&&\,\,
-\int_{|z|\geq\delta}\left[u(x+z)-u(x)-\chi_{B_1(0)}(z)D\varphi(x)\cdot z\right]d\mu_x^{\alpha\beta}(z)
+c_{\alpha\beta}(x)u(x)+f_{\alpha\beta}(x)\Big\}\leq 0.
\end{eqnarray*}
A function $u\in \textnormal{BUC}(\mathbb R^d)$ is a viscosity supersolution of \eqref{eqn:intro equation form} if whenever $u-\varphi$ has a 
global minimum over $\mathbb R^d$ at $x\in\mathcal{O}$ for some $\varphi\in C^2(\mathbb R^d)$, then for every $0<\delta<1$
\begin{eqnarray*}
&&\sup_{\alpha\in \mathcal A}\inf_{\beta\in \mathcal B}\Big\{-\int_{|z|<\delta}\delta\varphi(x,z)d\mu_x^{\alpha\beta}(z)\\
&&\,\,
-\int_{|z|\geq\delta}\left[u(x+z)-u(x)-\chi_{B_1(0)}(z)D\varphi(x)\cdot z\right]d\mu_x^{\alpha\beta}(z)
+c_{\alpha\beta}(x)u(x)+f_{\alpha\beta}(x)\Big\}\geq 0.
\end{eqnarray*}
A function $u$ is a viscosity solution of \eqref{eqn:intro equation form} if it is both a viscosity subsolution and viscosity supersolution of 
\eqref{eqn:intro equation form}.
\end{DEF}
We remark that, since the L\'evy measures $\mu_x^{\alpha\beta}$ are in ${\mathbb{L}_p}(\mathbb{R}^d)$, we could use test functions in 
$C_b^p(\mathbb R^d)$ and $C^p(\mathbb R^d)$ instead of test functions in $C_b^2(\mathbb R^d)$ and $C^2(\mathbb R^d)$. However it is not clear if such definitions and the standard definitions provided above
are equivalent under general assumptions. It is easy to see however that they are equivalent for the most common measures considered
in Example \ref{EX:measures with kernels of order sigma<1}. 

Below we show that Definitions \ref{def:viscosity solution} and \ref{def:viscosity solution ver1} are equivalent to each other.

\begin{prop}\label{prop:equiv}
Under the assumptions of this paper Definitions \ref{def:viscosity solution} and \ref{def:viscosity solution ver1} are equivalent.
\end{prop}

\begin{proof} We only consider the case of subsolutions. It is obvious that if $u$ is a viscosity subsolution in the sense of
Definition \ref{def:viscosity solution ver1} then it is a viscosity subsolution in the sense of
Definition \ref{def:viscosity solution}. Let now $u$ be a viscosity subsolution in the sense of
Definition \ref{def:viscosity solution}. It is easy to see that without loss of generality all maxima/minima in both definitions can be assumed to be strict.
So let $u-\varphi$ have a strict global maximum over $\mathbb R^d$ at $x\in\mathcal{O}$ for some $\varphi\in C^2(\mathbb R^d)$ and we can obviously require that $\varphi(x)=u(x)$. Let $\varphi_n\in C_b^2(\mathbb R^d)$ be functions such that $u\leq \varphi_n\leq \varphi$ on 
$\mathbb R^d$,
$\varphi_n(x)=u(x),D\varphi_n(x)=D\varphi(x)$ and $\varphi_n\to u$ as $n\to+\infty$ uniformly on $\mathbb R^d$. Then
\begin{eqnarray*}
&&\sup_{\alpha\in \mathcal A}\inf_{\beta\in \mathcal B}\Big\{-\int_{|z|<\delta}\delta\varphi(x,z)d\mu_x^{\alpha\beta}(z)\\
&&\,\,
-\int_{|z|\geq\delta}\left[u(x+z)-u(x)-\chi_{B_1(0)}(z)D\varphi(x)\cdot z\right]d\mu_x^{\alpha\beta}(z)
+c_{\alpha\beta}(x)u(x)+f_{\alpha\beta}(x)\Big\}
\\
\end{eqnarray*}

\begin{eqnarray*}
&&=\lim_{n\to+\infty} \sup_{\alpha\in \mathcal A}\inf_{\beta\in \mathcal B}\Big\{-\int_{|z|<\delta}\delta\varphi(x,z)d\mu_x^{\alpha\beta}(z)\\
&&
-\int_{|z|\geq\delta}\left[\varphi_n(x+z)-\varphi_n(x)-\chi_{B_1(0)}(z)D\varphi_n(x)\cdot z\right]d\mu_x^{\alpha\beta}(z)
+c_{\alpha\beta}(x)\varphi_n(x)+f_{\alpha\beta}(x)\Big\}
\\
&&
\qquad\qquad\qquad\qquad
\leq
I(\varphi_n,x)\leq 0.
\end{eqnarray*}
\end{proof}

The main result of the paper is the following theorem.

\begin{thm}\label{thm:Main Comparison Result}
  Let Assumptions A-E hold for $p\in[1,2]$. Then the comparison principle holds for equation (\ref{eqn:intro equation form}). That is, if $u$ and $v$ are respectively a viscosity subsolution and a viscosity supersolution of (\ref{eqn:intro equation form}) and $u(x)\leq v(x)$ for all $x\not\in \mathcal{O}$, then 
  \begin{align*}
    u(x)\leq v(x)\;\;\forall\;x\in\mathcal{O}.
  \end{align*}
\end{thm}

The following is a special case of { Theorem \ref{thm:Main Comparison Result}}, which we highlight to illustrate its scope {  (see Section \ref{subsec:basic examples} and Section \ref{sec:variants} for further examples)}.

\begin{cor}\label{cor:Comparison Theorem Ops Order Less than 1}
  Let Assumptions C and D be satisfied. Suppose that the measures $\mu_x^{\alpha\beta}(z)$ are of the form
  \begin{align*}
   d\mu_x^{\alpha\beta}(z)= K_{\alpha \beta}(x,z)dz  
  \end{align*}
  and that, for some $\sigma \in (0,1)$, 
  \begin{align*}
    0  \leq K_{\alpha\beta}(x,z)\leq K(z):= \Lambda_1 |z|^{-(d+\sigma)},\\
    |K_{\alpha\beta}(x,z)-K_{\alpha\beta}(y,z)| \leq C|x-y|K(z).	
  \end{align*}
  Then, in this case Assumption B holds and Assumptions A and E hold with $p=1$. In particular, the comparison principle holds for equation
  \eqref{eqn:intro equation form} in this case.
\end{cor}

{ 

Theorem \ref{thm:Main Comparison Result} and Corollary \ref{cor:Comparison Theorem Ops Order Less than 1} will be proved in the next section. We also note that Theorem \ref{thm:Main Comparison Result}} { essentially} { covers several of the results in \cite{say}, where only operators of order less than or equal to one are considered.} {  However it cannot be applied directly to the equations in \cite{say} since the operators
considered there had a slightly different form.} {  This is discussed in greater detail in Examples \ref{ex:Sayah theorem I} and \ref{ex:Sayah theorem II}.}

%%%%%%%%%%%%%%%%%%%%%%%%%%%%%%
%%%%%%%%%%%%%%%%%%%%%%%%%%%%%%
%%%%%%%%%%%%%%%%%%%%%%%%%%%%%%
%%%%%%%%%%%%%%%%%%%%%%%%%%%%%%
%%%%%%%%%%%%%%%%%%%%%%%%%%%%%%
%%%%%%%%%%%%%%%%%%%%%%%%%%%%%%
%%%%%%%%%%%%%%%%%%%%%%%%%%%%%%
\section{Comparison Principle}\label{sec:Comparison}

In this section we prove Theorem \ref{thm:Main Comparison Result}. A well known property of sup/inf- convolutions is that they produce approximations of viscosity sub- and supersolutions which enjoy one-sided regularity (semi-convexity and semi-concavity), which makes it easier - under the right circumstances - to evaluate the operator $I(\cdot,x)$ in the classical sense.

\begin{rem}\label{rem:comparison sec no need for sub/inf convolutions}
  An approach to Theorem \ref{thm:Main Comparison Result} that does not rely on such approximations can be found in Section \ref{sec:other comparison}, where we prove a comparison result (Theorem \ref{thm:comparison principle with C1 subsolution}) under a different set of assumptions that are not amicable to such approximations. A posteriori, it became clear that the approach in Section \ref{sec:other comparison} leads to a 
  simpler proof of Theorem \ref{thm:Main Comparison Result}, 
  however we have decided to keep both approaches as the tools developed in this section are of interest in many other situations. See Remark \ref{rem:distance to mu_r follow up} for further comments.
\end{rem}

\begin{DEF}\label{def:sup/inf-conv}
Given $u,v\in \textnormal{BUC}(\mathbb{R}^d)$ and $0<\delta<1$ we define the sup-convolution $u^\delta$ of $u$ and the inf-convolution $v_\delta$ of $v$ by
  \begin{align*}
    u^\delta(x) & = \sup \limits_{y \in \mathbb R^d} \left \{ u(y)-\frac{1}{\delta}|x-y|^2 \right \} ,\\
    v_\delta(x) & = \inf \limits_{y \in \mathbb R^d} \left \{ v(y)+\frac{1}{\delta}|x-y|^2 \right \}. 
  \end{align*}
\end{DEF}
For the reader's convenience, we review some well known properties of the sup/inf-convolutions in the following proposition. 
\begin{prop}\label{prop:sup/inf conv basic properties}
 The sup-convolutions and the inf-convolutions have the following properties.
  \begin{enumerate}
    \item If $\delta_1\leq \delta_2$ then $u^{\delta_1}\leq u^{\delta_2}$ and $v_{\delta_1} \geq v_{\delta_2}$. Moreover 
    $\|u^\delta\|_\infty\leq \|u\|_\infty,\|v_\delta\|_\infty\leq \|v\|_\infty$.
     \item $u^\delta(x) \geq u(x)$ and $v_\delta(x) \leq v(x)$ for all $x\in \mathbb R^d$.	
    \item $u^{\delta}\to u$ and $v_{\delta} \to u$ uniformly on $\mathbb R^d$ as $\delta\to 0$.	
    \item The function $u^\delta$ is semi-convex and for any $x_0\in \mathbb R^d$ , $u^\delta$ is touched from 
    below at $x_0$ by a function of the form
    \begin{align*}
      u(x_0^*)-\tfrac{1}{\delta}|x-x_0^*|^2,\;\;\textnormal{ for some } x_0^* \in \mathbb R^d.
    \end{align*}
    The function $v_\delta$ is semi-concave and for any $x_0\in \mathbb R^d$ , $v_\delta$ is touched from above at $x_0$ by a function of the form
    \begin{align*}
      v(x_0^*)+\tfrac{1}{\delta}|x-x_0^*|^2,\;\;\textnormal{ for some } x_0^* \in \mathbb R^d.
    \end{align*}			
    \item Let $\omega$ be a modulus of continuity of $u$. For any $x_0\in \mathbb R^d$ and $x_0^* \in \mathbb R^d$ such that $u^\delta(x_0) = u(x_0^*)-\tfrac{1}{\delta}|x_0-x_0^*|^2$, we have
    \[
      |x_0-x_0^*| \leq ( 2\delta \|u\|_\infty)^{1/2}.
      \]
      and
	\begin{align*}
          \tfrac{1}{\delta}|x_0-x_0^*|^2 & \leq \omega \big ( ( 2\delta \|u\|_\infty)^{1/2} \big ) .
        \end{align*}
        The analogous property holds for $v_\delta$.		
  \item Let $\Omega\subset \mathbb R^d$ and let $h>0$. If $\omega$ is a modulus of continuity for $u$ in $\Omega$ then for 
  sufficiently small $\delta$, $\omega_1(s)=\max(\omega(s),\tfrac{2}{h}\|u\|_\infty s)$ is a modulus
  of continuity for $u^\delta$ in $\Omega_{2h}$. Similar property holds for $v_\delta$.
  \end{enumerate}
\end{prop}

\begin{proof}
  To prove (1), note that if $\delta_1\leq \delta_2$ then $\tfrac{1}{\delta_1}|x-y|^2 \geq \tfrac{1}{\delta_2}|x-y|^2$ for all $x$ and $y$, and thus $u^{\delta_1}(x)\leq u^{\delta_2}(x)$ for all $x$. The respective statement for $v_{\delta_1}$ and $v_{\delta_2}$ is proved in the same way.
  Property (2) is obvious from the definitions. Property (3) follows from (2) and (5).
  
  Regarding (4) we note that the semi-convexity follows from the fact that $u^\delta(x)+\tfrac{1}{\delta}|x|^2$ is the supremum of affine functions and
  is hence convex. If we fix $x_0$ and if $x_0^*$ is such that	
  \begin{align*}  
    u^\delta(x_0)= u(x_0^*) -\frac{1}{\delta}|x_0-x_0^*|^2,
  \end{align*}
  then for all other $x$ we have $u^\delta(x) \geq P(x) := u(x_0^*)-\frac{1}{\delta}|x-x_0^*|^2$ by the definition of $u^\delta$, so $P$ is the desired paraboloid. To prove (5), let $x_0$ and $x_0^*$ be as above. Then 
  \begin{align*}  
    \tfrac{1}{\delta}|x_0-x_0^*|^2= u(x_0^*) -u^\delta(x_0) \leq u(x_0^*) -u(x_0)\leq 2\|u\|_\infty
  \end{align*}
so
  \[
    |x_0-x_0^*| \leq ( 2\delta \|u\|_\infty)^{1/2}.
  \]
  This means that $u(x_0^*)-u(x_0)$ is in fact bounded from above by $\omega \big ( ( 2\delta \|u\|_\infty)^{1/2} \big )$ which gives (5).
 
 Finally to show (6) we observe that if $x,y\in\Omega_{2h}$ and $u^\delta(x)=u(x^*)-\tfrac{1}{\delta}|x-x^*|^2$ then for small $\delta$,
 $x\in \Omega_{h}$. Now 
 if $|y-x|<h$, we have $u^\delta(y)\geq u(x^*+y-x)-\tfrac{1}{\delta}|x-x^*|^2$ so
 \[
 u^\delta(x)-u^\delta(y)\leq u(x^*)-u(x^*+y-x)\leq \omega (|x-y|)
 \]
 If $|y-x|\geq h$ then obviously $u^\delta(x)-u^\delta(y)\leq \tfrac{2}{h}\|u\|_\infty |y-x|$.
 \end{proof}

\begin{DEF}
  Given $y \in\mathcal{O}$, and the operator $I(\cdot,x)$ from (\ref{eqn:intro equation form}), we define
\begin{align}\label{eqn:freezing coefficients equation}
  I^{(y)}(\phi,x) = \sup\limits_{\alpha}\inf \limits_{\beta} \left \{ -L_{\mu^{\alpha\beta}_y}(\phi,x) +c_{\alpha\beta}(y)\phi(x)+ f_{\alpha\beta}(y) \right \},
\end{align}
where 
\begin{align}\label{eq:Lmualphabeta-def}
   L_{\mu^{\alpha\beta}_y}(\phi,x) = \int_{\mathbb{R}^d} [\phi(x+z)-\phi(x)-\chi_{B_1(0)}(z) D \phi(x)\cdot z] \;d\mu^{\alpha\beta}_{y}(z).
\end{align}
\end{DEF}
Note that this last expression is almost identical to $L^{\alpha\beta}(\phi,x)$, except that the L\'evy measure used is the one corresponding to the point $y$. Moreover the coefficients in (\ref{eqn:freezing coefficients equation}) are evaluated at $y$.

In the rest of this section, unless stated otherwise, we will always assume that Assumptions A-E are satisfied.

\begin{prop}\label{prop:equations for sup/inf convolutions}
  If $u$ is a viscosity subsolution of $I(u,x)= 0$ in $\mathcal{O}$, then $u^\delta$ is a viscosity subsolution of $I_{\delta}(u^\delta,x) =0$ in 
  $\mathcal{O}_{h}$, $h=(2\delta \|u\|_\infty)^{1/2}$, where
  \begin{align*}
    I_{\delta}(\phi,x) := \inf \left \{ I^{(y)}(\phi,x) : |y-x|\leq h\right \}.
  \end{align*}	 
  If $v$ is a viscosity supersolution of $I(v,x)= 0$ in $\mathcal{O}$, then $v_\delta$ is a viscosity supersolution of $I^{\delta}(v_\delta,x) =0$ in  
  $\mathcal{O}_{h}$, $h=(2\delta \|v\|_\infty)^{1/2}$, where
  \begin{align*}
    I^{\delta}(\phi,x) := \sup\left \{ I^{(y)}(\phi,x) : |y-x|\leq h\right \}.
  \end{align*}	 

\end{prop}

\begin{proof}
  Let us prove the statement for $u$ and $I_\delta$ (the corresponding one for $v$ and $I^\delta$ is entirely analogous and we omit it). Let $\phi$ touch $u^\delta$ from above at some $x_0 \in \mathcal{O}_h$. Let $x_0^* \in \mathbb{R}^d$ be such that
  \begin{align*}
    u^\delta(x_0) = u(x_0^*) - \frac{1}{\delta}|x_0-x_0^*|^2.
  \end{align*}
  It follows from part (5) of Proposition \ref{prop:sup/inf conv basic properties} that $x_0^*\in\mathcal{O}$.
  Then, by the definition of $u^\delta$, for any $x$ and $y$ we have
  \begin{align*}
    u^\delta(x+x_0-x_0^*) \geq u(y)-\frac{1}{\delta}|x+x_0-x_0^*-y|^2.
  \end{align*}
  Choosing $y=x$ it follows that for every $x$ we have
  \begin{align*}
    u^\delta(x+x_0-x_0^*) \geq u(x)-\frac{1}{\delta}|x_0-x_0^*|^2,
  \end{align*}
  with equality for $x=x_0^*$. It follows that if define a new test function $\phi^*(x)$ by
  \begin{align*}
    \phi^*(x) = \phi(x+x_0-x_0^*)+\frac{1}{\delta}|x_0-x_0^*|^2,
  \end{align*}
  then $\phi^*$ touches $u$ from above at $x_0^*$. Since $u$ is a subsolution, it follows that
  \begin{align*}
    I(\phi^*,x_0^*) \leq 0.	  
  \end{align*}	  
  Let us rewrite the expression on the left. First, recall
  \begin{align*}
    I(\phi^*,x_0^*) = \sup_\alpha \inf_\beta \{ -L^{\alpha\beta}(\phi^*,x_0^*) +c_{\alpha\beta}(x_0^*)\phi^*(x_0^*)+f_{\alpha\beta}(x_0^*) \}. 
  \end{align*}	  
  Next, note that
  \begin{align*}
    L^{\alpha\beta}(\phi^*,x_0^*) = \int_{\mathbb{R}^d} [\phi^*(x_0^*+z)-\phi^*(x_0^*)- \chi_{B_1(0)} D \phi^*(x_0^*)\cdot z] \;d\mu^{\alpha\beta}_{x_0^*}(z).
  \end{align*}
  Since, 
  \begin{align*}
    \phi^*(x_0^*+z)-\phi^*(x_0^*)- \chi_{B_1(0)} D\phi^*(x_0^*)\cdot z = \phi(x_0+z)-\phi(x_0)- \chi_{B_1(0)} D \phi(x_0)\cdot z
  \end{align*}
  it follows that
  \begin{align*}
    L^{\alpha\beta}(\phi^*,x_0^*) = \int_{\mathbb{R}^d} [\phi(x_0+z)-\phi(x_0)- \chi_{B_1(0)} D \phi(x_0)\cdot z]\; d\mu^{\alpha\beta}_{x_0^*}(z) = L_{\mu^{\alpha\beta}_{x_0^*}} (\phi,x_0).
  \end{align*}
  In conclusion
  \begin{align*}
   0\geq I(\phi^*,x_0^*) & = \sup_\alpha\inf_\beta \{- L_{\mu^{\alpha\beta}_{x_0^*}} (\phi,x_0) +c_{\alpha\beta}(x_0^*)(\phi(x_0)+\tfrac{1}{\delta}|x_0-x_0^*|^2 )+f_{\alpha\beta}(x_0^*) \}  \\
	& \geq \sup_\alpha \inf_\beta \{ -L_{\mu^{\alpha\beta}_{x_0^*}} (\phi,x_0) +c_{\alpha\beta}(x_0^*)\phi(x_0)+f_{\alpha\beta}(x_0^*) \} \\
	& \geq I_\delta(\phi,x_0).	
  \end{align*}	  
  Using part (5) of Proposition \ref{prop:sup/inf conv basic properties} in this last inequality, the proposition follows.
  
\end{proof}

Let us also state in a single lemma two basic facts about classical evaluation of L\'evy operators and viscosity solutions. The proof of the lemma goes along lines similar to those of the proofs of \cite{CafSil-2009}[Lemma 4.3 and Lemma 5.7]. 
\begin{lem}\label{lem:classical evaluation pointwise C1p-1}
For any function $u \in \textnormal{BUC}(\mathbb{R}^d)$ that is pointwise-$C^{1,1}$ at a point $x_0 \in \mathcal{O}$ 
(respectively, $x_0\in \mathcal{O}_h$) the operator $I(u,x_0)$ (respectively, $I_\delta(u,x_0)$) is classically defined. If furthermore $u$ is a viscosity subsolution of $I(u,x) =0$ in $\mathcal{O}$ (respectively, $I_\delta(u,x)\leq 0$ in $\mathcal{O}_h$), then also $I(u,x_0) \leq 0$ (respectively, $I_\delta(u,x_0)\leq 0$) pointwise. Similar statement is true for viscosity supersolutions.
\end{lem}

\begin{proof}
 We will only prove the statement for $I(u,x_0)$ as the other statements are proved similarly. Recall that from Assumption E,
  \begin{align*}
    \Lambda = \sup\limits_{x,\alpha,\beta} \{ \mathcal{N}_2(\hat \mu_x^{\alpha\beta}) + \check \mu_x^{\alpha\beta}(B_1^c) \}.
  \end{align*}
 From the pointwise-$C^{1,1}$ assumption at $x_0$, we have
  \begin{align*}
    & \int_{B_1} |u(x_0+z)-u(x_0)-\chi_{B_1(0)}(z)D u(x_0)\cdot z|\; d\mu_{x_0}^{\alpha\beta}(z) \leq \int_{B_1} C_{u,x_0}|z|^2\;d\mu_{x_0}^{\alpha\beta}(z),\\
    & \int_{B_1^c} |u(x_0+z)-u(x_0)|\;d\mu_{x_0}^{\alpha\beta}(z) \leq 2\|u\|_{\infty}\int_{B_1^c}\;d\mu_{x_0}^{\alpha\beta}(z),	
  \end{align*}
  where $C_{u,x_0}$ is from Definition \ref{def:pointwise Cp}. It thus follows that each integral defining $L^{\alpha\beta}(u,x_0)$ converges and
  \begin{align*}
    \sup \limits_{\alpha\beta }|L^{\alpha\beta}(u,x_0)| \leq (C_{u,x_0}+2\|u\|_\infty ) \Lambda <\infty.
  \end{align*}
  From here, it is immediate that $I(u,x_0)$ is classically defined. As for the second assertion, define
  \begin{align*}
    u_r := \left \{ \begin{array}{l}
      \phi \,\,\,\textnormal{ in } B_r(x_0),\\
      u \,\,\,\textnormal{ outside of } B_r(x_0),  	
    \end{array}\right.	
  \end{align*}
  where $\phi(x)=u(x_0)+Du(x_0)\cdot (x-x_0)+C_{u,x_0}|x-x_0|^2$.
  The function $\phi$ is touching $u$ from above in a neighborhood of $x_0$. From Definition \ref{def:viscosity solution ver1} we have $I(u_r,x_0)\leq 0$ for every $r>0$. On the other hand, 
  \begin{align*}
    I(u,x_0) \leq I(u_r,x_0)+\mathcal{M}^+_{I}(u-u_r,x_0),
  \end{align*}
  where the operator $\mathcal{M}^+_{I}$ is given by
  \begin{align*}
    \mathcal{M}^+_{I}(u-u_r,x_0) = \sup \limits_{\alpha,\beta} \left \{ - L^{\alpha\beta}(u-u_r,x_0)\right \}.	  
  \end{align*}	  
  Using the special form of $u_r$, particularly that $u_r = u$ outside of $B_r$, we have
  \begin{align*}
    \mathcal{M}^+_{I}(u-u_r,x_0) & = \sup \limits_{\alpha,\beta} \left \{  -L^{\alpha\beta}(u-u_r,x_0)) \right \}  \\
	  & = \sup \limits_{\alpha,\beta} \left \{ \int_{B_r} [\phi(x+z)-u(x+z)] \;d\mu_{x_0}^{\alpha\beta}(z) \right \}\\
	  & \leq 2C_{u,x_0}\sup \limits_{\alpha,\beta} \left \{ \int_{B_r} |z|^2 \;d\mu_{x_0}^{\alpha\beta}(z) \right \}\leq \theta(r),
  \end{align*}	  
  where the last inequality follows from (\ref{eqn:uniform integrability Levy measures}). Taking the limit as $r\to 0$, we conclude that
  \begin{align*}
    I(u,x_0)\leq 0.
  \end{align*}
\end{proof}

We will need smooth approximations of functions $|x-y|^p$ for $p\in[1,2]$. For $\kappa>0$ we define a function $\tilde\psi_\kappa:[0,+\infty)
\to [0,+\infty)$ by
\[
\tilde\psi_\kappa(r)=\big(\kappa+r^2\big)^{\frac{p}{2}}-\kappa^{\frac{p}{2}}.
\]
Then the function
\[
\psi_\kappa(x):=\tilde\psi_\kappa(|x|)
\]
is smooth and converges as $\kappa\to 0$ to $|x|^p$ uniformly on $\mathbb{R}^d$. We will be using the following lemma.

\begin{lem}\label{lem:psigamma}
Let $p\in[1,2]$. For every $R>0$ the function $\psi_\kappa(x)$ is uniformly pointwise-$C^p$ on $B_R$, i.e. there exists a constant $C_{p,R}$
such that for every $0<\kappa<1$ and every 
$x_0,x\in B_R$ 
  \[
    |\psi_\kappa(x)-\psi_\kappa(x_0)-D\psi_\kappa(x_0)\cdot (x-x_0)|\leq C_{p,R}|x-x_0|^p\quad \textnormal{ if } 1<p\leq 2,
    \]
 \[
    |\psi_\kappa(x)-\psi_\kappa(x_0)|\leq C_{1,R}|x-x_0|\quad \textnormal{ if } p=1.
    \]
 \end{lem}
 
 The following is the main lemma of the paper. We refer the reader to Definition \ref{def:Levy measure as a functional} 
for the definition of $L_\mu$.

\begin{lem}\label{lem:measures comparison singular}
  Let $u,v\in\textnormal{BUC}(\mathbb{R}^d)$.  Let $\alpha>0,p\geq 1,0<\kappa<1$ and suppose that $(x_*,y_*)\in \mathcal{O}\times\mathcal{O}$ is a global maximum point of the function
  \begin{align*}
    w(x,y) := u(x)-v(y)-\alpha \psi_\kappa(x-y).
  \end{align*}
  Furthermore, suppose that $u$ and $v$ are pointwise-$C^{1,1}$ at $x_*$ and $y_*$, respectively. Then, for any two L\'evy measures $\mu,\nu \in \mathbb{L}_p(B_1)$, we have the inequality
  \begin{align*}
    L_{\mu}(u,x_*)-L_{\nu}(v,y_*) \leq C_p \alpha \textnormal{d}_{\mathbb{L}_p}(\mu,\nu)^p,
  \end{align*}
  where $C_p$ is independent of $\kappa$.
\end{lem}

\begin{proof}
  First, note that as $(x_*,y_*)$ is a maximum point of $w$, we have
  \begin{align*}
    u(x) & \leq \alpha \psi_\kappa(x-y_*)+v(y_*)+(u(x_*)-v(y_*)-\alpha \psi_\kappa(x_*-y_*))\\	  
    v(y) & \geq -\alpha \psi_\kappa(x_*-y)+u(x_*)-(u(x_*)-v(y_*)-\alpha \psi_\kappa(x_*-y_*))	    	
  \end{align*}
  with equalities at $x_*$ and $y_*$ respectively.	  	
  Second, for any $(x,y)\in \mathbb R^d\times\mathbb R^d$
  \begin{align*}
    w(x_*+x,y_*+y)-w(x_*,y_*) \leq 0.   	  
  \end{align*}	  
  Let $\gamma \in \textnormal{Adm}(\mu,\nu)$. Using that $\gamma((B_1\times B_1)^c)=0$, and since 
$\delta u(x_*,0)=0$ and $\delta v(y_*,0)=0$, we thus have
  \begin{align*}	
    & L_{\mu}(u,x_*)-L_{\nu}(v,y_*)\\	  
    &  = \int_{B_1 \times B_1} \bigg(u(x_*+x)-v(y_*+y)-(u(x_*)-v(y_*)) \\
    & \qquad\qquad\qquad\qquad    -\alpha D\psi_\kappa(x_*-y_*)\cdot(x-y)\bigg)d\gamma(x,y).
  \end{align*}
  On the other hand, if $x,y\in B_1$, using Lemma \ref{lem:psigamma}, we also have
  \begin{align*}
    & u(x_*+x)-v(y_*+y)-(u(x_*)-v(y_*)) -\alpha D\psi_\kappa(x_*-y_*)\cdot(x-y)\\
    & \leq \alpha \psi_\kappa(x_*+x-y-y_*)-\alpha \psi_\kappa(x_*-y_*) -\alpha D\psi_\kappa(x_*-y_*)\cdot(x-y)\\
    & \leq C_p\alpha|x-y|^p.			
  \end{align*}
  Therefore, 
  \[
    L_{\mu}(u,x_*)-L_{\nu}(v,y_*)\leq C_p\alpha \int_{B_1 \times B_1}|x-y|^pd\gamma(x,y).
  \]
  Taking the infimum over all $\gamma\in \textnormal{Adm}(\mu,\nu)$, it thus follows that
  \begin{align*}
    L_{\mu}(u,x_*)-L_{\nu}(v,y_*) & \leq C_p \alpha\textnormal{d}_{\mathbb{L}_p}(\mu,\nu)^p.
  \end{align*}

\end{proof}

\begin{cor}\label{cor:measures comparison full}
  Let $u,v,x_*$, and $y_*$ be as in Lemma \ref{lem:measures comparison singular}, and let $\mu,\nu\in  \mathbb{L}_{p}(\mathbb{R}^d)$. Then, 
  \begin{align*}
    L_{\mu}(u,x_*)-L_{\nu}(v,y_*) \leq C_p \alpha \textnormal{d}_{\mathbb{L}_p}(\hat \mu,\hat\nu)^p + 2\|v\|_{\infty}\textnormal{d}_{\textnormal{TV}}(\check \mu, \check \nu).
  \end{align*}

\end{cor}

\begin{proof}
  Let us write the difference as follows
  \begin{align*}
    L_\mu(u,x_*)-L_\nu(v,y_*) = L_{\hat \mu}(u,x_*)-L_{\hat \nu}(v,y_*) + L_{\check \mu}(u,x_*)-L_{\check \nu}(v,y*).
  \end{align*}
  Thanks to Lemma \ref{lem:measures comparison singular}, the first difference in the right-hand side above is less than 
  or equal to $C_p\alpha\textnormal{d}_{\mathbb{L}_p}(\hat \mu,\hat\nu)^p$. For the second one, note that
  \begin{align*}
     L_{\check \mu}(u,x_*)-L_{\check \nu}(v,y_*) & = \int_{B_1^c}[u(x_*+z)-u(x_*)]\;d\mu(z)-\int_{B_1^c} [v(y_*+z)-v(y_*)]\;d\nu(z)\\
	 & = \int_{B_1^c} [u(x_*+z)-v(x_*+z)-(u(x_*)-v(y_*))]\;d\mu(z)
	 \\
	 &+\int_{B_1^c} [v(y_*+z)-v(y_*)]\;d(\mu-\nu)(z).
  \end{align*}
 Since $w$ achieves its global maximum at $(x_*,y_*)$, it follows that $u(x_*+z)-v(y_*+z)-(u(x_*)-v(y_*))\leq 0$. Hence we obtain 
  \begin{align*}
     L_{\check \mu}(u,x_*)-L_{\check \nu}(v,y_*) & \leq \int_{B_1^c} [v(y_*+z)-v(y_*)]\;d(\mu-\nu)(z).\\
	   & \leq 2\|v\|_{\infty}\textnormal{d}_{\textnormal{TV}}(\check \mu, \check \nu).
  \end{align*}
\end{proof}

We need a variant of a well known doubling lemma (see e.g. \cite[Lemma 3.1]{CIL}).

\begin{lem}\label{lem:properties of w} Let $u,v\in \textnormal{BUC}(\mathbb{R}^d)$ be such that $M=\sup(u-v)>\tau>0$ and
$u(x)-v(x)\leq 0$ for $x\in B_R^c$ for some $R>0$. For any $\varepsilon,\delta,\kappa>0$, set
  \begin{align*}
    w(x,y) & := u^{\delta}(x)-v_{\delta}(y)-\frac{1}{\varepsilon}\psi_\kappa(x-y),\\ 	  
    M_{\varepsilon,\delta,\kappa} & := \sup\limits_{\mathbb R^d\times\mathbb R^d} w(x,y).
  \end{align*}	
 Then, for sufficiently small $\delta,\varepsilon,\kappa$, there exist $(x_\varepsilon,y_\varepsilon)$ such that
  \begin{align*}
   M_{\varepsilon,\delta,\kappa}=w(x_\varepsilon,y_\varepsilon).
  \end{align*}
Then, we have{ 
  \begin{align}
    &   \frac{|x_\varepsilon-y_\varepsilon|^p}{\varepsilon} \leq \omega((C\varepsilon+c_{\kappa,\varepsilon,\delta})^{1/p})+\frac{c_{\kappa,\varepsilon,\delta}}{\varepsilon},\label{eq:lim0}\\
    & \lim_{\varepsilon \to 0}\lim_{\delta\to 0}\lim_{\kappa\to 0}  M_{\varepsilon,\delta,\kappa}  = M,\nonumber
  \end{align}where above $C=\|u\|_\infty+\|v\|_\infty$, $\omega$ is a modulus of continuity of $u$, and $c_{\kappa,\varepsilon,\delta}$ is a constant that converges to $0$ uniformly in $\varepsilon$ and $\delta$ as $\kappa\to 0^+$.}
  
  If $\Omega$ is an open subset of $\mathbb{R}^d$ and in addition { $u\in C^{0,r}(\Omega), 0<r\leq 1$}, and all the points
  $x_\varepsilon,y_\varepsilon\in \Omega$,
  then
  \begin{equation}\label{eq:calphaeps}
  \limsup_{\kappa\to 0}  \frac{|x_\varepsilon-y_\varepsilon|^{p-r}}{\varepsilon}\leq C_1
  \end{equation}
  for some constant $C_1$ independent of $\delta,\varepsilon,\kappa$.
  
\end{lem}
\begin{proof}
It is easy to see that the uniform convergence of the $u^{\delta},v_{\delta}$ to $u,v$, the uniform convergence of 
$\psi_\kappa(x-y)$
to $|x-y|^p$ and the uniform continuity of $u,v$ (and hence of $u^{\delta},v_{\delta}$, uniform in $\delta$) implies that for
sufficiently small $\delta,\varepsilon,\kappa$ we must have $w(x,y)\leq \tau/2$ when either $x$ or $y$ is in $B_R^c$. Thus $w$ must attain maximum at some point $(x_\varepsilon,y_\varepsilon)\in B_R\times B_R$.

Denote
\[
M_{\varepsilon,\delta} := \sup\limits_{\mathbb R^d\times\mathbb R^d} (u^{\delta}(x)-v_{\delta}(y)-\frac{1}{\varepsilon}|x-y|^p),
\]
\[
M_{\varepsilon} := \sup\limits_{\mathbb R^d\times\mathbb R^d} (u(x)-v(y)-\frac{1}{\varepsilon}|x-y|^p).
\]
Again, using the uniform convergence of $u^{\delta},v_{\delta},\psi_\kappa(x-y)$ and the uniform continuity of $u,v$ we easily find (see also the proof of \cite[Lemma 3.1]{CIL}) that
\[
\lim_{\kappa\to 0}  M_{\varepsilon,\delta,\kappa}=M_{\varepsilon,\delta},\quad
\lim_{\delta\to 0}  M_{\varepsilon,\delta}=M_{\varepsilon}, \quad
\lim_{\varepsilon\to 0}  M_{\varepsilon}=M.
\]
{ We obviously have
\begin{equation}\label{eq:limsupgamma}
 \frac{1}{\varepsilon}\psi_\kappa(x_\varepsilon-y_\varepsilon)+\frac{c_{\kappa,\varepsilon,\delta}}{\varepsilon}
= \frac{1}{\varepsilon}|x_\varepsilon-y_\varepsilon|^p
\end{equation}
where $c_{\kappa,\varepsilon,\delta}$ is a constant which converges to $0$ uniformly in $\varepsilon$ and $\delta$ as $\kappa\to 0^+$.
Now
\[
u^{\delta}(y_\varepsilon)-v_{\delta}(y_\varepsilon)
\leq u^{\delta}(x_\varepsilon)-v_{\delta}(y_\varepsilon)-\frac{1}{\varepsilon}\psi_\kappa(x_\varepsilon-y_\varepsilon)
\]
which, by (\ref{eq:limsupgamma}), implies
\[
  \frac{|x_\varepsilon-y_\varepsilon|^p}{\varepsilon}
\leq u^{\delta}(x_\varepsilon)-u^{\delta}(y_\varepsilon)+\frac{c_{\kappa,\varepsilon,\delta}}{\varepsilon}
 \leq \omega(|x_\varepsilon-y_\varepsilon|)+\frac{c_{\kappa,\varepsilon,\delta}}{\varepsilon}.
\]
This, together with the fact that we must have
\[
\frac{|x_\varepsilon-y_\varepsilon|^p}{\varepsilon}\leq \|u\|_\infty+\|v\|_\infty+\frac{c_{\kappa,\varepsilon,\delta}}{\varepsilon},
\]
gives (\ref{eq:lim0}). The last claim (\ref{eq:calphaeps}) follows by a similar argument since now
\[
\frac{1}{\varepsilon}(\limsup_{\kappa\to 0} |x_\varepsilon-y_\varepsilon|)^p
=\limsup_{\kappa\to 0}  \frac{|x_\varepsilon-y_\varepsilon|^p}{\varepsilon}
 \leq \limsup_{\kappa\to 0} C|x_\varepsilon-y_\varepsilon|^r=C(\limsup_{\kappa\to 0} |x_\varepsilon-y_\varepsilon|)^r.
\]
}
\end{proof}

\begin{proof}[Proof of Theorem \ref{thm:Main Comparison Result}]
  Arguing by contradiction, assume there is some $\ell>0$ such that
  \begin{align*}
    \sup \limits_{x\in\mathbb{R}^d} \{u(x)-v(x)\} = \ell>0.
  \end{align*}	  	
  
  \noindent Step 1. (Taking inf/sup-convolutions) 
  
  Let $u^\delta$ and $v_\delta$ denote the sup- and inf-convolutions of $u$ and $v$ for $\delta>0$. Then,
  \begin{align*}
    \sup \limits_{x\in\mathbb{R}^d} \{u^\delta(x)-v_\delta(x)\} \geq \ell.	  
  \end{align*}	  	
We may make $\delta_0$ small enough so that for $\delta<\delta_0$ we have
  \begin{align*}
    \sup \limits_{x\not\in \mathcal{O}} \{u^\delta(x)-v_\delta(x)\} \leq \tfrac{1}{4}\ell.
  \end{align*}	  	
 Recall that if $\omega$ is a modulus of continuity of $u$ and $v$, then it is also a modulus of continuity of $u^\delta$ and $v_\delta$. Therefore, reducing $\delta_0$ if necessary, we have
  \begin{align*}   
    u^\delta(x)-v_\delta(x) \leq \tfrac{1}{2}\ell \,\,\textnormal{ for } x \in \mathcal{O}\setminus \mathcal{O}_{2h_0},   
  \end{align*}   
  as long as $\delta<\delta_0$, where $\mathcal{O}_{h_0} = \{ x \in \mathcal{O} \mid d(x,\partial \mathcal{O})>h_0\}$ and
  $h_0>0$ is some constant. In particular, for such $\delta$ the supremum of $u^\delta-v_\delta$ in $\mathbb{R}^d$ can only be achieved within $\overline{\mathcal{O}}_{2h_0}$.\\
  
  \noindent Step 2. (Doubling of variables)
  
  For $\varepsilon,\delta,\kappa >0$, we let $w$ be as in Lemma \ref{lem:properties of w} and let $(x_\varepsilon,y_\varepsilon)\in \mathbb{R}^d\times\mathbb{R}^d$ be such that
  \begin{align*}
    w(x_\varepsilon,y_\varepsilon) = \max \limits_{\mathbb{R}^d \times \mathbb{R}^d} w(x,y).
  \end{align*}
  From Step 1, we know that $u^\delta-v_\delta \leq \ell/2$ in { $\mathcal{O}\setminus \mathcal{O}_{2h_0}$} and $u^{\delta}-v_{\delta} \geq \ell$ somewhere in $\mathcal{O}$. Furthermore, we know $u^\delta$ and $v_\delta$ are uniformly continuous in $\mathcal{O}$, and uniformly so with respect to $\delta<1$. From these facts, and (\ref{eq:lim0}), it follows that $(x_\varepsilon,y_\varepsilon)$ must belong to { $\mathcal{O}_{h_0}\times\mathcal{O}_{h_0}$} for all sufficiently small $\varepsilon,\delta$ and $\kappa$ or else it cannot be the maximum point of $w^\varepsilon$.
  
  On the other hand, Proposition \ref{prop:equations for sup/inf convolutions} says that $u^\delta$ is a viscosity subsolution of
$I_\delta(u^\delta,x) = 0$ and $v_\delta$ is a viscosity supersolution of $I^\delta(v_\delta,x) = 0$ in { $\mathcal{O}_{h_0}$} for sufficiently small
$\delta$.
  The function $u^\delta$ is touched from above by a smooth function at $x_\varepsilon$ and $v_\delta$ is touched from below at $y_\varepsilon$. It follows that $u^\delta$ and $v_\delta$ are pointwise-$C^{1,1}$  at $x_\varepsilon$ and $y_\varepsilon$, respectively (see Definition \ref{def:pointwise Cp}). Applying Lemma \ref{lem:classical evaluation pointwise C1p-1}, we conclude that $I_\delta(u^\delta,x_\varepsilon)$ and $I^\delta(v_\delta,
  y_\varepsilon)$ are well defined in the classical sense, with $I_\delta(u^\delta,x_\varepsilon) \leq 0$ and $I^\delta(v_\delta,x_\varepsilon) \leq 0$. It follows from Proposition \ref{prop:equations for sup/inf convolutions} that there are points $x_\delta^*$ and $y_\delta^*$ such that
  \begin{align*}
    I^{(x_\delta^*)}(u^\delta,x_\varepsilon) &\leq \delta, \;\;I^{( y_\delta^*)}(v_\delta,y_\varepsilon) \geq -\delta,
  \end{align*}	
  and
  \begin{equation}\label{eq:hdelta}
  |x_\varepsilon-x_\delta^*|,|y_\varepsilon-y_\delta^*|\leq h,
   \end{equation}
  where $h=\big(2\delta(\|u\|_\infty+\|v\|_\infty)\big)^{\frac{1}{2}}$.
 
  \noindent Step 3. (Equation structure)	
  
  Let us use the structure of $I(\cdot,x)$ to bound $I^{(x_\delta^*)}(u^\delta,x_\varepsilon)-I^{(y_\delta^*)}(v^\delta,y_\varepsilon)$ from below. Using the expression in \eqref{eqn:freezing coefficients equation}, we have
  \begin{align*}
    I^{(x_\delta^*)}(u^\delta,x_\varepsilon) & = \sup \limits_{\alpha}\inf \limits_{\beta} \left \{ -L_{\mu^{\alpha\beta}_{x_\delta^*}}(u^\delta,x_\varepsilon)+c_{\alpha\beta}(x_\delta^*)u^\delta(x_\varepsilon) + f_{\alpha\beta}(x_\delta^*) \right \},\\
    I^{(y_\delta^*)}(v_\delta,y_\varepsilon) & = \sup \limits_{\alpha} \inf \limits_{\beta} \left \{ -L_{\mu^{\alpha\beta}_{y_\delta^*}}
    (v_\delta,y_\varepsilon)+c_{\alpha\beta}(y_\delta^*)v_\delta(y_\varepsilon) + f_{\alpha\beta}(y_\delta^*) \right \}.
  \end{align*}
  Therefore, for our purposes it suffices to compare the expressions appearing on the right hand side for each fixed $\alpha,\beta$. Let us write
  \begin{align*}
    & (\textnormal{I})_{\alpha\beta} = -L_{\mu^{\alpha\beta}_{x_\delta^*}}(u^\delta,x_\varepsilon)+c_{\alpha\beta}(x_\delta^*)u^\delta(x_\varepsilon) + f_{\alpha\beta}(x_\delta^*),\\
    & (\textnormal{II})_{\alpha\beta} = -L_{\mu^{\alpha\beta}_{y_\delta^*}}(v_\delta,y_\varepsilon)+c_{\alpha\beta}(y_\delta^*)v_\delta(y_\varepsilon) + f_{\alpha\beta}(y_\delta^*).
  \end{align*}
  We now look for an upper bound for $(\textnormal{I})_{\alpha\beta}-(\textnormal{II})_{\alpha\beta}$ which is independent of $\alpha$ and $\beta$ by breaking this difference into parts. First, recall that the function $u^\delta(x)-v_\delta(y)-\tfrac{1}{\varepsilon}\psi_\kappa(x-y)$ achieves its global maximum at $(x_\varepsilon,y_\varepsilon)$, in which case Corollary \ref{cor:measures comparison full} guarantees that
{   \begin{align*}
     L_{\mu^{\alpha\beta}_{x_\delta^*}}(u^\delta,x_\varepsilon)-L_{\mu^{\alpha\beta}_{y_\delta^*}}(v_\delta,y_\varepsilon) & \leq \frac{C}{\varepsilon}\textnormal{d}_{\mathbb L_p}(\hat \mu^{\alpha\beta}_{x_\delta^*},\hat \mu^{\alpha\beta}_{y_\delta^*})^p+2\|v_\delta\|_{\infty} \textnormal{d}_{\textnormal{TV}}(\check \mu_{x_\delta^*},\check \mu_{y_\delta^*})
  \end{align*}  
  }
  Then, thanks to Assumptions A and B, and (\ref{eq:hdelta}), we have
  \begin{eqnarray}
    L_{\mu^{\alpha\beta}_{x_\delta^*}}(u^\delta,x_\varepsilon)-L_{\mu^{\alpha\beta}_{y_\delta^*}}(v_\delta,y_\varepsilon) & \leq &\frac{C}{\varepsilon}
    |x_\delta^*-y_\delta^*|^p +2\|v\|_{\infty}\theta(|x_\delta^*-y_\delta^*|)
    \nonumber
    \\
    &\leq&
    \frac{C}{\varepsilon}
    |x_\varepsilon-y_\varepsilon|^p +2\|v\|_{\infty}\theta(|x_\varepsilon-y_\varepsilon|)+\rho_\varepsilon(\delta),
    \label{eqn:main comparison Levy part estimate}
  \end{eqnarray}
  where for a fixed $\varepsilon$, $\lim_{\delta\to 0}\rho_\varepsilon(\delta)=0$.
 
  Next, we have the elementary inequality
  \begin{eqnarray}     
    c_{\alpha\beta}(x_\delta^*)u^\delta(x_\varepsilon) - c_{\alpha\beta}(y_\delta^*)v_\delta(y_\varepsilon) & \geq& c_{\alpha\beta}(x_\delta^*)(u^\delta
    (x_\varepsilon) - v_\delta(y_\varepsilon)) -|c_{\alpha\beta}(x_\delta^*) - c_{\alpha\beta}(y_\delta^*) ||v_\delta(y_\varepsilon)|
    \nonumber
    \\
    &\geq&
    \lambda\ell -\theta(|x_\varepsilon -y_\varepsilon|)\|v\|_\infty-\rho_\varepsilon(\delta),
    \label{eqn:main comparison zero part estimate}
  \end{eqnarray}
  where $\rho_\varepsilon(\delta)$ is a function as before and we used that $u^\delta(x_\varepsilon) - v_\delta(y_\varepsilon)\geq \ell$.

  Finally, by Assumption C
  \begin{align}\label{eqn:main comparison free term part estimate}
   | f_{\alpha\beta}(x_\varepsilon^*)-f_{\alpha\beta}(y_\varepsilon^*)|  \leq \theta (|x_{\varepsilon}-y_{\varepsilon}|)+\rho_\varepsilon(\delta).
  \end{align}
  
  Now, combining \eqref{eqn:main comparison Levy part estimate}, \eqref{eqn:main comparison zero part estimate}, \eqref{eqn:main comparison free term part estimate}, we have the estimate
  \begin{align*}
    (\textnormal{I})_{\alpha\beta} - (\textnormal{II})_{\alpha\beta} \geq \lambda \ell- \frac{C}{\varepsilon}
    |x_\varepsilon-y_\varepsilon|^p
     -C \theta\big(|x_\varepsilon-y_\varepsilon|) -\rho_\varepsilon(\delta),
  \end{align*}  
  where $C$ above is some absolute constant. Therefore we conclude that
  \begin{align*}
    I^{(x_\delta^*)}(u^\delta,x_\varepsilon)-I^{(y_\delta^*)}(v_\delta,y_\varepsilon) \geq \lambda \ell- \frac{C}{\varepsilon}
    |x_\varepsilon-y_\varepsilon|^p
     -C \theta\big(|x_\varepsilon-y_\varepsilon|) -\rho_\varepsilon(\delta).
  \end{align*}  
  
  \noindent Step 4. (Using the subsolution and supersolution property)
  
  Recalling the way $x_\delta^*$ and $y_\delta^*$ were selected, we have $I^{(x_\delta^*)}(u^\delta,x_\varepsilon)-I^{(y_\delta^*)}(v_\delta,y_\varepsilon) \leq 2\delta$, and therefore
  \begin{align*}
    \lambda \ell\leq 2\delta +\frac{C}{\varepsilon}
    |x_\varepsilon-y_\varepsilon|^p
     +C \theta\big(|x_\varepsilon-y_\varepsilon|) +\rho_\varepsilon(\delta).
  \end{align*}  
 It now remains to take $\lim_{\varepsilon \to 0}\lim_{\delta\to 0}\limsup_{\kappa\to 0}$ on both sides of the above inequality and use (\ref{eq:lim0})  to obtain a contradiction.
\end{proof}

\subsection{Estimating $\textnormal{d}_{\mathbb{L}_p}$ in special cases}\label{subsec:basic examples}

%%%%%%%%%%%%%%%%%%%%%%%%%%%%%%%%%
%%%%%%%%%%%%%%%%%%%%%%%%%%%%%%%%%
%%%%%%%%%%%%%%%%%%%%%%%%%%%%%%%%%
%%%%%%%%%%%%%%%%%%%%%%%%%%%%%%%%%
\begin{prop}\label{prop:bound via dual form}
Let $p\in[1,2]$. 

\noindent
(i) Let $\mu,\nu\in \mathbb{L}_p(B_1)$ {  and $\phi,\psi\in  \textnormal{Adm}^p$.}
% on $\overline B_1$ such that $\phi(0)=\psi(0)=0$, and
%$\phi(x)+\psi(y)\leq |x-y|^p$ for all $x,y\in \overline B_1$.  
If $\mu-\nu$ is a positive measure then
  \begin{equation}\label{eq:ptv}
    \int_{B_1} \phi(x)d\mu(x) +\int_{B_1} \psi(y)d\nu(y) \leq \int_{B_1} |x|^p d(\mu-\nu)(x).	
  \end{equation}
  
\noindent
(ii) For any $\mu,\nu\in \mathbb{L}_p(B_1)$ we have
  \begin{equation}\label{eq:ptv1}
  \textnormal{d}_{\mathbb{L}_p}( \mu,\nu)\leq 2^{\frac{p-1}{p}}d_{\textnormal{TV}}(\mu_p,\nu_p)^{\frac{1}{p}},	
  \end{equation}
  where $d\mu_p=|x|^pd\mu, d\mu_p=|x|^pd\mu$. 
\end{prop}

\begin{proof} (i):
%We have $\phi(x)\leq |x|^p$. 
%If $\int_{B_1} \phi(x)d\mu(x)=-\infty$ there is nothing to prove. 
{  Since $\phi\in L^1(\mu)$ we also have $\phi\in L^1(\nu)$.}
 Then we may write 	
  \begin{align*}
    & \int_{B_1} \phi(x)d\mu(x) +\int_{B_1} \psi(y)d\nu(y)\\
    & = \int_{B_1} \phi(x)d\mu(x)-\int_{B_1}\phi(x)d\nu(x)+\int_{B_1}\phi(y)d\nu(y) +\int_{B_1} \psi(y)d\nu(y)\\	
    &=\int_{B_1} \phi(x)d(\mu-\nu)(x)+\int_{B_1} (\phi(y)+\psi(y))d\nu(y)\\
    & \leq  \int_{B_1} \phi(x)d(\mu-\nu)(x)\leq \int_{B_1} |x|^p d(\mu-\nu)(x), 	
  \end{align*}
where in the last line we used $\phi(y)+\psi(y) \leq |y-y|^p = 0\;\forall\;y\in B_1$ and $\phi(x)\leq |x|^p
\;\forall\;x\in B_1$. 

(ii): Denoting by $(\mu-\nu)^+$ and $(\mu-\nu)^-$, the positive and negative parts of $\mu-\nu$, we have
$|\mu-\nu|=(\mu-\nu)^++(\mu-\nu)^-$. We also notice that $\mu-(\mu-\nu)^+=\nu-(\mu-\nu)^-$. It thus follows from 
\eqref{eq:ptv} and Lemma \ref{lem:duality general} that
\[
\textnormal{d}_{\mathbb{L}_p}( \mu,\mu-(\mu-\nu)^+)^p\leq \int_{B_1} |x|^p d(\mu-\nu)^+(x),
\]
\[
\textnormal{d}_{\mathbb{L}_p}( \nu,\mu-(\mu-\nu)^+)^p\leq \int_{B_1} |x|^p d(\mu-\nu)^-(x).
\]
Moreover it is obvious that
\[
d_{\textnormal{TV}}(\mu_p,\nu_p)=\int_{B_1} |x|^p d(\mu-\nu)^+(x)+\int_{B_1} |x|^p d(\mu-\nu)^-(x).
\]
Therefore, using the triangle inequality for the distance and the inequality $a+b\leq 2^{\frac{p-1}{p}}(a^p+b^p)^{\frac{1}{p}}$ for $a,b\geq 0$, we obtain
 \begin{align*}
   \textnormal{d}_{\mathbb{L}_p}( \mu,\nu)&\leq \textnormal{d}_{\mathbb{L}_p}( \mu,\mu-(\mu-\nu)^+)
   +\textnormal{d}_{\mathbb{L}_p}( \nu,\mu-(\mu-\nu)^+)
   \\
   &\leq 2^{\frac{p-1}{p}}\left(\int_{B_1} |x|^p d(\mu-\nu)^+(x)+\int_{B_1} |x|^p d(\mu-\nu)^-(x)\right)^{\frac{1}{p}}
   \\
   &
   =2^{\frac{p-1}{p}}d_{\textnormal{TV}}(\mu_p,\nu_p)^{\frac{1}{p}}.
   \end{align*}
   \end{proof}

Let us now discuss the case when the L\'evy measures $\mu_x^{\alpha\beta}$ are absolutely continuous with respect to the Lebesgue measure.
%%%%%%%%%%%%%%%%%%%%%%%%%%%%%%%%%
%%%%%%%%%%%%%%%%%%%%%%%%%%%%%%%%%
%%%%%%%%%%%%%%%%%%%%%%%%%%%%%%%%%
%%%%%%%%%%%%%%%%%%%%%%%%%%%%%%%%%
\begin{ex}\label{EX:measures with kernels of order sigma<1}
  Let us consider operators whose L\'evy measures $d\mu_x^{\alpha\beta}(z)$ are all of the form $K_{\alpha\beta}(x,z)dz$. 
  Assumption E holds for instance if  
  \begin{align}\label{eq:int111222}
    0\leq K_{\alpha\beta}(x,z)\leq K(z),
  \end{align}
  where $K(z)$ is such that for some $p\in [1,2]$
  \begin{equation}\label{eq:3.122aa}
    \int_{\mathbb R^d}\min(1,|z|^p)K(z)dz <+\infty.
  \end{equation}
  Regarding Assumption A, suppose that there are some $\gamma \in (0,1]$ and $C\geq 0$, such that 
  \begin{align}\label{eq:K reg1}
   \int_{B_1}|z|^p |K_{\alpha\beta}(x,z)-K_{\alpha\beta}(y,z)|dz\leq C|x-y|^{\gamma}\quad\forall x,y\in\mathcal{O},\forall \alpha,\beta.
  \end{align}
  Condition \eqref{eq:K reg1} is obviously satisfied if
   \begin{align}\label{eq:K reg}
  |K_{\alpha\beta}(x,z)-K_{\alpha\beta}(y,z)|dz\leq |x-y|^{\gamma}K(z)\quad\forall x,y\in\mathcal{O},\forall z\in B_1,\forall \alpha,\beta.
  \end{align}
 Let now $x,y\in\mathcal{O}$. To estimate $\textnormal{d}_{\mathbb{L}_p}(\hat\mu^{\alpha\beta}_x, \hat\mu^{\alpha\beta}_y)$, we use Proposition \ref{prop:bound via dual form}.  It follows from \eqref{eq:ptv1}, and (\ref{eq:K reg1}), that
  \begin{align}
    \textnormal{d}_{\mathbb{L}_p}(\hat \mu_x^{\alpha\beta},\hat \mu_y^{\alpha\beta})\leq 
    2^{\frac{p-1}{p}}\left(\int_{B_1}|z|^p\left|K_{\alpha\beta}(x,z)-K_{\alpha\beta}(y,z)\right|dz\right)^{\frac{1}{p}}
    \leq C|x-y|^{\frac{\gamma}{p}}.\label{eq:3.11aa}
  \end{align}
  In particular, \eqref{eqn:LipL2} is satisfied for these measures when $p=1$ and $\gamma=1$.
\end{ex}

We can now prove Corollary \ref{cor:Comparison Theorem Ops Order Less than 1}.
 
\begin{proof}[Proof of Corollary \ref{cor:Comparison Theorem Ops Order Less than 1}]
We notice that the measures $\mu_x^{\alpha\beta}$ satisfy \eqref{eq:int111222}, \eqref{eq:3.122aa} and \eqref{eq:K reg} with 
$\gamma=p=1$ and $K(z)=
\Lambda_1|z|^{-d-\sigma}$ and hence they satisfy Assumptions A and E. It is also easy to see that they satisfy Assumption B. Thus the result follows from Theorem \ref{thm:Main Comparison Result}.
\end{proof}

\begin{ex}\label{ex:levyito}
  A well studied subclass of operators which arise in zero-sum two-player stochastic differential games are those of L\'evy-It\^ o form. This corresponds to the situation where the $L^{\alpha\beta}$ appearing in \eqref{eqn:intro equation form} have the form
  \begin{align}\label{eq:levyito}
    L^{\alpha\beta}(u,x) = \int_{U\setminus\{0\}}[u(x+T_x^{\alpha\beta}(z))-u(x)-Du(x)\cdot T_x^{\alpha\beta}(z)] \;d\mu(z).
  \end{align}
  Here $U$ is a separable Hilbert space and $\mu$ is a fixed reference L\'evy measure on $U\setminus\{0\}$.  The maps $T_x^{\alpha\beta}:U\to\mathbb R^d$ are Borel measurable and  such that for all $\alpha\in\mathcal{A}, \beta\in\mathcal{B}, x,y\in\mathcal{O}, z\in U\setminus\{0\}$,
  \begin{align*}
    |T_x^{\alpha\beta}(z)-T_y^{\alpha\beta}(z)|\leq C\rho(z)|x-y|,\quad |T_x^{\alpha\beta}(z)|\leq C\rho(z),
  \end{align*}
  for some positive Borel function $\rho:U\setminus\{0\}\to \mathbb{R}$ which is bounded on bounded sets, $\inf_{|z|>r}\rho(z)>0$ for every $r>0$, and
  \begin{equation}\label{eq:rho2}
    \int_{U\setminus\{0\}}\rho(z)^2d\mu(z)\leq C.
  \end{equation}
  
  Under these conditions the measures $\mu_x^{\alpha\beta}=(T_x^{\alpha\beta})_\# \mu$ are L\'evy measures. The comparison principle for 
  sub/super solutions of \eqref{eqn:intro equation form} with $L^{\alpha\beta}$ as in \eqref{eq:levyito} is known to hold, as discussed in the introduction. Let us revisit it using the transport metric.
  For every $x_0,y_0\in\mathcal{O},\alpha\in\mathcal{A}, \beta\in\mathcal{B}$, we have $\gamma=(T_{x_0}^{\alpha\beta}\times T_{y_0}^{\alpha\beta})_\#\mu\in 
  \textnormal{Adm}(\mu_{x_0}^{\alpha\beta},\mu_{y_0}^{\alpha\beta})$ and therefore,
  \begin{align*}
    \textnormal{d}_{\mathbb{L}_2}(\mu_{x_0}^{\alpha\beta},\mu_{y_0}^{\alpha\beta})^2 \leq \int_{\mathbb{R}^d\times\mathbb{R}^d}|x-y|^2d\gamma(x,y)
  =\int_{U\setminus\{0\}}|T_{x_0}^{\alpha\beta}(z)-T_{y_0}^{\alpha\beta}(z)|^2\;d\mu(z)\leq C|x_0-y_0|^2.
  \end{align*}
  In this case the whole measures $\mu_{x_0}$ and $\mu_{y_0}$ satisfy Assumption A for $p=2$ and our approach
  can be applied without the decomposition of the measures $\mu_{x}^{\alpha\beta}$ into $\hat\mu_{x}^{\alpha\beta}$
  and $\check\mu_{x}^{\alpha\beta}$.
 
 If  the operators $L^{\alpha\beta}$ in \eqref{eqn:intro equation form} have a more common L\'evy-It\^ o form
  \begin{equation}\label{eq:levyito1}
    L^{\alpha\beta}(u,x) = \int_{U\setminus\{0\}}[u(x+T_x^{\alpha\beta}(z))-u(x)-\chi_{B_1(0)}Du(x)\cdot T_x^{\alpha\beta}(z)] \;d\mu(z),
  \end{equation}
  where instead of (\ref{eq:rho2}) we now only assume
\[
    \int_{U\setminus\{0\}}(\rho(z)^2\chi_{|z|<1}+\chi_{|z|\geq 1})d\mu(z)\leq C,
  \]
we need to modify this approach. We now do the decomposition
\[
\mu_{x}^{\alpha\beta}=\hat\mu_{x}^{\alpha\beta}+\check\mu_{x}^{\alpha\beta},
\]
where
\[
\hat\mu_{x}^{\alpha\beta}=(T_x^{\alpha\beta})_\#\hat\mu,\,\,\hat\mu=\mu|_{\{|z|<1\}},
\quad \check\mu_{x}^{\alpha\beta}=(T_x^{\alpha\beta})_\#\check\mu,\,\,\check\mu=\mu|_{\{|z|\geq 1\}}
\]
and consider the measures $\hat\mu_{x}^{\alpha\beta}+\check\mu_{x}^{\alpha\beta}$ and $\check\mu_{x}^{\alpha\beta}$ as measures on
$\mathbb{R}^d$ by the usual extension. Then the measures $\hat\mu_{x}^{\alpha\beta}\in \mathbb{L}_2(\mathbb{R}^d)$
and they satisfy Assumption A for $p=2$. Unfortunately the measures $\check\mu_{x}^{\alpha\beta}$ may not
satisfy Assumption B now, however the terms containing them can be handled in a standard way (see e.g. \cite{jk1}) and thus
our approach can still be implemented {  (see also the next section and Example \ref{ex:fracLapl})}.
\end{ex}

\section{Variants of the approach}\label{sec:variants}
The approach to proving a comparison principle presented so far was based on the splitting of the measures $\mu_{x}^{\alpha\beta}$
into $\hat\mu_{x}^{\alpha\beta}$ and $\check\mu_{x}^{\alpha\beta}$, their restrictions to $B_1(0)$ and $B_1^c(0)$ respectively.
The reader should think about it as the basic technique. However in many cases, this splitting may not be ideal. When calculating the distance between two measures $\hat\mu_{x}^{\alpha\beta}$ and $\hat\mu_{y}^{\alpha\beta}$, we only have the set $\Gamma=\{0\}$
where we can deposit some excess mass and moving mass there may be costly. Thus sometimes a much better estimate can be obtained
if we allow for a more sophisticated splitting $\mu_{x}^{\alpha\beta}=\hat\mu_{x}^{\alpha\beta}+\tilde\mu_{x}^{\alpha\beta}
+\check\mu_{x}^{\alpha\beta}$, where the measures $\hat\mu_{x}^{\alpha\beta}$ are now supported in some neighborhoods of the
origin contained in $B_1(0)$, $\tilde\mu_{x}^{\alpha\beta}$ are bounded measures also supported in $B_1(0)$, and $\check\mu_{x}^{\alpha\beta}$ are as in \eqref{eqn:Levy measure singular regular decomposition}. In such a case
we may only require that Assumption A (i.e. \eqref{eqn:LipL2}) be satisfied for the new measures $\hat\mu_{x}^{\alpha\beta}$. We will illustrate the advantage of this approach in Example \ref{ex:fracLapl}. Thus the main message is that we should look at the technique of using coupling distance
in the proof of comparison principle as flexible, and Assumption A should really be considered to be an assumption about the behavior of L\'evy measures
for small $z$, not necessarily for $z\in B_1(0)$.

Suppose then that for every $\alpha,\beta$ we have a decomposition $\mu_{x}^{\alpha\beta}=\hat\mu_{x}^{\alpha\beta}+\tilde\mu_{x}^{\alpha\beta}
+\check\mu_{x}^{\alpha\beta}$ as described above, and we decompose
\begin{align*}
  L^{\alpha\beta}(u,x) & = \hat L^{\alpha\beta}(u,x) + \tilde L^{\alpha\beta}(u,x)+\check L^{\alpha\beta}(u,x),
\end{align*}
where
\begin{align*}
  \hat L^{\alpha\beta}(u,x) &= \int_{\mathbb{R}^d} \delta u(x,z)\;d\hat \mu^{\alpha\beta}_x(z),\quad
    \tilde L^{\alpha\beta}(u,x)= \int_{\mathbb{R}^d} \delta u(x,z)\;d\tilde \mu^{\alpha\beta}_x(z),\\
  \check L^{\alpha\beta}(u,x) & = \int_{\mathbb{R}^d} [u(x+z)-u(x)] \;d\check \mu^{\alpha\beta}_x(z).
\end{align*}
We can then prove the following variant of Theorem \ref{thm:Main Comparison Result}.
\begin{thm}\label{thm:Main Comparison Result-1}
  Let Assumptions A-E (with the new measures $\hat \mu^{\alpha\beta}_x$) hold for $p\in[1,2]$, and suppose there
  are $L_1,L_2\geq 0$ such that
   \begin{equation}\label{eq:mutilde1}
\tilde \mu^{\alpha\beta}_x(B_1(0))\leq L_1\quad\forall x\in\mathcal{O},\,\,\forall\alpha,\beta,
  \end{equation}
  \begin{equation}\label{eq:mutilde2}
 \int_{\mathbb{R}^d}|z|\,d|\tilde \mu^{\alpha\beta}_x-\tilde \mu^{\alpha\beta}_y|(z)\leq L_2|x-y| \quad\forall x,y\in\mathcal{O},\,\,\forall\alpha,\beta.
  \end{equation}
  Then the comparison principle holds for equation (\ref{eqn:intro equation form}). That is, if $u$ and $v$ are respectively a viscosity subsolution and a viscosity supersolution of (\ref{eqn:intro equation form}) and $u(x)\leq v(x)$ for all $x\not\in \mathcal{O}$, then 
  \begin{align*}
    u(x)\leq v(x)\;\;\forall\;x\in\mathcal{O}.
  \end{align*}
\end{thm}
\begin{proof}
The proof proceeds exactly as the proof of Theorem \ref{thm:Main Comparison Result} except that now in Step 3 we also need to find an estimate from
above for
\[
L_{\tilde\mu^{\alpha\beta}_{x_\delta^*}}(u^\delta,x_\varepsilon)-L_{\tilde\mu^{\alpha\beta}_{y_\delta^*}}(v_\delta,y_\varepsilon)
\]
which is independent of $\alpha$ and $\beta$, where the operators $L_{\tilde\mu^{\alpha\beta}_{x_\delta^*}}(u^\delta,x_\varepsilon)$ and $L_{\tilde\mu^{\alpha\beta}_{y_\delta^*}}(v_\delta,y_\varepsilon)$
are defined as in \eqref{eq:Lmualphabeta-def} for the measures $\tilde\mu^{\alpha\beta}_{x_\delta^*}$ and 
$\tilde\mu^{\alpha\beta}_{y_\delta^*}$. We have
\begin{align*}
&L_{\tilde\mu^{\alpha\beta}_{x_\delta^*}}(u^\delta,x_\varepsilon)-L_{\tilde\mu^{\alpha\beta}_{y_\delta^*}}(v_\delta,y_\varepsilon)
\leq
\int_{B_1}\left|u^\delta(x_\varepsilon+z)-u^\delta(x_\varepsilon)-\frac{1}{\varepsilon}D\psi_\kappa(x_\varepsilon-y_\varepsilon)\cdot z\right|
d|\tilde\mu^{\alpha\beta}_{x_\delta^*}-\tilde\mu^{\alpha\beta}_{y_\delta^*}|(z)
\\
&
+\int_{B_1}\left(u^\delta(x_\varepsilon+z)-u^\delta(x_\varepsilon)-(v_\delta(y_\varepsilon+z)-v_\delta(y_\varepsilon))\right)
d\tilde\mu^{\alpha\beta}_{y_\delta^*}(z).
\end{align*}
Let $\omega$ be the modulus of continuity of $u$. It is also a modulus of continuity for $u^\delta$. The modulus $\omega$ is bounded and we can assume that it is concave. We notice that the integrand of
the second integral above is non-positive. Therefore, we obtain
\[
L_{\tilde\mu^{\alpha\beta}_{x_\delta^*}}(u^\delta,x_\varepsilon)-L_{\tilde\mu^{\alpha\beta}_{y_\delta^*}}(v_\delta,y_\varepsilon)
\leq
\int_{B_1}\left(\omega(|z|)+\frac{1}{\varepsilon}|D\psi_\kappa(x_\varepsilon-y_\varepsilon)||z|\right)
d|\tilde\mu^{\alpha\beta}_{x_\delta^*}-\tilde\mu^{\alpha\beta}_{y_\delta^*}|(z).
\]
Using \eqref{eq:mutilde1}, \eqref{eq:mutilde2}, the concavity of $\omega$, Jensen's inequality, the subadditivity of $\omega$, and
$|D\psi_\kappa(x_\varepsilon-y_\varepsilon)|\leq p|x_\varepsilon-y_\varepsilon|^{p-1}$,
we can now estimate 
\begin{align*}
L_{\tilde\mu^{\alpha\beta}_{x_\delta^*}}(u^\delta,x_\varepsilon)-L_{\tilde\mu^{\alpha\beta}_{y_\delta^*}}(v_\delta,y_\varepsilon)
&\leq
C\omega\left( \int_{\mathbb{R}^d}|z|\,d|\tilde \mu^{\alpha\beta}_{x_\delta^*}-\tilde \mu^{\alpha\beta}_{y_\delta^*}|(z)\right)
+\frac{C}{\varepsilon}|x_\varepsilon-y_\varepsilon|^{p-1}|x_\delta^*-y_\delta^*|
\\
&\leq C\omega\left(|x_\delta^*-y_\delta^*|\right)
+\frac{C}{\varepsilon}|x_\varepsilon-y_\varepsilon|^{p-1}|x_\delta^*-y_\delta^*|.
\end{align*}
This allows us to complete the proof by following the rest of the proof of Theorem \ref{thm:Main Comparison Result}.
\end{proof}

The next example illustrates the usefulness of this modified approach and Theorem \ref{thm:Main Comparison Result-1}.

\begin{ex}\label{ex:fracLapl}
Let the measures $\mu^{\alpha\beta}_x$ be such that
\[
d\mu^{\alpha\beta}_x(z)=\frac{a_{\alpha\beta}(x)}{|z|^{d+\sigma}}
\]
for some $1<\sigma<2$. 
Assume that the functions $a_{\alpha\beta}:\mathcal{O}\to\mathbb{R}$ are nonnegative and such that there exists $L\geq 0$ such that
\[
\left|a_{\alpha\beta}^{\frac{1}{\sigma}}(x)-a_{\alpha\beta}^{\frac{1}{\sigma}}(y)\right|\leq L|x-y| \quad\forall x,y\in\mathcal{O},\,\,\forall\alpha,\beta.
\]
Without loss of generality we will also assume that $a_{\alpha\beta}\leq 1$.
The case $\sigma= 1$ can also be considered similarly but since calculations are slightly different, it is omitted here as it is
an easy variation. The case $0<\sigma< 1$ is taken care of by Corollary \ref{cor:Comparison Theorem Ops Order Less than 1}.

We decompose the measures $\mu^{\alpha\beta}_x$ in the following way. We set $r_x^{\alpha\beta}:=a_{\alpha\beta}^{\frac{1}{\sigma}}(x)$.
\[
\hat\mu_{x}^{\alpha\beta}={\mu_{x}^{\alpha\beta}}_{|B_{r_x^{\alpha\beta}}},\quad
\tilde\mu_{x}^{\alpha\beta}={\mu_{x}^{\alpha\beta}}_{|B_{1}\setminus B_{r_x^{\alpha\beta}}},\quad
\check\mu_{x}^{\alpha\beta}={\mu_{x}^{\alpha\beta}}_{|B_1^c},
\]
and as always all measures are then extended to measures on $\mathbb{R}^d$. We claim that these measures satisfy the assumptions of
Theorem \ref{thm:Main Comparison Result-1} with $\sigma<p\leq 2$.

Assumptions $B$ and $E$ are obvious so we will only focus on Assumption A, \eqref{eq:mutilde1}, and \eqref{eq:mutilde2}. Regarding
Assumption A we note that, by an elementary calculation, if $a_{\alpha\beta}(x)>0$, then $\hat\mu_{x}^{\alpha\beta}=({T_x^{\alpha\beta}})_\# \mu$, where
\[
{T_x^{\alpha\beta}}(z)=a_{\alpha\beta}^{\frac{1}{\sigma}}(x)z,\quad d\mu(z)=\frac{1}{|z|^{d+\sigma}}\chi_{B_1}dz.
\]
These types of transformations were used in \cite{ACJ-2014, CJ-2017}.
If $a_{\alpha\beta}(x)=0$ then $\hat\mu_{x}^{\alpha\beta}=0$.
Then, if $a_{\alpha\beta}(x)>0, a_{\alpha\beta}(y)>0$, $\gamma=(T_{x}^{\alpha\beta}\times T_{y}^{\alpha\beta})_\#\mu\in 
  \textnormal{Adm}(\hat\mu_{x}^{\alpha\beta},\hat\mu_{y}^{\alpha\beta})$, and 
 \begin{align*}
    \textnormal{d}_{\mathbb{L}_p}(\hat\mu_{x}^{\alpha\beta},\hat\mu_{y}^{\alpha\beta})^p &\leq \int_{\mathbb{R}^d\times \mathbb{R}^d}|z_1-z_2|^2d\gamma(z_1,z_2)
  =\int_{\mathbb{R}^d}|T_{x}^{\alpha\beta}(z)-T_{y}^{\alpha\beta}(z)|^p\;d\mu(z)
  \\
  &
  =\int_{B_1}|a_{\alpha\beta}^{\frac{1}{\sigma}}(x)-a_{\alpha\beta}^{\frac{1}{\sigma}}(y)|^p\frac{|z|^p}{|z|^{d+\sigma}}dz\leq C|x-y|^p.
  \end{align*}
  If $a_{\alpha\beta}(x)>0$ and $\hat\mu_{y}^{\alpha\beta}=0$, then also $\gamma=({T_x^{\alpha\beta} }\times T_{y}^{\alpha\beta})_\#\mu\in 
  \textnormal{Adm}(\hat\mu_{x}^{\alpha\beta},\hat\mu_{y}^{\alpha\beta})$, where $T_{y}^{\alpha\beta}(z)=0$, and
  \begin{align*}
    \textnormal{d}_{\mathbb{L}_p}(\hat\mu_{x}^{\alpha\beta},\hat\mu_{y}^{\alpha\beta})^p &
  \leq\int_{B_1}|T_{x}^{\alpha\beta}(z)-0|^p\;d\mu(z)
  \\
  &
  =\int_{B_1}|a_{\alpha\beta}^{\frac{1}{\sigma}}(x)-a_{\alpha\beta}^{\frac{1}{\sigma}}(y)|^p\frac{|z|^p}{|z|^{d+\sigma}}dz\leq C|x-y|^p.
  \end{align*}
  Regarding \eqref{eq:mutilde1}, we see that
  \[
 \tilde \mu^{\alpha\beta}_x(B_1)=C\int_{r_x^{\alpha\beta}}^1 \frac{a_{\alpha\beta}(x)}{r^{1+\sigma}}dr=\frac{C}{\sigma}(1-a_{\alpha\beta}(x)).
 \]
 It remains to check condition \eqref{eq:mutilde2}. Suppose that $a_{\alpha\beta}(x)<a_{\alpha\beta}(y)$. Then
 \[
  \int_{\mathbb{R}^d}|z|\,d|\tilde \mu^{\alpha\beta}_x-\tilde \mu^{\alpha\beta}_y|(z)
 =
  \int_{r_x^{\alpha\beta}\leq |z|<r_y^{\alpha\beta}}|z|\,d\tilde \mu^{\alpha\beta}_x(z)
  +
 \int_{r_y^{\alpha\beta}\leq |z|<1}\frac{(a_{\alpha\beta}(y)-a_{\alpha\beta}(x))|z|}{|z|^{d+\sigma}}dz.
   \]
   We estimate each integral separately.
  \begin{align*}
    \int_{r_x^{\alpha\beta}\leq |z|<r_y^{\alpha\beta}}|z|\,d\tilde \mu^{\alpha\beta}_x(z)&=Ca_{\alpha\beta}(x) \frac{1}{r^{\sigma-1}}\bigg|_{r_x^{\alpha\beta}}^{r_y^{\alpha\beta}}
    =
    Ca_{\alpha\beta}(x)\frac{a_{\alpha\beta}^{1-\frac{1}{\sigma}}(y)-a_{\alpha\beta}^{1-\frac{1}{\sigma}}(x)}{a_{\alpha\beta}^{1-\frac{1}{\sigma}}(y)a_{\alpha\beta}^{1-\frac{1}{\sigma}}(x)}
    \\
    &
    \leq Ca_{\alpha\beta}^{\frac{2}{\sigma}-1}(x)\left(a_{\alpha\beta}^{1-\frac{1}{\sigma}}(y)-a_{\alpha\beta}^{1-\frac{1}{\sigma}}(x)\right).
     \end{align*}
     By the mean value theorem,
     \[
    a_{\alpha\beta}^{1-\frac{1}{\sigma}}(y)-a_{\alpha\beta}^{1-\frac{1}{\sigma}}(x)=c^{\sigma-2}\left(a_{\alpha\beta}^{\frac{1}{\sigma}}(y)-a_{\alpha\beta}^{\frac{1}{\sigma}}(x)\right)
    \leq
    a_{\alpha\beta}^{1-\frac{2}{\sigma}}(x)\left(a_{\alpha\beta}^{\frac{1}{\sigma}}(y)-a_{\alpha\beta}^{\frac{1}{\sigma}}(x)\right),
    \] 
   where $c$ above is some number such that $a_{\alpha\beta}^{\frac{1}{\sigma}}(x)<c<a_{\alpha\beta}^{\frac{1}{\sigma}}(y)$. Thus we obtain
   \[
     \int_{r_x^{\alpha\beta}\leq |z|<r_y^{\alpha\beta}}|z|\,d\tilde \mu^{\alpha\beta}_x(z)\leq C\left(a_{\alpha\beta}^{\frac{1}{\sigma}}(y)-a_{\alpha\beta}^{\frac{1}{\sigma}}(x)\right)\leq
     C_1|x-y|.
     \]
     For the second integral, by an elementary calculation we obtain
     \[
      \int_{r_y^{\alpha\beta}\leq |z|<1}\frac{(a_{\alpha\beta}(y)-a_{\alpha\beta}(x))|z|}{|z|^{d+\sigma}}dz
      =C(a_{\alpha\beta}(y)-a_{\alpha\beta}(x))\left(\frac{1}{a_{\alpha\beta}^{1-\frac{1}{\sigma}}(y)}-1\right).
      \]
      Now, again by the mean value theorem,
      \[
   a_{\alpha\beta}(y)-a_{\alpha\beta}(x)=c^{\sigma-1} \left(a_{\alpha\beta}^{\frac{1}{\sigma}}(y)-a_{\alpha\beta}^{\frac{1}{\sigma}}(x)\right)  
   \leq
   a_{\alpha\beta}^{1-\frac{1}{\sigma}}(y)\left(a_{\alpha\beta}^{\frac{1}{\sigma}}(y)-a_{\alpha\beta}^{\frac{1}{\sigma}}(x)\right),
   \]
    where $c$ above is some number such that $a_{\alpha\beta}^{\frac{1}{\sigma}}(x)<c<a_{\alpha\beta}^{\frac{1}{\sigma}}(y)$. Therefore, it follows that
    \[
    \int_{r_y^{\alpha\beta}\leq |z|<1}\frac{(a_{\alpha\beta}(y)-a_{\alpha\beta}(x))|z|}{|z|^{d+\sigma}}dz
     \leq 
     C\left(a_{\alpha\beta}^{\frac{1}{\sigma}}(y)-a_{\alpha\beta}^{\frac{1}{\sigma}}(x)\right)(1-a_{\alpha\beta}^{1-\frac{1}{\sigma}}(y))\leq C_2|x-y|.
     \]
     This completes the proof of \eqref{eq:mutilde2}.
\end{ex}

We remark that if we apply the estimate of Example \ref{EX:measures with kernels of order sigma<1} to the kernels
\[
K_{\alpha\beta}(x,z)=\frac{a_{\alpha\beta}(x)}{|z|^{d+\sigma}}
\]
and only use the information that the functions $a_{\alpha\beta}$ are Lipschitz continuous, we obtain
\[
 \textnormal{d}_{\mathbb{L}_p}(\hat \mu_x^{\alpha\beta},\hat \mu_y^{\alpha\beta})
    \leq C|x-y|^{\frac{1}{p}}.
    \]
Thus Example \ref{ex:fracLapl} shows that the general estimate of Example \ref{EX:measures with kernels of order sigma<1}
{ coming from Proposition \ref{prop:bound via dual form}} is not optimal for $1<p\leq 2$.

{ 
The following two examples concern Sayah's comparison results in \cite{say}. Example \ref{ex:Sayah theorem I} in particular shows Theorem \ref{thm:Main Comparison Result} implies the main comparison result in \cite{say}, in the case where the measures $\mu^{\alpha\beta}$ are all supported in a ball.} 
{ 
We note that \cite[Theorem III.1]{say} dealt with the case of $\mathcal{O}=\mathbb{R}^d$ and the non-local operators there
were slightly different so Theorem \ref{thm:Main Comparison Result} cannot be applied directly to the case considered in
\cite{say}, however our approach covers the essential difficulties of the proof of the general result of \cite{say}.}
%It might be possible that our method can be adapted to remove this compact support assumption, but we did not check this assertion. 
{ 
Example \ref{ex:Sayah theorem II} is related to an alternative assumption discussed later in the paper \cite[Section III.1, p. 1065]{say}.}

\begin{ex}\label{ex:Sayah theorem I}
 {  In this example we explain how the assumption of \cite[equation (1.3)]{say} (reproduced below in \eqref{equation:Sayah condition}) implies our Assumption A with $p=1$: assume there is a constant $C>0$ such that for any} { $\phi \in C^{0,1}(B_1)$ with $\phi(0) = 0$ and 
  Lipschitz constant $1$}, { we have the inequality
  \begin{align}\label{equation:Sayah condition}
    \int_{B_1}\phi(z)\;d\mu_x(z)-\int_{B_1}\phi(z)\;d\mu_y(z) \leq C|x-y|,\;\;\forall\;x,y\in\mathcal{O},
  \end{align}
 }{  where we assume that the measures $\mu_x\in \mathbb{L}_1(B_1)$. 
 We will show that then the measures $\hat \mu_x=\mu_x$} { satisfy Assumption $A$ with $p=1$, that is $\textnormal{d}_{\mathbb{L}_1}(\hat \mu_x,\hat \mu_y) \leq C|x-y|$ for all $x,y$.} { To this end, let $\phi$ be any function as above. The Lipschitz condition means that for every 
 $z_1,z_2$ we have $\phi(z_1)-\phi(z_2)\leq |z_1-z_2|$,} { and in particular the pair $(\phi,-\phi)$ belongs to the set $\textnormal{Adm}^1$ defined in Lemma \ref{lem:duality}. Then, Lemma \ref{lem:duality} says that
 \begin{align*} 
   \textnormal{d}_{\mathbb{L}_1}(\hat \mu_x,\hat \mu_y) \leq \int_{B_1}\phi(z)\;d\mu_x(z) + \int_{B_1}(-\phi(z))\;d\mu_y(z),
 \end{align*}
 and thus assumption \eqref{equation:Sayah condition} implies that $\textnormal{d}_{\mathbb{L}_1}(\hat \mu_x,\hat \mu_y) \leq C|x-y|$.} 
\end{ex}
  
\begin{ex}\label{ex:Sayah theorem II}  
 {  Assume that the measures $\mu_x\in \mathbb{L}_1(B_1)$ and there is a constant $C>0$ such that for all $\tau>0$ and $x,y\in\mathcal{O}$
  \begin{align}\label{equation:Sayah condition II}
    \sup \limits_{h\in C(B_1\setminus\{0\}),\|h\|_\infty\leq 1}\left \{ \int_{\tau\leq |z|<1}h(z)|z|d\mu_x(z)-\int_{\tau\leq |z|<1} h(z)|z|\;d\mu_y(z) \right \} \leq C|x-y|.
  \end{align}
  Then, the measures $\hat \mu_x=\mu_x$ satisfy Assumption A with $p=1$, that is 
  \begin{align*}
     \textnormal{d}_{\mathbb{L}_1}(\hat \mu_x,\hat \mu_y) \leq C|x-y|.
  \end{align*}
  To show it we start arguing as in Example \ref{ex:Sayah theorem I}. We take any function $\phi$ with Lipschitz constant $1$ and such that 
  $\phi(0)=0$. The Lipschitz condition means that for every $z_1,z_2$ we have $\phi(z_1)-\phi(z_2)\leq |z_1-z_2|$, and in particular $(\phi,-\phi)
  \in\textnormal{Adm}^1$. Then, Lemma \ref{lem:duality} guarantees that
  \begin{align*}
    \textnormal{d}_{\mathbb{L}_1}(\hat \mu_x,\hat \mu_y) \leq  \int_{B_1}\phi(z)\;d\mu_x(z)-\int_{B_1}\phi(z)\;\mu_y(z).	  
  \end{align*}	  
  Since $\phi$ has Lipschitz constant $1$, and $\phi(0)=0$, it follows that $|\phi(z)|\leq |z|$. In particular, if $h(z) = |z|^{-1}|\phi(z)|$ then $h$ is continuous in $B_1\setminus \{0\}$ and $|h(z)|\leq 1$ for all $z$,  and so $h(z)$ is an admissible function for the supremum in \eqref{equation:Sayah condition II}. We conclude that for every $\tau>0$ 
  \begin{align*}
     & \int_{\tau\leq |z|<1}\phi(z)\;d\mu_x(z)-\int_{\tau\leq |z|<1}\phi(z)\;\mu_y(z) \\
     & = \int_{\tau\leq |z|<1}h(z)|z|\;d\mu_x(z)-\int_{\tau\leq |z|<1}h(z)|z|\;\mu_y(z) \leq C|x-y|.
  \end{align*}
  Letting $\tau \to 0$, the integral on the left converges to $\int_{B_1(0)}\phi(z)\;d\mu_x(z)-\int_{B_1(0)}\phi(z)\;\mu_y(z)$, so
  \begin{align*}
    \textnormal{d}_{\mathbb{L}_1}(\hat \mu_x,\hat \mu_y) \leq C|x-y|.
  \end{align*}
  }
 % and the first inequality is proved. To prove the second inequality, note that if $|z|\geq 1$ and $h:\mathbb{R}^d\setminus B_1\to \mathbb{R}$ is %a continuous function such that $|h(z)|\leq 1$ then $\tilde h(z) := |z|^{-1}h(z)$ is continuous and $|\tilde h(z)|\leq 1$ in $\mathbb{R}^d\setminus %B_1$, since in this case $h(z) = |z|\tilde h(z)$, we have 
 % \begin{align*}
  %   \textnormal{d}_{\textnormal{TV}}(\check \mu_x,\check\mu_y) & = \sup \limits_{|h|\leq 1} \left \{ \int_{\mathbb{R}^d\setminus B_1}h(z)\;d
  %\mu_x(z) -\int_{\mathbb{R}^d\setminus B_1}h(z)\;d\mu_y(z) \right \}\\
	 %& \leq \sup \limits_{|\tilde h|\leq 1} \left \{ \int_{\mathbb{R}^d\setminus B_1}\tilde h(z)|z|\;d\mu_x(z) -\int_{\mathbb{R}^d\setminus B_1}\tilde %h(z)|z|\;d\mu_y(z) \right \}\\          
	 %& \leq C|x-y|,
 % \end{align*}
  %and the second inequality is proved.
\end{ex}

%%%%%%%%%%%%%%%%%%%%%%%%%%%%%%%%%%%%%%%%%%%%%%%%
%%%%%%%%%%%%%%%%%%%%%%%%%%%%%%%%%%%%%%%%%%%%%%%%
%%%%%%%%%%%%%%%%%%%%%%%%%%%%%%%%%%%%%%%%%%%%%%%%
%%%%%%%%%%%%%%%%%%%%%%%%%%%%%%%%%%%%%%%%%%%%%%%%
%%%%%%%%%%%%%%%%%%%%%%%%%%%%%%%%%%%%%%%%%%%%%%%%
%%%%%%%%%%%%%%%%%%%%%%%%%%%%%%%%%%%%%%%%%%%%%%%%
\section{Comparison Principles under additional assumptions}\label{sec:other comparison}

As in the previous section, throughout this section we consider a fixed bounded domain $\mathcal{O}\subset \mathbb{R}^d$.  In this section we prove a few comparison results for more regular viscosity sub/supersolutions. In return, we are allowed to replace Assumption A by a weaker assumption.

\bigskip

\noindent \textbf{Assumption A1}. 
Let $p\in[1,2]$. There exist $C>0$ and $s \in (0,1)$ such that
  \begin{align*}
    \textnormal{d}_{\mathbb{L}_p}(\hat \mu_{x}^{\alpha\beta},\hat \mu_{y}^{\alpha\beta})\leq C|x-y|^s,
    \quad \forall x,y\in\mathcal{O},\,\,\forall\alpha,\beta.
  \end{align*}

\begin{rem}\label{rem:distance to mu_r}
  Consider a L\'evy measure $\mu \in \mathbb{L}_p(B_1)$. For $r\in(0,1)$ we define
  \begin{align*}
    \mu_r(\cdot) := \mu (\cdot \cap B_r^c).
  \end{align*}
  Then, we have the estimate
  \begin{align*}
    \textnormal{d}_{\mathbb{L}_p}(\mu,\mu_r)^p \leq \int_{B_r}|x|^p\;d\mu(x).	  
  \end{align*}	  
  To see why this is so, simply note that among the admissible plans we have the one that sends all of the mass of $\mu$ in $B_r\setminus\{0\}$ to $0$, and leaves the rest of the mass fixed in place. To be more precise, define
  \[
 T(x)= \left\{\begin{array}{ll}0\quad \quad\,\text{for}\,\,\, x\in B_r\setminus\{0\},\\
x\qquad \text{otherwise}.
\end{array}
  \right.
  \]
  Then $\gamma=(T\times {\rm Id})_\#\mu\in \textnormal{Adm}(\mu_r,\mu)$ and
 { \[
    \textnormal{d}_{\mathbb{L}_p}(\mu,\mu_r)^p \leq \int_{\mathbb{R}^d\times\mathbb{R}^d}|x-y|^pd\gamma(x,y)
  =\int_{\mathbb{R}^d}|T(x)-x|^2\;d\mu(x)=\int_{B_r}|x|^p\;d\mu(x).
  \] }
\end{rem}

\begin{rem}\label{rem:distance to mu_r follow up}
  Estimating the distance between $\mu$ and $\mu_r$ is of interest to us since it can be used to bound the difference between the operators
 {  \begin{align*}
    L_{\hat\mu}(u,x) := \int_{B_1}\;\delta u(x,z) \;d\hat\mu(z) \,\,\textnormal{ and }\,\, L_{\hat\mu_r}(u,x) := 
    \int_{B_1\cap B_r^c} \delta u(x,z) \;d\hat\mu_r(z).
  \end{align*}	 }
  The operator on the right can be classically evaluated for any continuous function, while the one on the left in general is not. Being able to estimate the difference between them will be an important step in the proof of Theorem \ref{thm:comparison principle with C1 subsolution}, removing the need for the use of the sup/inf-convolutions (as mentioned in Remark \ref{rem:comparison sec no need for sub/inf convolutions}).   
\end{rem}

\begin{thm}\label{thm:Comparison Principle 2}
Let Assumptions A1 and B-E be satisfied.   Let $u$ be a viscosity subsolution and $v$ be a viscosity supersolution of \eqref{eqn:intro equation form} and let $u(x)\leq v(x)$ for all $x\not\in \mathcal{O}$. If either $u$ or $v$ is in { $C^{0,r}(\mathcal{O})$} {  for some $r\in(0,1)$}, and we have $1-\frac{r}{p}<s$ (where $s$ is from Assumption A1), then
  \begin{align*}
    u(x)\leq v(x)\;\;\forall\;x\in\mathcal{O}.	
  \end{align*}
\end{thm}

\begin{proof}
  The proof follows the lines of the proof of Theorem \ref{thm:Main Comparison Result}. The only difference is that when either $u$ or $v$ is 
  { $C^{0,r}$}, instead of \eqref{eq:lim0} we now have (\ref{eq:calphaeps}), i.e. 
  \begin{align*}
   \limsup_{\kappa\to 0} \frac{|x_{\varepsilon}-y_{\varepsilon}|^{p-r}}{\varepsilon}\leq C
  \end{align*}
  and in this case, following the original proof we obtain
  \begin{eqnarray*}
    \lambda \ell    \leq  2\delta +\frac{C}{\varepsilon}|x_{\varepsilon}-y_{\varepsilon}|^{sp}+C \theta\big(|x_\varepsilon-y_\varepsilon|) +\rho_\varepsilon(\delta).
  \end{eqnarray*} 
  which produces a contradiction after taking $\lim_{\varepsilon \to 0}\lim_{\delta\to 0}\limsup_{\kappa\to 0}$ if $s p>p-r$, which is the case precisely when $1-\frac{r}{p}<s$.
\end{proof}

{  The previous theorem does not cover the limiting situation where $s=1-\frac{r}{p}$, however, with extra work one can show that if $u$ or $v$ is of class $C^1$ then we can choose $s=1-\tfrac{1}{p}$ and we still have comparison.} The proof is different from that of Theorem \ref{thm:Main Comparison Result} since we do not use the sup/inf-convolutions. 
  
 \begin{thm}\label{thm:comparison principle with C1 subsolution}
  Let Assumptions B-F hold, and Assumption A1 hold with $p>1$ and $s=1-\tfrac{1}{p}$. Suppose that $u$ and $v$ are respectively a viscosity subsolution and viscosity supersolution of \eqref{eqn:intro equation form} and $u(x)\leq v(x)$ for all $x\not\in \mathcal{O}$. If either $u$ or $v$ is in $C^1(\mathcal{O})$ then 
  \begin{align*}
    u(x)\leq v(x)\;\;\forall\;x\in\mathcal{O}.
  \end{align*}
\end{thm}

\begin{rem}
  If we allowed $C^1$ functions to be test functions in the case $p=1$ then Theorem \ref{thm:comparison principle with C1 subsolution} would trivially hold for $p=1$ without the need for Assumptions A, B and the continuity of the coefficients, since then either $u$ or $v$ would be a classical sub/supersolution of \eqref{eqn:intro equation form} and could thus be used as a test function.
\end{rem}

\begin{proof}
  Without loss of generality, let us say that $u\in C^{1}(\mathcal{O})$. As before we argue by contradiction, in which case there is some $\ell>0$ such that
  \begin{align*}
    \sup \limits_{x\in\mathbb{R}^d} \{u(x)-v(x)\} = \ell.
  \end{align*}	  	
  
  Step 1. (Doubling of variables and perturbation)
  
  Let $K\subset\mathcal{O}$ be a compact neighborhood of the set of maximum points of $u-v$ in $\mathcal{O}$. There exists a sequence of $C^{3}(\mathbb R^d)\cap \textnormal{BUC}(\mathbb R^d)$ functions $\{\phi_n\}_{n}$ each of which has second and third derivatives bounded in $\mathbb{R}^d$ and such that $|u-\phi_n|\to 0$ uniformly in $\mathbb{R}^d$ and 
  \begin{align*}
    \lim \limits_{n\to \infty}\sup \limits_{K}| D u- D \phi_n| = 0.
  \end{align*}
  Now, let $(x_{n,\varepsilon},y_{n,\varepsilon})\in \mathbb R^d\times \mathbb R^d$ be a global maximum point of $w^{n,\varepsilon}$ over $\mathbb{R}^d\times\mathbb{R}^d$, where
  \begin{equation*}
    w^{n,\varepsilon}(x,y):=\big(  u(x)-\phi_n(x)\big)-\big( v(y)- \phi_n(y)\big)-\frac{\psi_\kappa(x-y)}{\varepsilon}.
  \end{equation*}
  Similarly to Lemma \ref{lem:properties of w}, one can show that for any $n$
   \begin{align*}
    \limsup_{\varepsilon\to 0}\limsup_{\kappa\to 0} \frac{|x_{n,\varepsilon}-y_{n,\varepsilon}|^p}{\varepsilon}=0, 
  \end{align*}  
  \begin{align*}
  \limsup_{\varepsilon\to 0}\limsup_{\kappa\to 0}\big( u(x_{n,\varepsilon})- v(y_{n,\varepsilon})\big)=\ell.
  \end{align*} 
  Observe that $u$ is touched from above at $x_{n,\varepsilon}$ by
  \begin{align*}
    \bar \phi(x) := u(x_{n,\varepsilon}) - \phi_n(x_{n,\varepsilon}) + \phi_n(x)+ \frac{1}{\varepsilon}\left( \psi_{\kappa}(x-y_{n,\varepsilon}) - \psi_{\kappa}(x_{n,\varepsilon}-y_{n,\varepsilon}) \right ),
  \end{align*}
  while  $v$ is touched from below at $y_{n,\varepsilon}$ by
  \begin{align*}
    \underline \phi(y) := v(y_{n,\varepsilon}) - \phi_n(y_{n,\varepsilon}) + \phi_n(y) - \frac{1}{\varepsilon} \left ( \psi_{\kappa}(x_{n,\varepsilon}-y) -\psi_{\kappa}(x_{n,\varepsilon}-y_{n,\varepsilon}) \right ).	  
  \end{align*}	  
  Since $u$ is $C^1$ this means first that 
  \begin{align*}
    D u(x_{n,\varepsilon})-D\phi_n(x_{n,\varepsilon}) = p(\kappa+|x_{n,\varepsilon}-y_{n,\varepsilon}|^{2})^{\frac{p}{2}-1}\frac{x_{n,\varepsilon}-y_{n,\varepsilon}}{\varepsilon}.
  \end{align*}
  There is some small $c>0$ such that $B_c(x_{n,\varepsilon})\cup B_c(y_{n,\varepsilon})\subset K$
  if $\varepsilon$ and $\kappa$ are sufficiently small. Therefore, 
  \begin{equation}\label{eqn:second comparison theorem w epsilon limit as epsilon to zero}
    \limsup \limits_{\varepsilon\to 0}\limsup_{\kappa\to 0}\frac{|x_{n,\varepsilon}-y_{n,\varepsilon}|^{p-1}}{\varepsilon} = o_\frac{1}{n}(1).
  \end{equation}
  On the other hand, since $u$ is a viscosity subsolution and $v$ a viscosity supersolution, for any $0<r<1$ we have
  \begin{align*}
    I(u_r,x_{n,\varepsilon}) \leq 0 \textnormal{ and } I(v_r,y_{n,\varepsilon}) \geq 0,
  \end{align*}
  where (recall Definition \ref{def:viscosity solution})
  \begin{align*}
    u_r(x) := \left \{ \begin{array}{rr}
      \bar \phi(x) & \textnormal{ in } B_r(x_{n,\varepsilon})\\
      u(x) & \textnormal{ in } B_r^c(x_{n,\varepsilon})	   	
    \end{array}\right.\;\;\textnormal{ and } \;\;\;\;  v_r(x) := \left \{ \begin{array}{rr}
      \underline \phi(x) & \textnormal{ in } B_r(y_{n,\varepsilon})\\
      v(x) & \textnormal{ in } B_r^c(y_{n,\varepsilon}).	   	
    \end{array}\right. 
  \end{align*}
  
  Step 2. (Equation structure, main term)

  Using that $I(\cdot,x)$ has the inf-sup representation in \eqref{eqn:intro equation form}, it follows that
  \begin{align}   
    I(v_r,y_{n,\varepsilon})-I(u_r,x_{n,\varepsilon}) & \leq  \sup \limits_{\alpha,\beta} \Big \{ \hat L^{\alpha\beta}(u_r,x_{n,\varepsilon}) - \hat L^{\alpha\beta}(v_r,y_{n,\varepsilon}) \Big \} \notag \\
      & \;\;\;\;+ \sup \limits_{\alpha,\beta} \Big \{  \check L^{\alpha\beta}(u_r,x_{n,\varepsilon}) - \check L^{\alpha\beta}(v_r,y_{n,\varepsilon}) \Big \} \notag \\
      & \;\;\;\;+ \sup \limits_{\alpha,\beta} \Big \{ c_{\alpha\beta}(y_{n,\varepsilon})v(y_{n,\varepsilon}) - c_{\alpha\beta}(x_{n,\varepsilon})u(x_{n,\varepsilon}) \Big \} \notag \\
      & \;\;\;\;+ \sup \limits_{\alpha,\beta} \Big \{ f_{\alpha\beta}(y_{n,\varepsilon}) - f_{\alpha\beta}(x_{n,\varepsilon}) \Big \}.\label{eqn:C1 comparison structure}	  	     
  \end{align}
  Let us bound each of the terms on the right hand side of \eqref{eqn:C1 comparison structure}. As before, the most delicate term is the first one. Fix $\alpha$ and $\beta$, we note that
  \begin{align*}
    \hat L^{\alpha\beta}(u_r,x_{n,\varepsilon}) &  =  \int_{B_r}\delta u_r(x_{n,\varepsilon},x)  \;d\hat \hat \mu^{\alpha\beta}_{x_{n,\varepsilon}}(x)  + \int_{B_r^c}\delta u(x_{n,\varepsilon},x)  \;d\hat \mu^{\alpha\beta}_{x_{n,\varepsilon}}(x),\\
     \hat L^{\alpha\beta}(v_r,y_{n,\varepsilon}) &   =  \int_{B_r}\delta v_r(y_{n,\varepsilon},y)  \;d\hat \mu^{\alpha\beta}_{y_{n,\varepsilon}}(y) + \int_{B_r^c}\delta v(y_{n,\varepsilon},y)  \;d\hat \mu^{\alpha\beta}_{y_{n,\varepsilon}}(y).
  \end{align*}
  
  Let us choose $\gamma_{r} \in \textnormal{Adm}(\hat \mu^{\alpha\beta}_{x_{n,\varepsilon},r},\hat \mu^{\alpha\beta}_{y_{n,\varepsilon},r})$ (using the notation introduced in Remark \ref{rem:distance to mu_r}) which minimizes the $p$-cost. Denote
\[
A_r:=(\{0\}\times(B_1\setminus B_r))\cup ((B_1\setminus B_r)\times\{0\})\cup ((B_1\setminus B_r)\times(B_1\setminus B_r)).
\]
Since $\delta u(x_{n,\varepsilon},0)=\delta v(y_{n,\varepsilon},0)=0$ and
\begin{equation}\label{eq;gammar}
\gamma_{r}\left((((B_1\setminus B_r)\cup\{0\})^c\times \mathbb{R}^d)\cup(\mathbb{R}^d\times
((B_1\setminus B_r)\cup\{0\})^c)\right)=0,
\end{equation}
we have
\begin{align*}
&\int_{B_r^c}\delta u(x_{n,\varepsilon},x)  \;d\hat \mu^{\alpha\beta}_{x_{n,\varepsilon}}(x)
=\int_{B_r^c}\delta u(x_{n,\varepsilon},x)  \;d\hat \mu^{\alpha\beta}_{x_{n,\varepsilon},r}(x)\\
&=\int_{\mathbb{R}^d\times\mathbb{R}^d}\delta u(x_{n,\varepsilon},x)  \;d\gamma_r(x,y)
=\int_{A_r}\delta u(x_{n,\varepsilon},x)  \;d\gamma_r(x,y).
\end{align*} 
Similarly,
\[
\int_{B_r^c}\delta v(y_{n,\varepsilon},y)  \;d\hat \mu^{\alpha\beta}_{y_{n,\varepsilon}}(y)
=\int_{A_r}\delta v(y_{n,\varepsilon},y)  \;d\gamma_r(x,y).
\]
Therefore we obtain
  \begin{align*}
    & \hat L^{\alpha\beta}(u_r,x_{n,\varepsilon}) - \hat L^{\alpha\beta}(v_r,y_{n,\varepsilon}) \\
    & = \int_{B_r}\delta u_r(x_{n,\varepsilon},x)  \;d\hat \mu^{\alpha\beta}_{x_{n,\varepsilon}}(x) - \int_{B_r}\delta v_r(y_{n,\varepsilon},y)  \;d\hat \mu^{\alpha\beta}_{y_{n,\varepsilon}}(y)\\
    & \;\;\;\;+\int_{A_r} [\delta u(x_{n,\varepsilon},x)-\delta v(y_{n,\varepsilon},y)]\;d\gamma_r(x,y).
  \end{align*}  
  Using that $(x_{n,\varepsilon},y_{n,\varepsilon})$ is a maximum point of $w^{n,\varepsilon}$, we have the following pointwise bound for pairs $(x,y) \in A_r$
  \begin{align*}
    \delta u(x_{n,\varepsilon},x)-\delta v (y_{n,\varepsilon},y) \leq \frac{C}{\varepsilon}|x-y|^p+\delta \phi(x_{n,\varepsilon},x)-\delta \phi(y_{n,\varepsilon},y),
  \end{align*}
  It thus follows (again using $\delta \phi(x_{n,\varepsilon},0)=\delta \phi(y_{n,\varepsilon},0)=0$ and \eqref{eq;gammar}) that 
  \begin{align*}
    & \int_{A_r} [\delta u(x_{n,\varepsilon},x)-\delta v(y_{n,\varepsilon},y)]\;d\gamma_r(x,y) \\
    & \leq \frac{C}{\varepsilon}\int_{A_r} |x-y|^{p}\;d\gamma_r(x,y) \\
    & \;\;\;\;+\int_{B_r^c}\delta \phi(x_{n,\varepsilon},x)  \;d\hat \mu^{\alpha\beta}_{x_{n,\varepsilon}}(x)-\int_{B_r^c}\delta \phi(y_{n,\varepsilon},y)  \;d\hat \mu^{\alpha\beta}_{y_{n,\varepsilon}}(y), 	
  \end{align*}	  
  and since $\gamma_r$ is the optimizer in $\textnormal{Adm}(\hat \mu_{x_{n,\varepsilon}},\hat \mu_{y_{n,\varepsilon}})$, 
  \begin{align*}
    & \int_{A_r} [\delta u(x_{n,\varepsilon},x)-\delta v(y_{n,\varepsilon},y)]\;d\gamma_r(x,y) \\
    & \leq \frac{C}{\varepsilon}\textnormal{d}_{\mathbb{L}_p}(\hat \mu_{x_{n,\varepsilon},r}^{\alpha\beta},\hat \mu_{y_{n,\varepsilon},r}^{\alpha\beta} )^p +\int_{B_r^c}\delta \phi(x_{n,\varepsilon},x)  \;d\hat \mu^{\alpha\beta}_{x_{n,\varepsilon}}(x)-\int_{B_r^c}\delta \phi(y_{n,\varepsilon},y)  \;d\hat \mu^{\alpha\beta}_{y_{n,\varepsilon}}(y). 	
  \end{align*}
  As for the integrals over $B_r(0)$, note that 
  \begin{align*}
    & \int_{B_r}\delta u_r(x_{n,\varepsilon},x)  \;d\hat \mu^{\alpha\beta}_{x_{n,\varepsilon}}(x) - \int_{B_r}\delta v_r(y_{n,\varepsilon},y)  \;d\hat \mu^{\alpha\beta}_{y_{n,\varepsilon}}(y) \\
    & = \int_{B_r} [\delta \phi_n(x_{n,\varepsilon},x) + \frac{1}{\varepsilon} \delta \psi_\kappa(\cdot-y_{n,\varepsilon})(x_{n,\varepsilon},x)]\;d\hat \mu^{\alpha\beta}_{x_{n,\varepsilon}}(x)\\
    & -\int_{B_r} [\delta \phi_n(y_{n,\varepsilon},y) - \frac{1}{\varepsilon} \delta \psi_\kappa(x_{n,\varepsilon}-\cdot)(y_{n,\varepsilon},y)]\;d\hat \mu^{\alpha\beta}_{y_{n,\varepsilon}}(x).	 	
  \end{align*}
  Putting the last inequality and last equality together, we have
  \begin{align*}
    & \hat L^{\alpha\beta}(u_r,x_{n,\varepsilon}) - \hat L^{\alpha\beta}(v_r,y_{n,\varepsilon}) \\
    & \leq \frac{C}{\varepsilon}\textnormal{d}_{\mathbb{L}_p}(\hat \mu^{\alpha\beta}_{x_{n,\varepsilon},r},\hat \mu^{\alpha\beta}_{y_{n,\varepsilon},r})^p + \hat L^{\alpha\beta}(\phi_n,x_{n,\varepsilon})-\hat L^{\alpha\beta}(\phi_n,y_{n,\varepsilon}) +
    \frac{C}{\varepsilon}\theta(r),
  \end{align*}
  where we used Assumption E and Lemma \ref{lem:psigamma}.

  Next, we use Remark \ref{rem:distance to mu_r} to get $\textnormal{d}_{\mathbb{L}_p}(\hat \mu_{x_{n,\varepsilon},r}^{\alpha\beta},\hat \mu_{y_{n,\varepsilon},r}^{\alpha\beta} )^p \leq \textnormal{d}_{\mathbb{L}_p}(\hat \mu_{x_{n,\varepsilon}}^{\alpha\beta},\hat \mu_{y_{n,\varepsilon}}^{\alpha\beta} )^p+\rho(r)$, where $\rho(r)\to 0$ as $r\to 0$. Then using Assumption A1 (recall $s=1-1/p$) it follows that 
 {  \begin{align*}
    \textnormal{d}_{\mathbb{L}_p}(\hat \mu_{x_{n,\varepsilon},r}^{\alpha\beta},\hat \mu_{y_{n,\varepsilon},r}^{\alpha\beta} )^p \leq C|x_{n,\varepsilon}-y_{n,\varepsilon}|^{p-1}+\rho(r).
  \end{align*}}
  Thus 
  { \begin{align*}
    \hat L^{\alpha\beta}(u_r,x_{n,\varepsilon}) - \hat L^{\alpha\beta}(v_r,y_{n,\varepsilon})  &\leq \frac{C}{\varepsilon}|x_{n,\varepsilon}-y_{n,\varepsilon}|^{p-1} 
    \\
    &+ \hat L^{\alpha\beta}(\phi_n,x_{n,\varepsilon})-\hat L^{\alpha\beta}(\phi_n,y_{n,\varepsilon}) + 
 \frac{C}{\varepsilon}(\theta(r) +\rho(r)).
 \end{align*}}
  Letting $r\to 0$, it follows that for every $\kappa$, $n$, and $\varepsilon$, 
  \begin{align}
    & \limsup\limits_{r \to 0} \sup \limits_{\alpha,\beta} \Big \{ \hat L^{\alpha\beta}(u_r,x_{n,\varepsilon}) - \hat L^{\alpha\beta}(v_r,y_{n,\varepsilon})\Big \} \notag \\
    & \leq \frac{C}{\varepsilon}|x_{n,\varepsilon}-y_{n,\varepsilon}|^{p-1} + \sup \limits_{\alpha,\beta} \big \{ \hat L^{\alpha\beta}(\phi_n,x_{n,\varepsilon})-\hat L^{\alpha\beta}(\phi_n,y_{n,\varepsilon}) \big \}.  \label{eq:contradiction307}
  \end{align}    

  Step 3. (Equation structure, remaining terms)

  For any $r \in (0,1)$ and any $\alpha,\beta$ we have
  \begin{align*}
    & \check L^{\alpha\beta}(u_r,x_{n,\varepsilon})-\check L^{\alpha\beta}(v_r,y_{n,\varepsilon}) = \check L^{\alpha\beta}(u,x_{n,\varepsilon})-\check L^{\alpha\beta}(v,y_{n,\varepsilon}).
  \end{align*}   
  Furthermore, arguing as in Step 3 of the proof of Theorem \ref{thm:Main Comparison Result}
  \begin{align}
    c_{\alpha\beta}(x_{n,\varepsilon})u(x_{n,\varepsilon})-c_{\alpha\beta}(y_{n,\varepsilon})v(y_{n,\varepsilon})
  &  \geq  \lambda(u(x_{n,\varepsilon})-v({y_{n,\varepsilon}}))-\|v\|_{L^{\infty}}\theta(|x_{n,\varepsilon}-y_{n,\varepsilon}|)\label{eq:contradiction309},\\
 | f_{\alpha\beta}(x_{n,\varepsilon})-f_{\alpha\beta}(y_{n,\varepsilon})| & \leq \theta(|x_{n,\varepsilon}-y_{n,\varepsilon}|).\label{eq:contradiction310}
  \end{align}
  Going back to \eqref{eqn:C1 comparison structure} and combining it with \eqref{eq:contradiction307}-\eqref{eq:contradiction310}, it follows that for any $\kappa, n,$ and $\varepsilon$
  \begin{align}\label{eq:vsusupa1}
    \limsup \limits_{r\to 0} \big \{ I(v_r,y_{n,\varepsilon})-I(u_r,x_{n,\varepsilon}) \big \} & \leq \frac{C}{\varepsilon}|x_{n,\varepsilon}-y_{n,\varepsilon}|^{p-1} + (1+\|v\|_{L^\infty}) \theta(|x_{n,\varepsilon}-y_{n,\varepsilon})\nonumber\\
    & \;\;\;\;\;+\lambda (v(y_{n,\varepsilon})-u(x_{n,\varepsilon}))\nonumber\\
    & \;\;\;\;\;+ \sup \limits_{\alpha,\beta} \big \{ \check L^{\alpha\beta}(u,x_{n,\varepsilon})-\check L^{\alpha\beta}(v,y_{n,\varepsilon}) \big \}\nonumber\\
    & \;\;\;\;\;+ \sup \limits_{\alpha,\beta} \big \{ \hat L^{\alpha\beta}(\phi_n,x_{n,\varepsilon})-\hat L^{\alpha\beta}(\phi_n,y_{n,\varepsilon}) \big \}.	
  \end{align}
  Let us handle the last two terms on the right. Using Assumption B and the fact that $\phi_n\in C^{0,1}(\mathbb R^d)$, we have for any $\alpha$ and $\beta$,
  \begin{align*}
    & \check L^{\alpha\beta}(u,x_{n,\varepsilon}) - \check L^{\alpha\beta}(v,y_{n,\varepsilon})\notag\\
    & = \int_{B_1^c}\left[\delta u(x_{n,\varepsilon},x) - \delta v(y_{n,\varepsilon},x)\right]d\check \mu_{x_{n,\varepsilon}}^{\alpha\beta}(x)
+ \int_{B_1^c}\delta v(y_{n,\varepsilon},y)d\left[\check \mu_{x_{n,\varepsilon}}^{\alpha\beta}(y) - \check \mu_{y_{n,\varepsilon}}^{\alpha\beta}(y)\right]\notag\\
    & \leq \int_{B_1^c}\left(\delta \phi_{n}(x_{n,\varepsilon},x)-\delta \phi_{n}(y_{n,\varepsilon},x)\right)d\check \mu_{x_{n,\varepsilon}}^{\alpha\beta}(x)+2\|v\|_{L^\infty}\textnormal{d}_{\textnormal{TV}}(\check \mu_{x_{n,\varepsilon}}^{\alpha\beta},\check \mu_{y_{n,\varepsilon}}^{\alpha\beta})\notag\\
    & \leq  C(n)|x_{n,\varepsilon}-y_{n,\varepsilon}| + 2\|v\|_{L^\infty}\theta(|x_{n,\varepsilon}-y_{n,\varepsilon}|).
 \end{align*}
  Then, we have
  \begin{align}
    \sup \limits_{\alpha,\beta} \big \{ \check L^{\alpha\beta}(u,x_{n,\varepsilon}) - \check L^{\alpha\beta}(v,y_{n,\varepsilon}) \big \}  \leq  C(n)|x_{n,\varepsilon}-y_{n,\varepsilon}| + 2\|v\|_{L^\infty}\theta(|x_{n,\varepsilon}-y_{n,\varepsilon}|). \label{eq:aa112}
  \end{align}
  For the other remaining term, we note that 
  \begin{align*}
   \hat L^{\alpha\beta}(\phi_n,x_{n,\varepsilon})-\hat L^{\alpha\beta}(\phi_n,y_{n,\varepsilon}) & =\int_{B_1} [\delta\phi_n(x_{n,\varepsilon},x)-\delta\phi_n(y_{n,\varepsilon},x)]\;d\hat \mu_{x_{n,\varepsilon}}^{\alpha\beta}(x)\\
   & \;\;\;\; +\int_{B_1}\delta\phi_n(y_{n,\varepsilon},x)d\hat \mu_{x_{n,\varepsilon}}^{\alpha\beta}(x)-\int_{B_1}\delta\phi_n(y_{n,\varepsilon},y)d\hat \mu_{y_{n,\varepsilon}}^{\alpha\beta}(y).
  \end{align*}
  Since the third derivatives of $\phi_n$ are bounded, we have
  \begin{align*}
  &\left|\int_{B_1} [\delta\phi_n(x_{n,\varepsilon},x)-\delta\phi_n(y_{n,\varepsilon},x)]\;d\hat \mu_{x_{n,\varepsilon}}^{\alpha\beta}(x)\right|
  \\
  &\quad\quad\leq
     \int_{B_1}\int_0^1|\langle (D^2\phi_n(x_{n,\varepsilon}+sx)-D^2\phi_n(y_{n,\varepsilon}+sx))x,x\rangle(1-s)|ds\;d\hat \mu_{x_{n,\varepsilon}}^{\alpha\beta}(x)
     \\
     &\quad\quad\quad\quad
     \leq C(n)|x_{n,\varepsilon}-y_{n,\varepsilon}|.
  \end{align*}		  
  The remaining integrals are estimated as follows. For $\tau \in (0,1)$, let $\eta_\tau$ be a smooth function such that $0\leq \eta_\tau\leq 1$,\; $\eta_{\tau} \equiv 1$ in $B_\tau(0)$ and  $\eta_\tau \equiv 0$ outside of $B_{2\tau}$. Then, we may write
  \begin{align*}
   & \int_{B_1}\delta\phi_n(y_{n,\varepsilon},x)d\hat \mu_{x_{n,\varepsilon}}^{\alpha\beta}(x)  = \int_{B_1}(1-\eta_\tau(x))\delta\phi_n(y_{n,\varepsilon},x)d\hat \mu_{x_{n,\varepsilon}}^{\alpha\beta}(x) + \int_{B_1}\eta_\tau(x)\delta\phi_n(y_{n,\varepsilon},x)d\hat \mu_{x_{n,\varepsilon}}^{\alpha\beta}(x),\\
   & \int_{B_1}\delta\phi_n(y_{n,\varepsilon},y)d\hat \mu_{y_{n,\varepsilon}}^{\alpha\beta}(y)  = \int_{B_1}(1-\eta_\tau(y))\delta\phi_n(y_{n,\varepsilon},y)d\hat \mu_{y_{n,\varepsilon}}^{\alpha\beta}(y) + \int_{B_1}\eta_\tau(y)\delta\phi_n(y_{n,\varepsilon},y)d\hat \mu_{y_{n,\varepsilon}}^{\alpha\beta}(y). 
  \end{align*}
  Applying Proposition \ref{prop:bound for distances of restricted measures}, together with \eqref{eqn:uniform integrability Levy measures1}, and using again Assumption A1, it is straightforward to observe that for fixed $n$ and $\tau$,
  { \begin{align*}
    \lim \limits_{\varepsilon\to 0} \sup \limits_{\alpha,\beta}\left | \int_{B_1} (1-\eta_{\tau}(x)) \delta\phi_n(y_{n,\varepsilon},x)d\hat \mu_{x_{n,\varepsilon}}^{\alpha\beta}(x) - \int_{B_1} (1-\eta_{\tau}(y)) \delta\phi_n(y_{n,\varepsilon},y)d\hat \mu_{y_{n,\varepsilon}}^{\alpha\beta}(y) \right | = 0.
  \end{align*}}
  On the other hand, since each $\phi_n$ is $C^3$, we have $|\delta \phi_n(y_{n,\varepsilon},x)| \leq C_n|x|^{2}$ for all $x \in B_1$. Therefore
 {  \begin{align*}
    \left | \int_{B_1} \eta_{\tau}(x) \delta\phi_n(y_{n,\varepsilon},x)d\hat \mu_{x_{n,\varepsilon}}^{\alpha\beta}(x)\right | & \leq \int_{B_1} |\eta_{\tau}(x) \delta\phi_n(y_{n,\varepsilon},x)|d\hat \mu_{x_{n,\varepsilon}}^{\alpha\beta}(x),\\
 & \leq C_n\int_{B_\tau} |x|^{2}d\hat \mu_{x_{n,\varepsilon}}^{\alpha\beta}(x),\\
    \left | \int_{B_1} \eta_{\tau}(y) \delta\phi_n(y_{n,\varepsilon},y)d\hat \mu_{y_{n,\varepsilon}}^{\alpha\beta}(y)\right | & \leq C_n\int_{B_\tau} |y|^{2}d\hat \mu_{y_{n,\varepsilon}}^{\alpha\beta}(y), 
  \end{align*}}
  and we have
 {  \begin{align*}
    \limsup \limits_{\tau \to 0}\sup \limits_{\alpha,\beta} \int_{B_\tau} |x|^{2}d\hat \mu_{x_{n,\varepsilon}}^{\alpha\beta}(x) = \limsup \limits_{\tau \to 0}\sup \limits_{\alpha,\beta} \int_{B_\tau} |y|^{2}d\hat \mu_{y_{n,\varepsilon}}^{\alpha\beta}(y) = 0. 
  \end{align*}}
  Gathering these estimates, we conclude that for every $n$,
  \begin{align}
    & \lim \limits_{\varepsilon \to 0} \sup\limits_{\alpha,\beta} \big |\hat L^{\alpha\beta}(\phi_n,x_{n,\varepsilon})-\hat L^{\alpha\beta}(\phi_n,y_{n,\varepsilon}) \big | =0.  \label{eq:aa113}
  \end{align}
  
  Step 4. (Using the subsolution and supersolution property)
  
  Using Definition \ref{def:viscosity solution ver1} we obtain from \eqref{eq:vsusupa1} that
  \begin{align}\label{eq:aa111}
    \lambda (u(x_{n,\varepsilon})-v(y_{n,\varepsilon})) & \leq \frac{C}{\varepsilon}|x_{n,\varepsilon}-y_{n,\varepsilon}|^{p-1} + (1+\|v\|_{L^\infty}) \theta(|x_{n,\varepsilon}-y_{n,\varepsilon}|)\notag \\
    & \;\;\;\;\;+ \sup \limits_{\alpha\beta} \big \{ \check L^{\alpha\beta}(u,x_{n,\varepsilon})-\check L^{\alpha\beta}(v,y_{n,\varepsilon}) \big \}\notag\\	
    & \;\;\;\;\;+ \sup \limits_{\alpha\beta} \big \{ \hat L^{\alpha\beta}(\phi_n,x_{n,\varepsilon})-\hat L^{\alpha\beta}(\phi_n,y_{n,\varepsilon}) \big \}.	
  \end{align}
  Letting {  $\kappa\to 0$ first and then $\varepsilon\to 0$} in \eqref{eq:aa111} and using \eqref{eqn:second comparison theorem w epsilon limit as epsilon to zero}, \eqref{eq:aa112} and \eqref{eq:aa113}, we now obtain
  \begin{align*}
    0<\lambda l\leq o_{\frac{1}{n}}(1),
  \end{align*}
  which gives a contradiction.
\end{proof}

  The following is an example of measures satisfying the assumptions of Theorem \ref{thm:comparison principle with C1 subsolution}.
\begin{ex}\label{ex:uniform}
Let $d\mu_x^{\alpha\beta}(z):=K_{\alpha\beta}(x,z)dz$ be such that \eqref{eq:int111222} and \eqref{eq:K reg} hold for some $p\in(1,2]$ and $\gamma\in(0,1]$, where $K$ satisfies \eqref{eq:3.122aa}. Using \eqref{eq:3.11aa}, Assumption A1 is satisfied with { $s=1-\frac{1}{p}$} if $\gamma=p-1$. 
\end{ex}

\noindent\textbf{Assumption F.}
There is $\sigma \in (1,2)$ and positive constants $\bar\lambda<\bar\Lambda$ such that the measures $\mu_x^{\alpha\beta}$ are all of the form 
  $d\mu_x^{\alpha\beta}(z)=K_{\alpha\beta}(x,z)dz$, with $K_{\alpha\beta}(x,z)=K_{\alpha\beta}(x,-z)$, and
  \begin{align}\label{eqn:kernels alpha stable}
    \frac{\bar\lambda}{|z|^{d+\sigma}}\leq K_{\alpha\beta}(x,z)\leq \frac{\bar\Lambda}{|z|^{d+\sigma}}
  \end{align}
  and 
  \begin{align*}
  | K_{\alpha\beta}(x,z)-K_{\alpha\beta}(y,z)|\leq  \frac{\bar\Lambda|x-y|^{\gamma}}{|z|^{d+\sigma}}.
  \end{align*}

\begin{cor}\label{thm:comparison with one solution elliptic case}
 Let the measures $\mu_x^{\alpha\beta}$ be as above. Assume that Assumptions C, D and F hold with some $\sigma \in (1,2)$ and $\gamma>\sigma-1$. Then, given a viscosity solution $u$ and a viscosity subsolution (respectively, supersolution) $v$ of \eqref{eqn:intro equation form} such that $v\leq u$ (respectively, $u\leq v$) in 
 $\mathcal{O}^c$, we have
  \begin{align*}
    v\leq u\textnormal{ in } \mathcal{O}\,\,\,(\textnormal{respectively}, u\leq v\textnormal{ in } \mathcal{O}).
  \end{align*}
\end{cor}

\begin{proof}
  The proof is an immediate application of Theorem \ref{thm:comparison principle with C1 subsolution} with $p=1+\gamma$, the computation in Example \ref{EX:measures with kernels of order sigma<1} and the fact that $u\in C^1(\mathcal{O})$ by Theorem 4.1 in \cite{Kri}.
\end{proof}

\begin{rem}\label{rem:lambda0}
The comparison result of Corollary \ref{thm:comparison with one solution elliptic case} can be extended to the case $\sigma=1$
if we use Theorem \ref{thm:Comparison Principle 2} instead of Corollary \ref{thm:comparison with one solution elliptic case}. The result of
Corollary \ref{thm:comparison with one solution elliptic case} can be also extended to the case $\lambda=0$, see Theorem 4.1 in \cite{MS}.
\end{rem}

It is worth noting that in \cite{MS}, two of the authors obtained uniqueness results under similar assumptions to those of Corollaries \ref{cor:Comparison Theorem Ops Order Less than 1} and \ref{thm:comparison with one solution elliptic case} including Lipschitz-type assumption on the continuity of the kernels with respect to $x$. However, uniqueness results in \cite{MS} cover only $\sigma$  in the range $(0,3/2)$, whereas the combination of Corollaries
\ref{cor:Comparison Theorem Ops Order Less than 1} and \ref{thm:comparison with one solution elliptic case} (see also Remark \ref{rem:lambda0})
covers all $\sigma$ up to 2.

\begin{rem}
The assumption \eqref{eqn:kernels alpha stable} used in Corollary \ref{thm:comparison with one solution elliptic case} can be relaxed a great deal. This assumption was used merely in order to guarantee that the viscosity solution is $C^{1+\alpha}$ in the interior. Indeed, interior $C^{\alpha}$ and $C^{1+\alpha}$ regularity estimates are now available for non-local equations for a far larger class of kernels, including those $K(x,z)$ which may not be symmetric in $z$ or which vanish even for large sets of directions of $z$. See works of Schwab and Silvestre \cite[Section 8]{SS} and Kriventsov \cite{Kri}.
\end{rem}

\appendix

%%%%%%%%%%%%%%%%%%%%%%%%%%%%%%%%%%%%%%%%%%%%%%%%
%%%%%%%%%%%%%%%%%%%%%%%%%%%%%%%%%%%%%%%%%%%%%%%%
%%%%%%%%%%%%%%%%%%%%%%%%%%%%%%%%%%%%%%%%%%%%%%%%
%%%%%%%%%%%%%%%%%%%%%%%%%%%%%%%%%%%%%%%%%%%%%%%%
%%%%%%%%%%%%%%%%%%%%%%%%%%%%%%%%%%%%%%%%%%%%%%%%
%%%%%%%%%%%%%%%%%%%%%%%%%%%%%%%%%%%%%%%%%%%%%%%%
\section{A variant of the optimal transportation problem}\label{sec:appendix OT}

In this appendix, which follows \cite{FigalliGigli2010}, we describe the optimal transport problem ``with boundary''. Throughout we make the following assumptions: $\Omega$ is an open subset of $\mathbb{R}^d$ and $\Gamma$ is a compact subset of 
$\overline{\Omega}$. We are also given a function $c:\overline{\Omega}\times \overline{\Omega}\to \mathbb{R}$, known as the cost. We impose several assumptions on $c(x,y)$ and $\Gamma$, recorded in \eqref{eqn:cost assumptions1}, \eqref{eqn:app projections onto Gamma}.

First of all, we assume $c$ satisfies
\begin{align}\label{eqn:cost assumptions1}
  c(x,y) \textnormal{ is continuous};\; c(x,y) = c(y,x), \;\;c(x,x) = 0, \;\; c(x,y)>0 \textnormal{ if } x\neq y, \forall\ x,y 
\end{align}
Secondly, $\Gamma$ and $c$ must be such that there is a measurable function  
\begin{align*}
  P:\overline{\Omega} \to \Gamma
\end{align*}
which plays the role of the ``projection'' onto $\Gamma$, in the sense that 
\begin{align}\label{eqn:app projections onto Gamma}
  c(x,P(x)) = \inf \limits_{y\in \Gamma} c(x,y).
\end{align}

\begin{DEF}\label{def:app tilde c}
  Let $E$ be a Borel subset of $\overline{\Omega}$, we define the function
  \begin{align*}
    c(x,E) = \inf\limits_{y\in E} c(x,y).  
  \end{align*}
  Lastly, the following auxiliary cost will be relevant in what follows
  \begin{align*}
    \tilde c(x,y) = \min \{ c(x,y),\; c(x,\Gamma)+c(y,\Gamma) \}.	  
  \end{align*}	  
  We also consider the set
  \begin{align}\label{eqn:set of directions for the optimizer with boundary}
    \mathcal{K} = \{ (x,y) \in \overline{\Omega}\times\overline{\Omega} \;\mid\; c(x,y) \leq c(x,\Gamma)+c(y,\Gamma) \}.	  
  \end{align}	  
\end{DEF}

\begin{DEF}\label{def:Mc}
  Given $\Omega$ and $\Gamma$ we let $\mathcal{M}_c(\overline{\Omega})$ be the set of positive Borel measures $\mu$ on $\overline{\Omega}\setminus \Gamma$ such that 
  \begin{align*}
    \int_{\overline{\Omega}} c(x,\Gamma) \;d\mu(x) < \infty,
  \end{align*}
  {  and
  \begin{align*}
    \mu(\{x\in\overline{\Omega} \mid d(x,\Gamma) > r \} ) <\infty \textnormal{ for every } r>0.  
  \end{align*}}	  
 \end{DEF}

\begin{DEF}\label{def:Appendix admissible measures}
  Let $\mu,\nu \in \mathcal{M}_c(\overline{\Omega})$. By an admissible coupling of $\mu$ and $\nu$, we mean a positive Borel measure $\gamma$ over $\overline{\Omega}\times \overline{\Omega}$, satisfying $\gamma(\Gamma\times\Gamma)=0$ and
  \begin{align*}
    \pi_{\#}^1\gamma\mid_{\overline{\Omega} \setminus \Gamma} = \mu,\;\;\;\;\pi_{\#}^2\gamma\mid_{\overline{\Omega} \setminus \Gamma} = \nu.
  \end{align*}
  The set of admissible couplings will be denoted by $\textnormal{Adm}_{\Gamma}(\mu,\nu)$.
\end{DEF}
  Note that a measure in $\mathcal{M}_c(\overline{\Omega})$ may fail to have finite mass since $\inf c(x,\Gamma) = 0$. We are now ready to state the optimal transport problem ``with boundary''.
\begin{prob}\label{prob:OPT with boundary}
  Consider two measures $\mu,\nu \in \mathcal{M}_c(\overline{\Omega})$. Among all admissible measures $\gamma \in \textnormal{Adm}_{\Gamma}(\mu,\nu)$, find one that minimizes the functional
  \begin{align*}
    J_c(\gamma) := \int_{\overline{\Omega}\times \overline{\Omega}}c(x,y)d\gamma(x,y). 
  \end{align*}
  \end{prob}

We make no claim as to whether all of the assumptions on the cost and $\Omega$ are necessary, but they are sufficiently general for our purposes and make most of the proofs relatively straightforward (for instance, the symmetry assumption on $c(x,y)$ is not necessary but makes the notation simpler). In any case, the costs we care about in the main body of the paper are
\begin{align*}
  c_p(x,y) := |x-y|^p, 
 \;\;1\leq p \leq 2.
\end{align*}
Since we are specially concerned with these costs, we shall write $J_p(\gamma)$ to refer to the above functional when the cost is $c_p(x,y)$. At the same time, the main $\Omega$ and $\Gamma$ we care about are
\begin{align*}
  \Omega = \mathbb{R}^d\setminus \{0\},\;\textnormal{ and } \Gamma = \{0\}.
\end{align*}
Evidently, these sets, together with the costs $c_p$, comply with our requirements. The first basic fact about Problem \ref{prob:OPT with boundary} is the existence of minimizers. The proof is essentially the same as in the optimal transport case (compactness of the measures and lower semi-continuity of $J_c(\gamma)$) (cf. \cite[Theorem 1.5]{AG} and \cite[Section 2]{FigalliGigli2010}). 

\begin{thm}\label{thm:existence OT with boundary}
  Let $\mu,\nu \in \mathcal{M}_c(\overline{\Omega})$. Then $J_c(\gamma)<\infty$ for at least one $\gamma^* \in \textnormal{Adm}_{\Gamma}(\mu,\nu)$. Moreover, there exists at least one minimizer $\gamma$ for Problem \ref{prob:OPT with boundary}. 
	  
\end{thm}

\begin{proof}
  With the map $P$ is as in \eqref{eqn:app projections onto Gamma}, we define the measure
  \begin{align*}
    \gamma^* := (\textnormal{Id} \times P)_{\#} \mu + (P \times \textnormal{Id})_{\#}\nu.
  \end{align*}
  It is clear that $\gamma^* \in \textnormal{Adm}_{\Gamma}(\mu,\nu)$. At the same time, 
  \begin{align*}
    J_c(\gamma^*) = \int_{\overline{\Omega}\times \overline{\Omega}}c(x,y)\;d\gamma^*(x,y) = \int_{\Omega}c(x,\Gamma)\;d\mu(x)+\int_{\Omega}c(y,\Gamma)\;d\nu(y),
  \end{align*}	  
  and thus $J_c(\gamma^*)<\infty$ since $\mu,\nu\in \mathcal{M}_c(\overline{\Omega})$. 
  
{  In order to prove  the infimum is achieved we will first prove that $\textnormal{Adm}_{\Gamma}(\mu,\nu)$ is compact with respect to a certain notion of convergence. Let $K$ be any compact subset of $\overline \Omega\times \overline\Omega\setminus \Gamma\times\Gamma$. Since $\Gamma\times\Gamma$ and $K$ are compact, we have $d(K,\Gamma\times\Gamma)>0$. Then there exists a compact subset $\tilde K$ of $\overline\Omega\setminus \Gamma$ such that $K\subset (\tilde K \times \overline{\Omega}) \cup (\overline{\Omega} \times \tilde K)$. Since $\Gamma$ is compact and \eqref{eqn:cost assumptions1} holds, there is an $\varepsilon_0>0$ such that $\inf_{x\in\tilde K}c(x,\Gamma)>\varepsilon_0$. And thus $\mu(\tilde K)<+\infty$ since $\mu\in \mathcal{M}_c(\overline{\Omega})$. Similarly, we have $\nu(\tilde K)<+\infty$. Therefore, if $\gamma\in \textnormal{Adm}_{\Gamma}(\mu,\nu)$, we have
  \begin{align*}
    \gamma(K)\leq \mu(\tilde K)+\nu(\tilde K)<\infty,
  \end{align*}
  Since $\mu(\tilde K)+\nu(\tilde K)$ is independent of $\gamma$, it follows that given a sequence $\{\gamma_n\}$ in $\textnormal{Adm}_{\Gamma}(\mu,\nu)$ there is a subsequence $\gamma_{n_k}$ and a measure $\gamma$ in $\overline{\Omega}\times \overline{\Omega}$ such that $\gamma_{n_k} \rightharpoonup \gamma$, the convergence being in the following sense
  \begin{align}\label{equation:appendix weak convergence}
    \int_{\overline{\Omega}\times \overline{\Omega}} \phi(x,y) \;d\gamma(x,y) = \lim \limits_{k\to\infty} \int_{\overline{\Omega}\times \overline{\Omega}} \phi(x,y)\;d\gamma_{n_k}(x,y),\;\; \phi \in C^0_c ( \overline{\Omega}\times \overline{\Omega} \setminus \Gamma\times \Gamma).
  \end{align}	   
  Now we must show that $\gamma \in \textnormal{Adm}_{\Gamma}(\mu,\nu)$. Observe that for each $n$
  \begin{align*}
    & \gamma_{n} \in \textnormal{Adm}_{\Gamma}(\mu,\nu) \Rightarrow \\
    & \gamma_n ( \{ (x,y) \mid d(x,\Gamma) \geq r \textnormal{ or } d(y,\Gamma) \geq r\} ) \leq \mu( \{ d(x,\Gamma) \geq r\} ) + \nu ( \{ d(x,\Gamma) \geq r \} ),
  \end{align*}
  so from the assumptions on $\mu$ and $\nu$ (Definition \ref{def:Mc}) it follows that the right hand side goes to zero as $r\to \infty$ with a rate depending only on $\mu$ and $\nu$ (note that when $\overline{\Omega}$ is compact this last assertion holds trivially). From this estimate and the convergence in \eqref{equation:appendix weak convergence} it is not hard to see that 
  \begin{align*}
    \int_{\overline{\Omega}\times \overline{\Omega}} \phi(x)\;d\gamma(x,y) = \lim \limits_{k\to\infty}\int_{\overline{\Omega}\times \overline{\Omega}} \phi(x)\;d\gamma_{n_k}(x,y),\;\;\forall\;\phi\in C^0_c(\overline{\Omega}\setminus \Gamma),
  \end{align*}
  a similar statement holds for functions of $y$ with support away from $\Gamma$. In particular, 
  \begin{align*}
     \int_{\overline{\Omega}\times \overline{\Omega}} \phi(x)\;d\gamma(x,y) =  \int_{\overline{\Omega}}\phi(x)\;d\mu(x),\;\;\int_{\overline{\Omega}\times \overline{\Omega}} \psi(x)\;d\gamma(x,y) =  \int_{\overline{\Omega}}\psi(y)\;d\nu(y),
  \end{align*}    
  which shows that $\gamma \in \textnormal{Adm}_{\Gamma}(\mu,\nu)$. In conclusion, the set of admissible couplings $\textnormal{Adm}_{\Gamma}(\mu,\nu)$ is sequentially compact with respect to the notion of convergence in \eqref{equation:appendix weak convergence}.

  }

  Let $\gamma_n$ be a minimizing sequence in $\textnormal{Adm}_{\Gamma}(\mu,\nu)$, that is a sequence such that $J_c(\gamma_n)\to \inf J_c$ as $n\to\infty$. At the same time, let $c_k$ be a monotone increasing sequence of continuous functions with compact support in $\overline\Omega\times\overline \Omega \setminus \Gamma\times \Gamma$ and such that $c_k(x,y) \to c(x,y)$ locally uniformly in $\overline\Omega\times\overline\Omega \setminus \Gamma\times\Gamma$. Using a diagonal argument, there exist a subsequence, still denoted by $\gamma_n$, and $\gamma_*\in \textnormal{Adm}_{\Gamma}(\mu,\nu)$ such that for every fixed $k$
  \begin{align*}
    \lim\limits_{n\to \infty}\int_{\overline{\Omega}\times \overline{\Omega}} c_k(x,y)\;d\gamma_n(x,y) = \int_{\overline{\Omega}\times \overline{\Omega}} c_k(x,y)\;d\gamma_*(x,y).    
  \end{align*}
  Now, by the monotonicity of the $c_k$, we have
  \begin{align*}
    J_c(\gamma_*) = \int_{\overline{\Omega}\times \overline{\Omega}} c(x,y)\;d\gamma_*(x,y) = \sup \limits_{k}J_{c_k}(\gamma_*),
  \end{align*}
  while for any $k$ we have
  \begin{align*}
  J_{c_k}(\gamma_*)  =\lim\limits_{n\to\infty} \int_{\overline{\Omega}\times \overline{\Omega}}c_k(x,y)\;d\gamma_n(x,y)\leq \lim\limits_{n\to\infty} J(\gamma_n) = \inf\limits_{\gamma \in \textnormal{Adm}_{\Gamma}(\mu,\nu)} J(\gamma).
  \end{align*}
  This proves that $\gamma_*$ achieves the minimum value of $J_c$ among all admissible plans.
\end{proof}

We now characterize minimizers for Problem \ref{prob:OPT with boundary} using $c$-concave functions and $c$-cyclical monotonicity.
\begin{DEF}\label{def:c-concavity and c-transform}
  For a function  $\phi:\overline{\Omega}\to \mathbb{R} \cup \{\pm \infty\}$ with $\phi \neq -\infty$ for at least some $x$, its $c$-transform $\phi^c:\overline{\Omega} \to \mathbb{R}\cup \{-\infty\}$ is the function given by
  \begin{align*}
    \phi^c(y) & = \inf \limits_{x \in \overline{\Omega}} \big \{ c(x,y)-\phi(x) \big \}.
  \end{align*}
  A function $\phi$ is said to be $c$-concave if there is some $\psi$ such that
  \begin{align*}
    \phi = \psi^c.
  \end{align*}
  If $\phi$ and $\psi$ are two $c$-concave functions such that $\phi = \psi^c$ and $\psi = \phi^c$ then we say they are $c$-conjugate to one another. Just the same, we talk about $\tilde c$-transforms and $\tilde c$-concave functions.
\end{DEF}
{ 
\begin{rem}\label{remark:appendix c-transform of a function is upper semicontinuous}
  Since the cost is assumed to be continuous it follows that $\phi^c$ is the infimum of a family of continuous functions of $y$ ($\{ y\mapsto c(x,y)-\phi(x)\}_x$), accordingly, $\phi^c$ is upper semicontinuous. In particular, if $(\phi,\psi)$ is a $c$-conjugate pair then both $\phi$ and $\psi$ are upper semicontinuous functions. 

\end{rem}

\begin{rem}\label{remark:appendix inequality for c conjugate pairs}
  Suppose that $(\phi,\psi)$ are $c$-conjugate. Then for every $x$ and $y$ we have 
  \begin{align*}
    \phi(x)+\psi(y) \leq c(x,y).	  
  \end{align*}	  
\end{rem}
The set of pairs $(x,y)$ for which we have equality will be important in what follows.
\begin{DEF}\label{definition:appendix subdifferential}
  Let $\phi$ be a $c$-concave function and $\psi = \phi^c$. The $c$-subdifferential of $\phi$, denoted by $\partial^c \phi$, is defined as the set of pairs $(x,y)$ such that
  \begin{align*}
    \phi(x)+\psi(y) = c(x,y).
  \end{align*}
  Moreover, for each $x$ we define $\partial^c \phi(x)$ to be the set of all $y$ such that $(x,y) \in \partial^c\phi$. We define $\partial^{\tilde c}\phi$ and $\partial^{\tilde c}(x)$ for a $\tilde c$-concave $\phi$ in the same manner.
\end{DEF}

}
\begin{DEF}
  A subset of $\overline{\Omega} \times \overline{\Omega}$ is said to be $c$-cyclically monotone if given a finite sequence $\{(x_i,y_i) \}_{i=0}^n$ and any permutation $\sigma$, we have
  \begin{align*}
    \sum \limits_{i=1}^n c(x_i,y_i) \leq \sum \limits_{i=1}^n c(x_i,y_{\sigma(i)}).
  \end{align*}  
  If $c$ is replaced by $\tilde c$, we have $\tilde c$-cyclical monotonicity.
\end{DEF}

The following Proposition is a (minor) modification of a well known convex analysis result of Rockafellar (previously extended for $c$-concave functions). This modification pertains the set $\Gamma$ and the costs $c(x,y)$ and $\tilde c(x,y)$.

\begin{prop}\label{prop:Rockafellar theorem}
  Let $\gamma$ be a measure concentrated on $\mathcal{K}$ and such that $\textnormal{spt}(\gamma) \cup \Gamma \times \Gamma$ is $\tilde c$-cyclically monotone. Then, there are $c$-conjugate functions $\phi$ and $\psi$ such that
  \begin{align*}
    \phi \equiv \psi \equiv 0\textnormal{ on } \Gamma, \textnormal{ and }\;\textnormal{spt}(\gamma) \subset \partial^c \phi.
  \end{align*}	  
  
\end{prop}

\begin{proof}
  This follows from the standard optimal transport theory. Indeed, as shown in the proof of \cite[Theorem 1.13, (ii) $\Rightarrow$ (iii)]{AG}, since $\textnormal{spt}(\gamma) \cup \Gamma \times \Gamma$ is $\tilde c$-cyclically monotone, there must be a $\tilde c$-concave function $\phi$ such that
  \begin{align*}  
    \textnormal{spt}(\gamma) \cup \Gamma \times \Gamma \subset \partial^{\tilde c} \phi.
  \end{align*}
  Since any pair $(x,y) \in \Gamma\times \Gamma$ belongs to $\partial^{\tilde c}\phi$, it follows that
  \begin{align*}
    \phi(x)+\phi^{\tilde c}(y) = \tilde c(x,y) = 0 \;\;\forall\;x,y\in\Gamma.
  \end{align*}
  We emphasize that the above holds for any two points $x$ and $y$ in $\Gamma$, which in particular means that $\phi$ and $\phi^{\tilde c}$ are constant on $\Gamma$. Adding a constant to $\phi$ we can assume without loss of generality that $\phi=0$ on $\Gamma$, which in turn guarantees that $\phi^{\tilde c}=0$ on $\Gamma$ as well.
   
  We claim that $\partial^{\tilde c}\phi \cap \mathcal{K} \subset \partial^c \phi$. Indeed, if $(x_0,y_0) \in \mathcal{K}$ is such that $y_0 \in \partial^{\tilde c}\phi(x_0)$ then
  \begin{align*}
    \phi(x) \leq \tilde c(x,y_0) - \phi^c(y_0) \leq c(x,y_0)-\phi^c(y_0),	  
  \end{align*}	  
  since $\tilde c(x,y) \leq c(x,y)$ for all $x$ and $y$. Since $(x_0,y_0)\in\mathcal{K}$ we have $\tilde c(x_0,y_0) = c(x_0,y_0)$, so
  \begin{align*}
    \phi(x_0) = \tilde c(x,y_0) - \phi^c(y_0) = c(x,y_0)-\phi^c(y_0).
  \end{align*}	  
 It follows from this that $(x_0,y_0) \in \partial^c\phi(x)$, and the claim is proved. The same argument also shows that if $y\in \Gamma$, then $\phi^{c}(y) = \phi^{\tilde c}(y) = 0$. Since $\phi$ was chosen so that $\textnormal{spt}(\gamma)\subset \partial^{\tilde c}\phi$ and $\gamma$ is supported in $\mathcal{K}$, it follows that $\textnormal{spt}(\gamma) \subset \partial^c\phi$.  Therefore $\phi$ and $\phi^c$ are the desired $c$-conjugate functions.
  \end{proof}

As in the usual optimal transport problem, a basic tool for the analysis of Problem \ref{prob:OPT with boundary} is a \emph{dual problem}. This problem deals with a family of admissible pairs of functions
{ 
\begin{align}  
\textnormal{Adm}^c & := \Big \{ (\phi,\psi) \mid \phi\in L^1(\mu) \textnormal{ and } \psi \in L^1(\nu), \notag \\ 
      & \quad\quad\quad \quad \quad \;\; \phi \textnormal{ and } \;\psi \textnormal{ are upper semicontinuous},\notag\\
      & \quad\quad\quad \quad \quad \;\; \phi\equiv \psi \equiv 0 \textnormal{ on } \Gamma, \notag\;\\ 
      & \quad\quad\quad \quad \quad \;\; \textnormal{and } \phi(x)+\psi(y) \leq c(x,y) \textnormal{ in } \overline{\Omega}\times \overline{\Omega} \Big \}.\label{eqn:def dual problem admissible functions}	  	  
\end{align}	  
}
We now can state the problem dual to Problem \eqref{prob:OPT with boundary dual}.
	  
\begin{prob}\label{prob:OPT with boundary dual}	  
  Among all pairs $(\phi,\psi) \in \textnormal{Adm}^c$, find one that maximizes the functional
  \begin{align*}
    J^*(\phi,\psi) := \int_{\overline{\Omega}} \phi(x) \;d\mu(x) + \int_{\overline{\Omega}} \psi(y) \;d\nu(y).
  \end{align*}
\end{prob}	  

The characterization of minimizers in Problem \ref{prob:OPT with boundary} and maximizers for Problem \ref{prob:OPT with boundary dual} is the content of Theorem \ref{thm:characterization minimizing plan} and Lemma \ref{lem:duality general}. In the proof we will make use of Proposition \ref{prop:Rockafellar theorem}, together with the characterization of optimizers for the usual optimal transportation problem \cite[Theorem 1.13]{AG}.

\begin{thm}\label{thm:characterization minimizing plan}
  Let $\gamma \in \textnormal{Adm}_{\Gamma}(\mu,\nu)$ for two measures $\mu,\nu \in \mathcal{M}_c(\Omega)$. 
  Then $\gamma$ is a minimizer for Problem \ref{prob:OPT with boundary} if and only if $\gamma$ is concentrated on the set $\mathcal{K}$ defined in \eqref{eqn:set of directions for the optimizer with boundary} and $\textnormal{spt}(\gamma) \cup \Gamma \times \Gamma$ is a $\tilde c$-cyclically monotone set.

\end{thm}

\begin{proof}
  Assume first that $\gamma \in \textnormal{Adm}_{\Gamma}(\mu,\nu)$ is optimal.  Consider $\tilde \gamma$, the plan given { by $\tilde \gamma = \hat \gamma_{\mid \overline{\Omega}\times \overline{\Omega}\setminus \Gamma\times \Gamma}$ where $\hat \gamma$ is defined as
  \begin{align*}
    \hat \gamma := \gamma_{\mid_{ \mathcal{K}}}+(\pi_1,P\circ \pi_1)_{\#} \left ( \gamma_{\mid_{\overline{\Omega}\times \overline{\Omega}\setminus \mathcal{K} }} \right )\;\;+(P\circ \pi_2,\pi_2)_{\#} \left ( \gamma_{\mid_{\overline{\Omega}\times \overline{\Omega}\setminus \mathcal{K} }} \right ),	
  \end{align*}	
  here $P$ is as in \eqref{eqn:app projections onto Gamma}. What the plan $\hat \gamma$ is meant to do is adjusting the original plan $\gamma$ by shifting the transport of some of the mass so that it is sent to $\Gamma$, whenever it is advantageous to do so (and only for points $(x,y)$ outside of $\mathcal{K}$). The coupling $\tilde \gamma$ comes from taking $\hat \gamma$ and discarding any potential mass $\Gamma \times \Gamma$, this makes sure we have an admissible coupling. Therefore $\tilde \gamma \in \textnormal{Adm}_\Gamma(\mu,\nu)$. Moreover, we have the formula}
  \begin{align*}	
    \int_{\overline{\Omega}\times \overline{\Omega}} c(x,y)\;d\tilde \gamma(x,y) = \int_{\mathcal{K}}c(x,y)d\gamma(x,y) + \int_{\overline{\Omega}\times \overline{\Omega}\setminus \mathcal{K}} [c(x,\Gamma) + c(\Gamma,y)]\; d\gamma(x,y).
  \end{align*}	
  From the definition of $\mathcal{K}$, we have { $c(x,y) > c(x,\Gamma)+c(\Gamma,y)$} outside of $\mathcal{K}$, thus
  \begin{align*}	
    \int_{\overline{\Omega}\times\overline{\Omega}\setminus \mathcal{K}} [c(x,\Gamma) + c(\Gamma,y)] \;d\gamma(x,y) \leq  \int_{\overline{\Omega}\times \overline{\Omega}\setminus \mathcal{K}} c(x,y) \;d\gamma(x,y).
  \end{align*}
  It follows that
  \begin{align*}	
    \int_{\overline{\Omega}\times \overline{\Omega}} c(x,y)\;d\tilde \gamma(x,y)\leq \int_{\overline{\Omega}\times \overline{\Omega}} c(x,y)\;d\gamma(x,y),
  \end{align*}
  with strict inequality if and only if $\gamma(\overline{\Omega} \times \overline{\Omega} \setminus \mathcal{K})>0$. By the optimality of $\gamma$ we then conclude that $\gamma(\overline{\Omega} \times \overline{\Omega} \setminus \mathcal{K}) = 0$, that is, $\gamma$ is supported in $\mathcal{K}$.
  
  Now, we must show that $\textnormal{spt}(\gamma) \cup \Gamma\times \Gamma$ is $\tilde c$-monotone. We deal first with the case where $\gamma$ has finite mass. In this instance, let us write 
  \begin{align}\label{eqn:app bar mu and bar nu}
    \bar \mu = \pi_{\#}^1\gamma,\;\bar \nu = \pi_{\#}^2\gamma.
  \end{align}
  Then, as $\bar \mu$ and $\bar \nu$ are the marginals of $\gamma$ (in all of $\overline{\Omega}$), they must have the same total mass which is finite since $\gamma$ has finite mass. Let $\gamma_0$ denote the optimal transport plan between $\bar\mu$ and $\bar \nu$ according to $\tilde c$, and let $\tilde \gamma_0$ be constructed from $\gamma_0$ in the same way as $\tilde \gamma$ was constructed from $\gamma$ { (first by pushing parts of its mass to the boundary as done above, yielding a measure $\hat \gamma_0$, and then restricting to $\overline{\Omega}\times \overline{\Omega} \setminus \Gamma\times \Gamma$). Since $\bar \mu_{\mid \overline \Omega} = \mu$ and $\bar \nu_{\mid \overline \Omega}=\nu$ we have that $\tilde \gamma_0$ is a measure in $\textnormal{Adm}_{\Gamma}(\mu,\nu)$, and as argued above for $\gamma$ and $\tilde \gamma$ if $\gamma_0$ were not supported in $\mathcal{K}$ then $\tilde \gamma_0$ would be a better coupling. This shows that $c=\tilde c$ $\gamma$-a.e. and $\gamma_0$-a.e. and thus
  \begin{align*}
    J_{\tilde c}(\gamma) = J_c(\gamma),\;J_{\tilde c}(\gamma_0) = J_c(\gamma_0).
  \end{align*}
  Combining these identities with the optimality of $\gamma$ and $\gamma_0$ yields the inequalities
  \begin{align*}
    J_{\tilde c}(\gamma) \geq J_{\tilde c}(\gamma_0) = J_{c}(\gamma_0) \geq J_{c}({\gamma_0}_{\mid \overline{\Omega}\times \overline{\Omega}\setminus \Gamma\times\Gamma}) \geq J_{c}(\gamma),
  \end{align*}
  (we used that ${\gamma_0}_{\mid \overline{\Omega}\times \overline{\Omega}\setminus \Gamma\times\Gamma} \in \textnormal{Adm}_{\Gamma}(\mu,\nu)$ and that $c(x,y)\geq 0$) and we conclude that
  \begin{align*}
    J_{\tilde c}(\gamma_0) = J_{\tilde c}(\gamma).
  \end{align*}}
  Thus $\gamma$ is an optimal plan for the usual transport problem with cost $\tilde c$. By optimal transport theory, the support set $\textnormal{spt}(\gamma)$ is $\tilde c$-cyclically monotone. To prove that $\textnormal{spt}(\gamma) \cup \Gamma\times \Gamma$ is still $\tilde c$-cyclically monotone, simply note that if $\gamma_0$ is any measure supported in $\Gamma\times \Gamma$, {  then $\gamma+\gamma_0$ may not belong to $\textnormal{Adm}_\Gamma(\mu,\nu)$ but arguing as above we can show that it is optimal for the standard optimal transport problem with cost $\tilde c$ and marginals $\pi^1_{\#}(\gamma+\gamma_0)$ and $\pi^2_{\#}(\gamma+\gamma_0)$. This shows $\textnormal{spt}(\gamma+\gamma_0) = \textnormal{spt}(\gamma) \cup \Gamma\times \Gamma$ is $\tilde c$-cyclically monotone.}

  This covers the case where $\gamma$ has finite mass. For the general case, we argue just as in \cite[Proposition 2.3]{FigalliGigli2010}, that the one property from the classical optimal transport problem that we needed was that if the support of $\gamma$ is not $\tilde c$-cyclically monotone, then $\gamma$ cannot be optimal with respect to $\tilde c$. It is worth noting that that even if $\bar \mu$ and $\bar \nu$ do not have finite mass, they are still the marginals of $\gamma$ by definition \eqref{eqn:app bar mu and bar nu}, so the set of measures with marginals $\bar \mu$ and $\bar \nu$ is non-empty, so one can proceed with the Kantorovich problem as in the standard optimal transport theory. Therefore, the above argument extends to the case of $\gamma$ with infinite mass and we conclude that $\textnormal{spt}(\gamma) \cup \Gamma \times \Gamma$ is $\tilde c$-cyclically monotone in all cases.
   
  Conversely, assume that $\gamma$ is supported in $\mathcal{K}$ and that $\textnormal{spt}(\gamma) \cup \Gamma \times \Gamma$ is a $\tilde c$-cyclically monotone set. Then Proposition \ref{prop:Rockafellar theorem} says that there is a function 
  $\phi$ which is $c$-concave, such that $\phi$ and $\phi^c$ both vanish on $\Gamma$, and
  \begin{align*}
    \textnormal{spt}(\gamma) \subset \partial^c\phi.
  \end{align*}
  In particular, this means that  $\phi(x)+\phi^{c}(y) = c(x,y)$ on $\textnormal{spt}(\gamma)$, so
   \begin{align*}
    \int_{\overline{\Omega}\times \overline{\Omega}} c(x,y)\;d\gamma(x,y) & = \int_{\overline{\Omega}\times \overline{\Omega}} [\phi(x)+\phi^{c}(y)]\;d\gamma(x,y),\\
	  & = \int_{(\overline\Omega\setminus\Gamma)\times\overline{\Omega}} \phi(x)\;d\gamma(x,y) + \int_{\overline{\Omega}\times(\overline\Omega\setminus\Gamma)} \phi^{c}(y)\;d\gamma(x,y),\\
	  & = \int_{\overline{\Omega}\setminus\Gamma} \phi(x)\;d\mu(x) + \int_{\overline{\Omega}\setminus
	  \Gamma} \phi^c(y)\;d\nu(y).	  	  
  \end{align*}	  
  This suffices to guarantee the optimality of $\gamma$. Indeed, take any $\tilde \gamma \in \textnormal{Adm}_{\Gamma}(\mu,\nu)$, then
  \begin{align*}
    \int_{\overline{\Omega}\times \overline{\Omega}}  c(x,y)\;d\tilde \gamma(x,y) & \geq \int_{\overline{\Omega}
    \times \overline{\Omega}}  [\phi(x)+\phi^c(y)]\;d\tilde \gamma(x,y)\\
	  & = \int_{\overline{\Omega}\setminus\Gamma} \phi(x)\;d\mu(x) + \int_{\overline{\Omega}\setminus\Gamma} \phi^c(y)\;d\nu(y) = 
	  \int_{\overline{\Omega}\times \overline{\Omega}} c(x,y)\;d\gamma(x,y), 
  \end{align*}
  and we conclude that $\gamma$ achieves the minimum value. 
\end{proof}  

Just as in the usual optimal transport problem, a solution to Problem \ref{prob:OPT with boundary} corresponds to a solution to Problem \ref{prob:OPT with boundary dual}, and the corresponding values coincide.
\begin{lem}\label{lem:duality general}
  The problems \eqref{prob:OPT with boundary} and \eqref{prob:OPT with boundary dual} are dual, meaning that
  \begin{align*}
    \inf \limits_{\gamma \in \textnormal{Adm}_\Gamma(\mu,\nu)} J(\gamma) = \sup \limits_{(\phi,\psi) \in \textnormal{Adm}^c} J^*(\phi,\psi).
  \end{align*}
  \end{lem}
  
\begin{proof}  
  If $(\phi,\psi) \in \textnormal{Adm}^c$, then $\phi(x)+\psi(y) \leq c(x,y)$ for all $x$ and $y$ and $\phi\equiv\psi\equiv 0$ 
  on $\Gamma$. Therefore, for any $\gamma \in \textnormal{Adm}_\Gamma(\mu,\nu)$ we have
  \begin{align*}	
    \int_{\overline{\Omega}\times \overline{\Omega}} c(x,y) \;d\gamma(x,y) & \geq \int_{\overline{\Omega} \times \overline{\Omega}} [\phi(x)+\psi(y)]\;d\gamma(x,y)\\
	& = \int_{\overline{\Omega}\setminus\Gamma} \phi(x)\;d\mu(x) + \int_{\overline{\Omega}\setminus\Gamma}\psi(y)\;d\nu(y)
	\\
	& = \int_{\overline{\Omega}} \phi(x)\;d\mu(x) + \int_{\overline{\Omega}}\psi(y)\;d\nu(y).	 
  \end{align*}	 
  Since $(\phi,\psi) \in \textnormal{Adm}^c$ and $\gamma \in \textnormal{Adm}_{\Gamma}(\mu,\nu)$ were arbitrary, it follows that
  \begin{align}\label{equation:appendix inf sup easy inequality}
    \inf \limits_{\gamma \in \textnormal{Adm}_{\Gamma}(\mu,\nu) }\int_{\overline{\Omega}\times \overline{\Omega}} c(x,y) \;d\gamma(x,y) \geq \sup \limits_{(\phi,\psi) \in \textnormal{Adm}^c} \; \Big \{  \int_{\overline{\Omega}} \phi \;d\mu(x) + \int_{\overline{\Omega}} \psi \;d\nu(y) \Big \}.
  \end{align}
  The reverse inequality follows from Theorem \ref{thm:characterization minimizing plan}. To see why, let $\pi \in \textnormal{Adm}_{\Gamma}(\mu,\nu)$ be the minimizer, then the theorem says that { $\textnormal{spt}(\gamma) \cup \Gamma \times \Gamma$ is $\tilde c$-cyclically monotone and its support is} contained in $\mathcal{K}$, in which case Proposition \ref{prop:Rockafellar theorem} says that there are functions $\phi$ and $\psi$ which are $c$-conjugate, vanish on $\Gamma$, and such that $\phi(x)+\psi(y)=c(x,y)$ for $\gamma$-almost every $(x,y)$. {  The functions $\phi,\psi$ have a couple of extra properties. First, since $\psi(y) = 0$ for $y\in \Gamma$, we have $\phi(x) \leq c(x,y)$ for every $y\in \Gamma$ and taking the infimum in $y$ it follows that
  \begin{align*}
    \phi(x) \leq c(x,\Gamma) \;\forall\;x\in \Gamma.
  \end{align*}
  Likewise, it follows that $\psi(y) \leq c(y,\Gamma)$ for every $y\in \Gamma$. This implies that
  \begin{align}\label{equation:appendix max phi max psi in L1}
    \max\{\phi,0\} \in L^1(\mu),\; \max\{\psi,0\} \in L^1(\nu).
  \end{align}
  In particular, the integrals $\int_{\overline\Omega}\phi(x)d\mu(x)$ and $\int_{\overline{\Omega}}\psi(y)d\nu(y)$ are well defined. Secondly, using that $\phi(x)+\psi(y)=c(x,y)$ for $\gamma$-almost every $(x,y)$, that $\phi\equiv \psi \equiv 0$ on $\Gamma$, and $\gamma \in \textnormal{Adm}_{\Gamma}(\mu,\nu)$, it follows that
  \begin{align*}
     \int_{\overline{\Omega}} \phi(x)\;d\mu(x)+\int_{\overline{\Omega}}\psi(y)\;d\nu(y) & = \int_{\overline{\Omega}\setminus \Gamma} \phi(x)\;d\mu(x)+\int_{\overline{\Omega}\setminus \Gamma}\psi(y)\;d\nu(y)\\
	& = \int_{\overline{\Omega}\times \overline{\Omega}}[\phi(x)+\psi(y)]\;d\gamma(x,y),\\
	& = \int_{\overline{\Omega}\times \overline{\Omega}}c(x,y)\;d\gamma(x,y).
  \end{align*}
 Since this last integral is finite it follows that $\int_{\overline\Omega}\phi(x)d\mu(x)$ and $\int_{\overline{\Omega}}\psi(y)d\nu(y)$ are finite and in light of \eqref{equation:appendix max phi max psi in L1} it follows that $\phi \in L^1(\mu)$ and $\psi \in L^1(\nu)$. This shows that $(\phi,\psi) \in \textnormal{Adm}^c$ and this yields the reverse inequality to \eqref{equation:appendix inf sup easy inequality}, proving the lemma.}
  \end{proof}

The following lemma is a minor modification of \cite[Lemma 2.1]{FigalliGigli2010} and we omit its proof. The lemma itself is a variant of a standard lemma in optimal transport theory \cite[Lemma 5.3.2]{AGS}. We recall that below $\mathcal{M}_p(\overline\Omega):=\mathcal{M}_c(\overline\Omega)$ for $c(x,y)=|x-y|^p$.
\begin{lem}\label{lem:app gluing lemma}
  Let $p\geq 1$ and consider measures $\mu_1,\mu_2,\mu_3 \in \mathcal{M}_p(\overline\Omega)$, $\gamma^{12} \in \textnormal{Adm}_{\Gamma}(\mu_1,\mu_2)$, and $\gamma^{23} \in \textnormal{Adm}_{\Gamma}(\mu_2,\mu_3)$. 
  Then, there is a Borel measure in $\overline{\Omega}\times \overline{\Omega}\times \overline{\Omega}$, denoted $\gamma^{123}$, whose $2$-marginals satisfy
  \begin{align}\label{eqn:app gluing lemma marginals}
    \pi^{12}_{\#}\gamma^{123} =  \gamma^{12}+\sigma^{12},\;\pi^{23}_{\#}\gamma^{123} =  \gamma^{23}+\sigma^{23},
  \end{align}	  
  where $\sigma^{12}$ and $\sigma^{23}$ are measures concentrated on the  set $\{(x,x)\;|\;x\in\Gamma\}$ and
  $\pi^{12}(x_1,x_2,x_3)=(x_1,x_2), \pi^{2,3}(x_1,x_2,x_3)=(x_2,x_3)$.
\end{lem}

We can now prove that $\textnormal{d}_{\mathbb{L}_p}(\mu,\nu)$ is a metric in $\mathcal{M}_p(\overline\Omega)$.
\begin{thm}\label{th:metric}
The quantity
\[
\textnormal{d}_{\mathbb{L}_p}(\mu,\nu):=\inf_{\gamma\in\textnormal{Adm}_{\Gamma}(\mu,\nu)}\left( \int_{\overline{\Omega}\times \overline{\Omega}}|x-y|^pd\gamma(x,y)\right)^{\frac{1}{p}}
\]
defines a metric in $\mathcal{M}_p(\overline\Omega)$. 
\end{thm}
\begin{proof}
  It is clear that $\textnormal{d}_{\mathbb{L}_p}(\mu,\nu)=\textnormal{d}_{\mathbb{L}_p}(\nu,\mu)$ and that $\textnormal{d}_{\mathbb{L}_p}(\mu,\nu)\geq 0$ for all $\mu$ and $\nu$. Moreover, if $\textnormal{d}_{\mathbb{L}_p}(\mu,\nu)=0$ that means there is some $\gamma \in \textnormal{Adm}_{\Gamma}(\mu,\nu)$ such that
  \begin{align*}
   0= \int_{\overline{\Omega}\times \overline{\Omega}}|x-y|^p\;d\gamma(x,y) \Rightarrow \textnormal{spt}(\gamma) \subset \{ (x,y) \in \overline{\Omega}\times \overline{\Omega}\;\mid x=y\}.
  \end{align*}
  This implies that for any $\phi \in C^0_c(\overline{\Omega}\setminus \Gamma)$ we have
  \begin{align*}
    \int_{\overline{\Omega}\setminus \Gamma}\phi(x)\;d\mu(x) = \int_{\overline{\Omega}\times \overline{\Omega}}\phi(x)\;d\gamma(x,y) = \int_{\overline{\Omega}\times \overline{\Omega}}\phi(y)\;d\gamma(x,y) = \int_{\overline{\Omega}\setminus \Gamma}\phi(y)\;d\nu(y),	 
  \end{align*}	  
  in other words, $\mu=\nu$. It remains to prove the triangle inequality. Consider measures $\mu_1,\mu_2,\mu_3$ in $\mathcal{M}_p(\overline\Omega)$ and let the measures $\gamma^{12} \in \textnormal{Adm}_{\Gamma}(\mu_1,\mu_2)$ and $\gamma^{23} \in \textnormal{Adm}_{\Gamma}(\mu_2,\mu_3)$ be optimizers for the respective problems. Then Lemma \ref{lem:app gluing lemma} guarantees there is a measure $\gamma^{123}$ satisfying \eqref{eqn:app gluing lemma marginals}.
  
  It will be convenient to denote an element $\overline{\Omega}\times \overline{\Omega} \times \overline{\Omega}$ as $(x_1,x_2,x_3)$. At the same time, the ``coordinates'' $x_1,x_2,x_3$ define three functions $\overline{\Omega}\times \overline{\Omega} \times \overline{\Omega} \to \overline{\Omega} \subset \mathbb{R}^d$. With this in mind, we note that the function $|x_1-x_3|^p$ is independent of $x_2$, so (denoting $\pi^{13}(x_1,x_2,x_3)=(x_1,x_3)$)
  \begin{align}\label{eqn:app triangle inequality first measure}
    \textnormal{d}_{\mathbb{L}_p}(\mu_1,\mu_3)^p \leq \int_{\overline{\Omega}\times\overline{\Omega}}|x_1-x_3|^p\;d\pi^{13}_{\#}\gamma^{123}(x_1,x_3)= \int_{\overline{\Omega}\times \overline{\Omega}\times\overline{\Omega}}|x_1-x_3|^p\;d\gamma^{123}(x_1,x_2,x_3) 
  \end{align}
  On the other hand, applying the Minkowski's inequality in $L^p(\overline{\Omega}\times \overline{\Omega} \times \overline{\Omega}, d\gamma^{123})$ for the functions $x_1-x_2,$ and $x_2-x_3$, we have 
  \begin{align*} 
    & \left ( \int_{\overline{\Omega}\times \overline{\Omega}}|x_1-x_3|^p\;d\gamma^{123}(x_1,x_2,x_3)\right )^{\frac{1}{p}} \\
    & \leq \left ( \int_{\overline{\Omega}\times \overline{\Omega}}|x_1-x_2|^p\;d\gamma^{123}(x_1,x_2,x_3)\right )^{\frac{1}{p}}  + \left ( \int_{\overline{\Omega}\times \overline{\Omega}}|x_2-x_3|^p\;d\gamma^{123}(x_1,x_2,x_3)\right )^{\frac{1}{p}}. 	
  \end{align*}
  Then, using the optimality of $\gamma^{12}$ as well as \eqref{eqn:app gluing lemma marginals},
  \begin{align*}
     \int_{\overline{\Omega}\times \overline{\Omega}}|x_1-x_2|^p\;d\gamma^{123}(x_1,x_2,x_3) & =  \int_{\overline{\Omega}\times \overline{\Omega}}|x_1-x_2|^p\;d(\gamma^{12}+\sigma^{12})(x_1,x_2) \\
      & = \int_{\overline{\Omega}\times \overline{\Omega}}|x_1-x_2|^p(x_1,x_2)\;d\gamma^{12} = \textnormal{d}_{\mathbb{L}_p}(\mu_1,\mu_2)^p,
  \end{align*}
  where the second to last inequality used the fact that $\sigma^{12}$ is supported on the diagonal, so that $\sigma^{12}$-a.e. we have $|x_1-x_2|=0$. Just the same, we can see that 
  \begin{align*}
     \int_{\overline{\Omega}\times \overline{\Omega}}|x_2-x_3|^p\;d\gamma^{123}(x_1,x_2,x_3) = \textnormal{d}_{\mathbb{L}_p}(\mu_2,\mu_3)^p.	  
  \end{align*}  
  Then, recalling \eqref{eqn:app triangle inequality first measure}, we conclude that 
  \begin{align*}
    \textnormal{d}_{\mathbb{L}_p}(\mu_1,\mu_3) \leq \textnormal{d}_{\mathbb{L}_p}(\mu_1,\mu_2) + \textnormal{d}_{\mathbb{L}_p}(\mu_2,\mu_3),
  \end{align*}	  
  which finishes the proof that $\textnormal{d}_{\mathbb{L}_p}(\mu,\nu)$ is a metric.
\end{proof}

\begin{proof}[Proof of Proposition \ref{prop:bound for distances of restricted measures}]
  { For any $\gamma \in \textnormal{Adm}(\mu,\nu)$ (recall that now ${\Gamma}=\{0\}$),} we have
  \begin{align*}
    \int_{B_1} \psi \;d\mu(x) - \int_{B_1} \psi\;d\nu(y) & = \int_{\textnormal{spt}(\psi)\times \mathbb{R}^d} \psi(x)\;d\gamma(x,y) - \int_{\mathbb{R}^d\times \textnormal{spt}(\psi)} \psi(y)\;d\gamma(x,y) \\ 
	  & = \int_{A_\psi} [\psi(x)-\psi(y)]\;d\gamma(x,y),
  \end{align*}
  where $A_\psi := (\textnormal{spt}(\psi)\times \mathbb{R}^d) \cup (\mathbb{R}^d\times \textnormal{spt}(\psi))$. Then
  \begin{align*}
    \left | \int_{B_1} \psi \;d\mu(x) - \int_{B_1} \psi\;d\nu(y) \right | \leq \int_{A_\psi} |\psi(x)-\psi(y)|\;d\gamma(x,y)
	  \leq\int_{A_\psi} [\psi]_{\textnormal{Lip}}|x-y|\;d\gamma(x,y).
  \end{align*}
  Since $\textnormal{spt}(\psi)$ is a positive distance away from $\Gamma$, for any admissible $\gamma$ we have
  $\gamma(A_\psi)\leq \mu(\textnormal{spt}(\psi))+\nu(\textnormal{spt}(\psi))<+\infty$. Thus, by H\"older's inequality,
  \begin{align*}
    \left | \int_{B_1} \psi \;d\mu - \int_{B_1}\psi\;d\nu \right |  \leq  [\psi]_{\textnormal{Lip}} \gamma( A_\psi)^{\frac{p-1}{p}}\left ( \int_{A_\psi} |x-y|^p\;d\gamma(x,y) \right )^{\frac{1}{p}}.
  \end{align*}
  Taking infimum over all $\gamma\in\textnormal{Adm}(\mu,\nu)$, we thus obtain
  \begin{align*}
    \left | \int_{B_1} \psi \;d\mu - \int_{B_1}\psi\;d\nu \right | & \leq \left(\mu(\textnormal{spt}(\psi))+\nu(\textnormal{spt}(\psi))\right)^{\frac{p-1}{p}}[\psi]_{\textnormal{Lip}}\textnormal{d}_{\mathbb{L}_p}(\mu,\nu).
  \end{align*}	
  \end{proof}

%%%%%%%%%%%%%%%%%%%%%%%%%%%%%%%%%%%%%%%%%%%%%%

\end{document}